\documentclass[10pt,reqno]{amsart}
\usepackage{amssymb,amsmath,amstext,amsgen,amsfonts,amsthm}
\usepackage{verbatim,enumerate,mathrsfs,url,fullpage,hyperref}
\usepackage{setspace}

\newcommand{\N}{\mathbb{N}}
\newcommand{\R}{\mathbb{R}}

\newcommand{\Z}{\mathbb{Z}}

\newcommand{\h}{\mathfrak{h}}
\newcommand{\g}{\mathfrak{g}}
\newcommand{\D}{\mathcal{D}}
\newcommand{\V}{\mathcal{V}}
\newcommand{\Ndelta}{\mathcal{N}_\delta}

\newcommand{\dcc}{d_{cc}}
\newcommand{\Bcc}{B_{cc}}
\newcommand{\Ncc}{N_{cc}}
\newcommand{\qd}{\operatorname{cd}}
\newcommand{\diam}{\operatorname{diam}}
\newcommand{\Lip}{\operatorname{Lip}}

\renewcommand{\H}{\mathcal{H}}
\renewcommand{\epsilon}{\varepsilon}
\renewcommand{\phi}{\varphi}
\renewcommand{\angle}{\measuredangle}

\newtheorem{theorem}{Theorem}[section]
\newtheorem{lemma}[theorem]{Lemma}
\newtheorem{proposition}[theorem]{Proposition}

\theoremstyle{remark}
\newtheorem{remark}[theorem]{Remark}
\newtheorem{example}[theorem]{Example}

\title{Coarse differentiation and quantitative nonembeddability\\for Carnot groups}

\author{Sean Li}
\address{Courant Institute, New York University, New York NY 10012}
\email{seanli@cims.nyu.edu}

\date{}

\begin{document}

\maketitle
\begin{abstract}
  We give lower bound estimates for the macroscopic scale of coarse differentiability of Lipschitz maps from a Carnot group with the Carnot-Carath\'{e}odory metric $(G,\dcc)$ to a few different classes of metric spaces.  Using this result, we derive lower bound estimates for quantitative nonembeddability of Lipschitz embeddings of $G$ into a metric space $(X,d_X)$ if $X$ is either an Alexandrov space with nonpositive or nonnegative curvature, a superreflexive Banach space, or another Carnot group that does not admit a biLipschitz homomorphic embedding of $G$.  For the same targets, we can further give lower bound estimates for the biLipschitz distortion of every embedding $f : B(n) \to X$, where $B(n)$ is the ball of radius $n$ of a finitely generated nonabelian torsion-free nilpotent group $G$.  We also prove an analogue of Bourgain's discretization theorem for Carnot groups and show that Carnot groups have nontrivial Markov convexity.  These give the first examples of metric spaces that have nontrivial Markov convexity but cannot biLipschitzly embed into Banach spaces of nontrivial Markov convexity.
\end{abstract}

\tableofcontents

\section{Introduction}
Let $G$ be a Carnot group endowed with a Carnot-Carath\'{e}odory metric and unit ball $B_G$ and let $(X,d_X)$ be some other metric space.  Given a prescribed $\epsilon \in (0,1)$, one can ask what is the largest $\rho(\epsilon) > 0$ so that, given any Lipschitz function $f : B_G \to X$, there exists a subball $B(x,r) \subseteq B_G$ of radius $r \geq \rho(\epsilon)$ and a map of canonical form $T : B_G \to X$ (whose form depends on the class of metric space to which $X$ belongs) so that
\begin{align}
  \sup_{z \in B(x,r)} \frac{d_X(f(z),T(z))}{r} \leq \epsilon \|f\|_{lip}. \label{UAAP}
\end{align}

Estimates of the form \eqref{UAAP} originated from the work of \cite{BJLPS} in the setting of normed linear spaces where $T$ is an affine function.  There, the authors named the property of having positive $\rho(\epsilon)$ for any $\epsilon \in (0,1)$ as the Uniform Approximation by Affine Property (or UAAP) and, they showed that Lip($X,Y$), the space of Lipschitz functions from $X$ to $Y$, has the UAAP if and only if one of the spaces $\{X,Y\}$ is finite dimensional and the other is superreflexive.  For the case when $Y$ is superreflexive, estimates of $\rho(\epsilon)$ were given in \cite{LN} where they also used it to prove a restricted case of Bourgain's discretization theorem.  For a similar statement concerning Lipschitz maps of finite dimensional vector spaces to general metric spaces see \cite{AS:12}.  We generalize the results of \cite{LN} to the case when the domain is a Carnot group.  These results belong in a class of methods that can be called quantitative or coarse differentiation.  There is much research being done on this subject and its applications (cf. \cite{CKN,EFW1,EFW2,LP}).

We will provide some quantitative estimates for lower bounds of such $\rho$ when $(X,d)$ is a member of three classes: general metric spaces, superreflexive Banach spaces, and other Carnot groups.  The canonical forms for the metric space classes are maps where horizontal lines are mapped to constant speed geodesics in the first case---this is not quite accurate, but will be made precise in Theorem \ref{UAAP-M}---and group homomorphisms for the last two.  While the actual theorems require some preliminary material to state (which will be done in the following two sections), we can state the relevant consequences.

\begin{theorem} \label{UAAP-M}
  Let $G$ be a Carnot group with unit ball $B_G$, $S^{n-1}$ be the unit sphere of the horizontal layer of its Lie algebra, and $(X,d_X)$ be some metric space.  There exists some $\epsilon_0 > 0$ and $\alpha > 0$ depending only on $G$ so that if $\epsilon \in (0,\epsilon_0)$ and $f : B_G \to X$ is Lipschitz, then there exist some function $w : S^{n-1} \to \R^+$ and some subball $B \subset B_G$ with radius $r \geq e^{-\epsilon^{-\alpha}}$ so that
  \begin{align*}
    \sup \left\{ \left| d_X(f(xe^{tv}),f(xe^{sv})) - (t-s) w(v) \right| : x \in B, v \in S^{n-1}, -3r \leq s < t \leq 3r \right\} \leq \epsilon r \|f\|_{lip}.
  \end{align*}
\end{theorem}

\begin{theorem} \label{UAAP-rest}
  Let $G$ be a Carnot group with unit ball $B_G$ and $(X,d_X)$ be a superreflexive Banach space (resp. Carnot group).  There exists some $\epsilon_0 > 0$ and $\alpha > 0$ depending only on $G$ and $X$ so that if $\epsilon \in (0,\epsilon_0)$ and $f : B_G \to X$ is Lipschitz, then there exist some Lipschitz homomorphism $T : G \to X$ and some subball $B \subset B_G$ with radius $r \geq e^{-\epsilon^{-\alpha}}$ (resp. $r \geq e^{-e^{\epsilon^{-\alpha}}}$) so that
  \begin{align*}
    \sup_{x \in B} \frac{d_X(f(x),T(x))}{r} \leq \epsilon \|f\|_{lip}. 
  \end{align*}
\end{theorem}

To prove such bounds, we establish a uniform convexity condition for each of the target spaces and build on the ``iterated midpoint'' technique of \cite{LN} for the case of a Lipschitz mapping of $\R$ to $X$.  However, we cannot use the multidimensional argument of \cite{LN} as the noncommutativity of Carnot groups destroys the grid structure that had been exploited.  Instead, we will use a theorem of Christ about dyadic-like cubes in doubling metric spaces (to which Carnot groups belong) to adapt the averaging and decomposition result of \cite{schul}.

On the way, we prove a result of independent interest showing that, in some sense, Carnot groups are uniformly convex.  This will subsequently show that Carnot groups have nontrivial Markov convexity, a metric invariant that, when restricted to Banach spaces, is equivalent to being isomorphic to a uniformly convex Banach space.  Thus, Carnot groups are the first examples of metric spaces with nontrivial Markov convexity that do not embed into any Banach space with nontrivial Markov convexity.  This has connections to the larger Ribe program, an active research program in functional analysis and metric geometry.  The relevant background material as well as the Markov convexity proof will be given in section 7.1.

The coarse differentiation method will be the key technical tool that we use to derive quantitative estimates for how nonembeddable Carnot groups are into three classes of metric spaces.  To determine if a space does not embed geometrically well into another space, one can try to show that ``soft'' geometric embeddings imply the existence of ``rigid'' embeddings.  If such rigid embeddings cannot exist by other reasonings, then the soft embeddings cannot exist.  As a famous example, Pansu proved in \cite{Pansu} a generalization of the Rademacher differentiation theorm to show that blowups of Lipschitz maps (the ``soft'' maps) between Carnot groups converge to group homomorphisms (the ``rigid'' maps) almost everywhere.  Semmes observed in \cite{Semmes} that this means that blowups of maps from the Heisenberg group to Euclidean spaces must converge to a group homomorphism.  As there are no biLipschitz group homomorphisms from the nonabelian Heisenberg group to Euclidean space, any Lipschitz mapping cannot then be biLipschitz.  See \cite{CK1,CK2,Kir,Pauls} for more examples of differentiability statements.

One problem left unanswered by this method is finding the quantitative rate at which the map degrades from being biLipschitz.  Indeed, there are no guarantees as to how much one must blowup the map before one can start seeing convergence to the derivative.  In general, such an estimate cannot be done without knowledge of the second derivative, which is much too strict a condition given that we are working with non-smooth maps.  The coarse differentiation result will allow us to control the scale to which some kind of approximate convergence happens.  Note that this is fundamentally different from regular differentiation.  First, we do not have any control over where the canonical behavior occurs.  Regular differentiation, on the other hand, looks at the limiting behavior of the map around a specified point.  Secondly---and more importantly---the approximating map does not have to be related to any derivatives of the map.  Indeed, consider the sawtooth map whose derivative is $f'(x) = 2(\lfloor x \rfloor \mod 2) - 1$.  All derivatives of this map have slope either +1 or -1, but, looking at the map from a sufficiently large scale, we see that the best approximating affine map is the constant 0 function.

Using the estimates, we will prove the following theorem:

\begin{theorem} \label{full-nonembed}
  Let $(G,\dcc)$ be a Carnot group that is endowed with the Carnot-Carath\'{e}odory metric and $(X,d_X)$ be either an Alexandrov space of nonnegative or nonpositive curvature or a superreflexive Banach space.  Then there exist $c,C > 0$ depending only on $G$ and $X$ so that for any 1-Lipschitz function $f : B_G \to X$ there exist $x,y \in G$ with $\dcc(x,y)$ arbitrarily small so that
  \begin{align*}
    \frac{d_X(f(x),f(y))}{\dcc(x,y)} \leq C\left( \log \frac{1}{\dcc(x,y)} \right)^{-c}.
  \end{align*}
  If $(X,d_X)$ is another Carnot group that does not admit a biLipschitz homomorphic embedding of $G$, then the same statement holds except we have the estimate
  \begin{align*}
    \frac{d_X(f(x),f(y))}{\dcc(x,y)} \leq C\left( \log \log \frac{1}{\dcc(x,y)} \right)^{-c}.
  \end{align*}
\end{theorem}

Alexandrov spaces are spaces of curvature bounded above or below in the sense of the Toponogov comparison theorem.  Examples include simply connected manifolds with curvature bounds.  That Carnot groups do not embed biLipschitzly into Alexandrov spaces of nonnegative or nonpositive curvature was already proven in \cite{Pauls}.  However, an infinitesimal differentiation argument was used that gave no clue as to how the embedding must quantitatively break down.  Theorem \ref{full-nonembed} reproves the result of \cite{Pauls} in a quantitative form.  Sharp estimates for quantitative nonembeddability of Lipschitz maps from the Heisenberg group to superreflexive Banach spaces were obtained in \cite{Lafforgue-Naor} using Paley-Littlewood theory.  Theorem \ref{full-nonembed} is a purely geometric argument that generalizes the result to arbitrary Carnot groups, albeit while losing sharpness in the power of decay.  BiLipschitz nonembeddability into other Carnot groups follows from Semmes' argument using Pansu's differentiation theorem.  Our argument reproves this result in a quantitative form.

Our differentiability result will actually be for a range of maps wider than just Lipschitz.  We will show that the coarse differentiation method can hold for maps that are Lipschitz only at large distances.  This will allow us to extend our results to maps like uniform embeddings and even maps that are not necessarily continuous, like quasi-isometric embeddings.  Unfortunately, as the maps will have a scale above which they are Lipschitz, it breaks the scale invariance of the problem, which will complicate things.  It turns out that we will need the ball that is the domain to be large enough relative to the scale of Lipschitz behavior.  This kind of result was implied by the quantitative directions of \cite{BJLPS} and \cite{LN}, but not explicitly stated.

Given two metric spaces $(X,d_X)$ and $(Y,d_Y)$, recall that the biLipschitz distortion of $X$ into $Y$, denoted $c_Y(X)$, is the infimal $D > 1$ so that there exists some $f : X \to Y$ and some $s > 0$ for which
\begin{align*}
  s \cdot d_X(x,y) \leq d_Y(f(x),f(y)) \leq Ds \cdot d_X(x,y), \qquad \forall x,y \in X.
\end{align*}
A recent theorem of \cite{BLD} gives quantitative bounds on the rate of convergence of rescaled balls of finitely generated torsion-free nilpotent groups to their asymptotic cones, which are Carnot groups.  We chain this result together with coarse differentiation to give lower bounds to the biLipschitz distortion of embeddings of balls of such finitely generated groups into the same target spaces.  Specifically, we have
\begin{theorem} \label{fg-nonembed}
  Let $B_S(n)$ denote the ball of radius $n$ in $G$, a finitely generated nonabelian torsion-free nilpotent group with generating set $S$ (which will define the word metric).  Let $(X,d_X)$ be either a superreflexive Banach space or an Alexandrov space of nonnegative or nonpositive curvature.  Then there exists $c,C > 0$ depending only on $G$, $S$, and $X$ so that
  \begin{align*}
    c_X(B_S(n)) \geq C \left( \log n \right)^c.
  \end{align*}
  If $X$ is a Carnot group that does not admit a biLipschitz homomorphic embedding of $G$, then the same statement holds except we have the estimate
  \begin{align*}
    c_X(B_S(n)) \geq C \left( \log \log n \right)^c.
  \end{align*}
\end{theorem}
In the proof, we make critical use of the fact that coarse differentiation can work for maps that are Lipschitz only at large distances.

As we can coarsely differentiate maps that are Lipschitz at large distances, we can also then prove an analogue of Bourgain's discretization theorem in the nonabelian setting of Carnot groups.  Following \cite{GNS}, we define the discretization modulus for any two Banach spaces $X$ and $Y$, denoted $\delta_{X \hookrightarrow Y}(\epsilon)$, to be the supremal $\delta \in (0,1)$ for each $\epsilon \in (0,1)$ such that for every $\delta$-net $\Ndelta$ of $B_X$, $c_Y(\Ndelta) \geq (1-\epsilon) c_Y(X)$.  Bourgain proved in \cite{bourgain-discretization} that, for each pair of Banach spaces $X$ and $Y$ where $\dim X = n < \infty$ and $c_Y(X) < \infty$, the discretization modulus is always positive for any $\epsilon > 0$ and gave a lower bound that depends only on $n$.

It is then straightforward to define the same function for Carnot groups.  Given two Carnot groups, $G$ and $H$, we can then define the discretization modulus $\delta_{G \hookrightarrow H}(\epsilon)$ to be the supremal $\delta \in (0,1)$ such that for every $\delta$-net $\Ndelta$ of $B_G$, the unit ball of $G$, satisfies $c_H(\Ndelta) \geq (1-\epsilon) c_H(G)$.   To do so, we will take a near optimal embedding of the $\delta$-net, extend it to a map that is Lipschitz at large distances on all of $B_G$, and differentiate it to produce a homomorphism with the desired distortion.

Section 2 will review and discuss Carnot groups, Alexandrov spaces, and superreflexive Banach spaces and establish notation.  Section 3 through 6 contains the proof of the coarse differentiation result of Lipschitz maps from Carnot groups to three classes of spaces.  Section 7 is devoted to proving a result on convexity of the Carnot-Carath\'{e}odory metric on Carnot group that is required in Section 6 as well as showing the Markov convexity result and giving the relevant background.  Section 8 will be devoted to proving the quantitative nonembeddability and discretization theorems using the differentiability results.  As the methods used do not seem to give sharp bounds, we will not try to optimize parameters and will, in some places, knowingly use suboptimal bounds (usually by a multiplicative or exponential constant). \\

\noindent {\bf Acknowledgements.}  I am grateful to Jonas Azzam for introducing and explaining \cite{schul} to me and to Enrico Le Donne for teaching me the basics of Carnot groups.  I would also like to thank Assaf Naor for many helpful conversations and for pointing out how the proof of Markov convexity depended on Lemma 2.3 of \cite{MN:08}.  Parts of this work were completed while I was attending the Quantitative Geometry program at MSRI and while invited to Purdue University by Ben McReynolds, for which I am also thankful.


\section{Preliminaries and notations} \label{preliminaries}
In this section, we review some basic definitions and results for Carnot groups, Alexandrov spaces, and superreflexive Banach spaces as well as establish some basic notation.

Given a map between two metric spaces $h : (Y,d_Y) \to (X,d_X)$, we can define for $t \in \R^+$
\begin{align*}
  \Lip_h(t) = \sup_{d_Y(x,y) > t} \frac{d_X(h(x),h(y))}{d_Y(x,y)}.
\end{align*}
Clearly $\Lip_h(t) \leq \Lip_h(s)$ if $s \leq t$.  We will say the map is $\psi$-Lipschitz at large distances ($\psi$-LLD) for some $\psi \geq 0$ if $\Lip_h(\psi) < \infty$ and
\begin{align*}
  \sup_{d_Y(x,y) \leq \psi} d_X(h(x),h(y)) \leq \Lip_h(\psi) \cdot \psi.
\end{align*}
The $\Lip_h(\psi) \cdot \psi$ bound simply states that we can use the macroscopic Lipschitz bound to get an absolute bound for microscopic distances.  The form in merely a convenience; the important thing is that there is some finite bound that controls the behavior of $h$ on small scales.  Examples of maps that are Lipschitz at large distances are Lipschitz maps, uniformly continuous maps, and quasi-isometries.  One can see that $h$ need not even be continuous.  Clearly, if a map $h$ is 0-LLD, then it is Lipschitz and $\Lip_h(0) = \|h\|_{lip}$.  

\subsection{Carnot groups}
All Lie groups will be assumed to be simply connected.  Given a Lie algebra $\g$, we can define the decending central sequence $\{\mathcal{G}_j\}$ as follows
\begin{align*}
  \mathcal{G}_1 = \g, \qquad \mathcal{G}_{j+1} = [\mathcal{G}_j,\g].
\end{align*}
If there exists some $r > 0$ so that $\mathcal{G}_{r+1} = 0$ then we say that $\g$ is nilpotent.  If, in addition, $\mathcal{G}_r \neq 0$ then we say that $\g$ has nilpotency step $r$.  A Lie algebra $\g$ is graded if it can be decomposed as
\begin{align*}
  \g = \bigoplus_{j=1}^r \V_j,
\end{align*}
and the subspaces $\V_j$ satisfy
\begin{align*}
  [\V_i,\V_j] \subset \V_{i+j}.
\end{align*}
A Lie group is graded and nilpotent if its associated Lie algebra is so.  A graded nilpotent Lie algebra is stratified if $\V_1$ generates the entire Lie algebra, \textit{i.e.} for any $k \geq 2$ and $v \in \V_k$, there exists $v_1,...,v_k \in \V_1$ so that
\begin{align*}
  [v_1,[v_2,...[v_{k-1},v_k]...]] = v.
\end{align*}
A Carnot group is then a simply connected Lie group with a stratified nilpotent Lie algebra.  We will call the subspace $\V_1 \subset \g$ the horizontal layer.  The horizontal elements of $G$ are all elements of the form $e^{\lambda v}$ where $e$ is the exponential map, $\lambda \in \R$, and $v \in \V_1$.  We will sometimes use $\V_k(\g)$ if we are in a situation with multiple Lie algebras to avoid confusion.

It is known that the exponential map is a diffeomorphism for simply connected nilpotent Lie groups \cite{FS}.  We will use this diffeomorphism to canonically identify elements of the Lie group with elements of the Lie algebra.  Thus we get that a graded nilpotent Lie group $G$ is topologically a Euclidean space.  One can then push forward the coordinate system of the Lie algebra to the Lie group and so we can write elements $g \in G$ as $(g_1,g_2,...,g_r)$ where each $g_i$ is also a vector of dimension $\dim \V_i$.  These are called the exponential coordinates of $G$.  We will use 0 to denote the identity element.  Letting $|\cdot|$ denote the Euclidean norm, it then makes sense to talk about $|g_r|$, $|g_1 - h_1|$, and so forth.  The Baker-Campbell-Hausdorff theorem (BCH) describes how group multiplication in a Lie group is represented on the Lie algebra level.

\begin{theorem}[Baker-Campbell-Hausdorff formula \cite{Dynkin,varad}]
  Let $G$ be a simply connected Lie group with Lie algebra $\g$.  Then given $U,V \in \g$ and the equality
  \begin{align*}
    e^W = e^U e^V,
  \end{align*}
  we can write the formula for $W$ as
  \begin{align*}
    W = \sum_{k > 0} \frac{(-1)^{k-1}}{k} \sum_{\begin{smallmatrix} {r_i + s_i > 0,} \\ {1 \leq i \leq n} \end{smallmatrix}} \frac{ \left( \sum_{j=1}^k (r_i + s_i) \right)^{-1}}{r_1!s_1! \cdots r_k!s_k!} (ad U)^{r_1} (ad V)^{s_1} \cdots (ad U)^{r_k} (ad V)^{s_k-1} V.
  \end{align*}
  where $(ad X)Y = [X,Y]$.
\end{theorem}

Note that as we are working with nilpotent Lie algebras, the summation will be finite.  Because we are pushing foward the coordinates of $\g$ to $G$, we can use the BCH formula on the level of the coordinates of the Lie group.  Given two elements $(g_1,...,g_r)$ and $(h_1,...,h_r) \in G$, the BCH formula shows that
\begin{align*}
  (g_1,...,g_r) \cdot (h_1,...,h_r) = (g_1 + h_1,g_2 + h_2 + P_2,...,h_r+g_r + P_r)
\end{align*}
where $P_k$ is a polynomial of the coordinates $g_1,...,g_{k-1},h_1,...,h_{k-1}$.  We will call $P_k$ the BCH polynomials.

\begin{example}
  The Heisenberg algebra is a 2-step nilpotent Lie algebra spanned by three vectors $\{X,Y,Z\}$ with the Lie bracket relations $[X,Y] = Z$, $[X,Z] = 0$, $[Y,Z] = 0$.  We have that the first few terms of the BCH formula are
  \begin{align*}
    e^U e^V = e^{U + V + \frac{1}{2} [U,V] + \frac{1}{12} [U,[U,V]] - \frac{1}{12} [V,[U,V]] + ...}.
  \end{align*}
  We then get that
  \begin{align}
    e^{aX + bY + cZ} &e^{a'X + b'Y + c'Z} \notag \\
    &= e^{(a+a')X + (b+b')Y + (c+c')Z + \frac{1}{2} [aX + bY + cZ,a'X + b'Y + c'Z]} \notag \\
    &= e^{(a+a')X + (b+b')Y + (c+c')Z + \frac{1}{2} \left( (ab'-a'b)[X,Y] + (ac'-a'c) [X,Z] + (bc' - b'c)[Y,Z]\right)} \label{lie-bracket-use} \\
    &= e^{(a+a')X + (b+b')Y + \left(c+c'+\frac{1}{2}(ab'-a'b)\right) Z} \label{lie-relation-use}.
  \end{align}
  In \eqref{lie-bracket-use}, we used bilinearity and antisymmetry of the Lie bracket and in \eqref{lie-relation-use}, we used the Lie bracket relations of $X$, $Y$, and $Z$.  We stopped with the first Lie bracket because any further nesting of Lie brackets becomes trivial as the Lie algebra is nilpotent of step 2.
  
  If we use the exponential coordinates to identify $(a,b,c) \in \mathbb{R}^3$ with $e^{aX+bY+cZ}$, we recover the usual Heisenberg product:
  \begin{align*}
    (a,b,c) \cdot (a',b',c') \overset{\eqref{lie-relation-use}}{=} \left(a+a', b+b', c+c'+ \frac{1}{2}(ab'-a'b) \right).
  \end{align*}
  One then sees that $P_2((a,b,c),(a',b',c')) = \frac{1}{2}(ab'-a'b)$. \\
  \qed
\end{example}

Another important property of graded nilpotent Lie groups is that they admit a family of self-similarities.  Let $g \in G$ be an element of the Lie group.  Then $g = e^{g_1 + ... + g_r}$ where $g_i \in \V_i$.  Given $\lambda \geq 0$ we can define the dilation automorphism
\begin{align*}
  \delta_\lambda : G &\to G \\
  e^{g_1 + ... + g_r} &\mapsto e^{\lambda g_1 + \cdots + \lambda^r g_r}.
\end{align*}

A homogeneous norm on a graded nilpotent Lie group $G$ is a continuous nonnegative function $\rho : G \to \R^+$ such that
\begin{align*}
  \rho(g) &= \rho(g^{-1}), \\
  \rho(\delta_\lambda(g)) &= \lambda \rho(g), \\
  \rho(g) &= 0 \Leftrightarrow g = 0.
\end{align*}
A homogeneous norm defines a homogeneous (semi)metric on $G$ by $d(g,h) = \rho(g^{-1}h)$.  Any two homogeneous norms $N,N' : G \to \R^+$ are equivalent, \textit{i.e.} there exists some $C > 0$ (depending only on the two metrics) so that for all $g, h \in G$ we have
\begin{align*}
  \frac{N'(g)}{C} \leq N(g) \leq CN'(g).
\end{align*}
As passing to metrics equivalent to the Carnot-Carath\'{e}odory metric changes the bounds in coarse differentiability and quantitative nonembeddability only by a multiplicative constant, we can and will use different homogeneous metrics in portions of our subsequent analysis.  We will also define another group norm as
\begin{align*}
  N_\infty : G &\to \R^+ \\
  (g_1,...,g_r) &\mapsto \max_i \lambda_i |g_i|^{1/i},
\end{align*}
with the associated homogeneous metric $d_\infty$.  Here, $\lambda_i > 0$ are positive real scalars and $|\cdot|$ is the Euclidean norm.  It is known that for each graded nilpotent Lie group, there exists a configuration of $\{\lambda_i\}$ that makes $d_\infty$ into a true metric with no multiplicative factor in the triangle inequality \cite{Guivarc'h,breuillard}.  We may suppose for simplicity that $\lambda_i = 1$ for all $i$.  Everything that follows goes through in the general case with superficial modifications.

Let $S^{n-1}$ denote the unit sphere of $\V_1$.  Given $x \in G$ and a horizontal unit vector $v \in S^{n-1}$, we can isometrically embed $\R$ into $G$ via {\it horizontal lines} by the map
\begin{align*}
  t \mapsto x e^{tv} =: x \cdot tv, \qquad \forall t \in \R.
\end{align*}
Using these isometries, we can pushforward notions such as dyadic subdivision, midpoints, length, etc. from $\R$ to horizontal lines of $G$.  We also let $G \circleddash v$ denote the exponential image of the subspace of $\g$ orthogonal to $v$.

All graded nilpotent Lie groups admit a path metric on a class of restricted paths that we describe as follows.  We can construct a left invariant subbundle of the tangent bundle by taking, at each point $g \in G$, the fiber to be the left translate of the subspace $\V_1$.  We will denote this subbundle $\H$.  If we place a left invariant field of inner products $\langle\cdot,\cdot\rangle_g$ on $\H$, we can define the Carnot-Carath\'{e}odory metric (CC-metric) between two points $g,h \in G$ to be
\begin{align*}
  \dcc(g,h) = \inf \left\{ \int_0^1 \langle \gamma'(t), \gamma'(t) \rangle^{1/2}_{\gamma(t)} ~dt : \gamma(0) = g, \gamma(1) = h, \gamma'(t) \in \H_{\gamma(t)} \right\}
\end{align*}

If no such path exists, we set $\dcc(g,h) = \infty$.  The set of paths $\gamma : [a,b] \to G$ where $\gamma'(t) \in \H_{\gamma(t)}$ are called horizontal paths.  It is clear that this is a left invariant metric as all the fibers are defined in a left invariant manner.  Because we are taking the Riemannian length of a class of restricted paths, this is also called a sub-Riemannian metric.  It is natural then to ask if there always exists a horizontal path between two points of $G$.  Chow's theorem answers the question in the affirmative when $\V_1$ generates the entire Lie algebra, \textit{i.e.} when $G$ is a Carnot group (see \cite{gro:96,montgomery}).

Instead of taking a scalar product in the definition of the CC-metric, we could have taken a left-invariant field of vector norms instead and defined a sub-Finsler metric in a similar fashion.  We will show below that the CC-metric can be defined in terms of a homogeneous norm.  As the same reasoning works for sub-Finsler metrics, we get that these two metrics are equivalent, and so we will not make an effort to differentiate between the two of them.

Let $\{v_1,...,v_n\}$ be an orthonormal basis for $\mathcal{V}_1$, which we will suppose generates all of $\g$.  By the proof of Chow's theorem, there exists $M_G > 0$ depending only on $G$ so that, for any $r > 0$ and any element $g$ in the CC-unit ball $\delta_r(B_G)$, there exist $j \leq M_G$ and $i : \{1,...,j\} \to \{1,...,n\}$ so that $g$ can be written of the form
\begin{align*}
  g = e^{\lambda_1 v_{i(1)}} e^{\lambda_2 v_{i(2)}} \cdots e^{\lambda_j v_{i(j)}}.
\end{align*}
One can verify that in the case $G$ is a vector space or the $2n+1$ dimensional Heisenberg group, then $M_G$ is $n$ or $2n+4$, respectively.  We also have the bounds $|\lambda_\ell| \leq \lambda r$ for some $\lambda$ depending only on $G$.  The triangle inequality then gives us that
\begin{align*}
  e^{\lambda_1 v_{i(1)}} e^{\lambda_2 v_{i(2)}} \cdots e^{\lambda_\ell v_{i(\ell)}} \in \delta_{M_G\lambda r}B_G, \qquad \forall \ell \leq j.
\end{align*}
For simplicity, we will suppose that $\lambda \leq 1$.  Everything that follows goes through in the general case with superficial modifications.

One can easily verify, by looking at the action of $\delta_\lambda$ and the fact that the CC-metric depends only on the first layer $\V_1$, that the CC-metric is homogeneous with respect to the dilations:
\begin{align*}
  \dcc(\delta_\lambda(g),\delta_\lambda(h)) = \lambda \dcc(g,h).
\end{align*}
Thus, one easily verifies that the function
\begin{align*}
  \Ncc : G &\to \R^+ \\
  g &\mapsto \dcc(0,g)
\end{align*}
is a homogeneous norm on $G$.  Everything with the subscript $cc$ will be in terms of the CC-metric.

Note that the projection to the horizontal coordinate
\begin{align*}
  \pi : G &\to \R^n \\
  g &\mapsto g_1
\end{align*}
is a homomorphism.  This follows from the group product defined by the BCH formula.  Equipping $G$ with a CC-metric and viewing $\R^n$ as a Euclidean space, we get that $\pi$ is a 1-Lipschitz homomorphism that is distance preserving on horizontal lines.

As we've identified graded nilpotent Lie groups with Euclidean spaces, we can make sense of the Lebesgue measure $\mathcal{L}$ of subsets of graded nilpotent Lie groups.  It is known that the Lebesgue measure is a left invariant measure so that if $E \subset G$ is a measurable set, then
\begin{align}
  |\delta_\lambda(E)| = \lambda^N |E|, \label{carnot-set-dilation}
\end{align}
where $N = \sum_{k=1}^r k \dim \V_k$ is the homogeneous dimension of $G$.  This is verifiable by looking at the Jacobian of $\delta_\lambda$.  We can define the metric balls
\begin{align*}
  B(x,r) := \{g \in G : d(x,g) \leq r\}.
\end{align*}
Given $\lambda > 0$ and some ball $B$ with center $x$, we let $\lambda B := x \cdot \delta_\lambda (x^{-1}B)$.  One can then see that $B(x,r) = x \delta_r(B(0,1))$.  Then it is clear that \eqref{carnot-set-dilation} immediately implies that graded nilpotent Lie groups are doubling, \textit{i.e.} there exists some constant $C > 0$ so that for every $x \in G$ and $r > 0$ we have
\begin{align*}
  |B(x,2r)| \leq C|B(x,r)|.
\end{align*}

\subsection{Alexandrov spaces}
Alexandrov spaces are generalizations of Riemannian spaces with curvature bounds.  They are divided into two types, $CAT(k)$ spaces, which are spaces with curvature bounded from above, and $CBB(k)$ spaces, which are spaces with curvature bounded below.  Given $k \in \R$, we let $(M_k^2,d_k)$ be the two dimensional constant curvature $k$ model space (\textit{i.e.} $S_k^2$, $\R^2$, or $H_k^2$).  Let $\diam(k)$ be the diameter of the model space $M_k^2$ with the understanding that it is infinite if $k \leq 0$.  Given a triangle $\triangle abc$ in $X$ with vertices $a,b,c \in X$ and minimizing geodesic as sides, we can construct a comparison triangle $\widetilde{\triangle} abc$ with the same sidelengths in the model space $M_k^2$. 

A complete metric space $(X,d_X)$ is a $CAT(k)$ space if
\begin{itemize}
  \item Every pair $x,y \in X$ with $d_X(x,y) \leq \diam(k)$ is joined by a geodesic segment.
  \item Let $\triangle abc$ be a geodesic triangle in $X$ with $d_X(a,b) + d_X(b,c) + d_X(c,a) < 2 \diam(k)$.  For every two points $x,y \in \triangle$, if we let $x_k,y_k$ be the corresponding points in $\widetilde{\triangle} abc$, then $d_X(x,y) \leq d(x_k,y_k)$.
\end{itemize}

A complete metric space $(X,d_X)$ is a $CBB(k)$ space if
\begin{itemize}
  \item $(X,d_X)$ is a locally compact geodesic space.
  \item Let $\triangle abc$ be a geodesic triangle in $X$ with $d_X(a,b) + d_X(b,c) + d_X(c,a) < 2 \diam(k)$.  For every two points $x,y \in \triangle$, if we let $x_k,y_k$ be the corresponding points in $\widetilde{\triangle} abc$, then $d_X(x,y) \geq d(x_k,y_k)$.
\end{itemize}

The inequalities relating the distances in the geodesic triangles with the comparison triangles are called the triangle comparison properties.  One can visualize these conditions in the following way: geodesic triangles in $CAT(k)$ spaces (resp. $CBB(k)$ spaces) are skinnier (resp. fatter) than their comparison triangles in $M_k^2$.  In the literature, $CAT(0)$ spaces are also called Hadamard spaces.  Good reference for these spaces are \cite{Ballman,BBI,BGP}.  In this paper, we will only focus on $CAT(0)$ and $CBB(0)$ spaces and so the model space will be the Euclidean 2-plane $\R^2$.

One can also define a notion of angles between two geodesic segments in $CAT(0)$ and $CBB(0)$ spaces.  Let $p \in X$ and $\gamma_0,\gamma_1 : [0,1] \to X$ be minimizing geodesic segments where $\gamma_0(0) = \gamma_1(0) = p$.  Then for $s,t \in [0,1]$, one can construct the comparison triangle $\widetilde{\triangle} \gamma_0(s)p\gamma_1(t)$ in $\R^2$.  This triangle is clearly unique up to rigid motion and so we can define the comparison angle $\widetilde{\angle} \gamma_0(s) p \gamma_1(t)$ using the law of cosines
\begin{align*}
  \widetilde{\angle} \gamma_0(s) p \gamma_1(t) = \cos^{-1} \left( \frac{d_X(\gamma_0(s),p)^2 + d_X(p,\gamma_1(s))^2 - d_X(\gamma_0(s),\gamma_0(t))^2}{2 d_X(\gamma_0(s),p) d_X(p,\gamma_1(t))} \right).
\end{align*}
One can then define the angle between the two geodesics as
\begin{align*}
  \angle \gamma_0(1) p \gamma_1(1) = \lim_{s,t \to 0} \widetilde{\angle} \gamma_0(s) p \gamma_1(t).
\end{align*}

It is natural to ask if the limit on the right hand side actually converges.  Define the function $\theta(s,t) = \widetilde{\angle} \gamma_0(s) p \gamma_1(t)$.  Then by the triangle comparison property of $CAT(0)$ spaces (resp. $CBB(0)$ spaces) and the law of cosines, we see that for $s' \leq s$ and $t' \leq t$ that $\theta(s',t') \leq \theta(s,t)$ (resp. $\theta(s',t') \geq \theta(s,t)$).  This is called the monotonicity property.  Thus, the limit does exist and we get $\widetilde{\angle} \gamma_0(1)p\gamma_1(1) \geq \angle \gamma_0(1)p\gamma_1(1)$ (resp. $\widetilde{\angle} \gamma_0(1)p\gamma(1) \leq \angle \gamma_0(1)p\gamma_1(1)$).  This is called the angle property.  In fact, all these properties (triangle comparison, monotonicity, angle) provide equivalent definitions of $CAT(0)$ spaces (resp. $CBB(0)$ spaces).

\subsection{Superreflexive Banach spaces}
Recall that a Banach space $X$ is said to be finitely representable in $Y$ if there exists some $K > 1$ so that, for every finite dimensional subspace $Z \subset X$, there exists a subspace in $Y$ so that the Banach-Mazur distance between $X$ and $Y$ satisifies $d(X,Y) \leq K$.  A Banach space $Y$ is said to be superreflexive if every space that is finitely representable in it is reflexive.  Due to the deep works of James \cite{james64,james72}, Enflo \cite{enflo}, and Pisier \cite{pisier}, we know that a Banach space $X$ is superreflexive if and only if it admits an equivalent norm $\|\cdot\|$ so that there are $p > 1$ and $K > 0$ for which
\begin{align}
  \left\| \frac{x+y}{2} \right\|^p + \left\| \frac{x-y}{2K} \right\|^p \leq \frac{\|x\|^p + \|y\|^p}{2}, \qquad \forall x,y \in X. \label{p-convexity}
\end{align}
That is, $X$ is isomorphic to a Banach space with a uniform convexity modulus of power type $p$ (also known as $p$-convexity).  In fact, all uniformly convex Banach space can be renormed to be uniformly convex of power type and so uniformly convex spaces and superreflexive spaces are the same subclass under the isomorphic category.  We will use the definition of $p$-convex Banach spaces from now on.  As an example, we have that for $p \in (1,\infty)$, the usual norm on $L_p(\mu)$ space satisfies \eqref{p-convexity} with $p = \min\{p,2\}$ and $K = \max\{1 / \sqrt{p-1},1\}$ \cite{ball,figiel}.


\section{A Carleson packing condition for coarse differentiation}
In this section, we consider the metric measure space $(G,\dcc,\mathcal{L})$ where $G$ is a Carnot group of homogeneous dimension $N$ with the Carnot-Carath\'{e}odory metric $\dcc$ and Lebesgue measure $\mathcal{L}$.  Denote its Lie algebra by $\g$ which is stratified by the layers $\{\V_i\}_{i=1}^r$ where $\V_1$ is the $n$-dimensional horizontal layer with orthonormal basis $\{v_1,...,v_n\}$.  As $(G,\dcc,\mathcal{L})$ is doubling, a theorem of Christ says that there exists a collection of partitions of $G$ that behave akin to dyadic cubes.
\begin{theorem}[Christ cubes \cite{Christ}] \label{christ-cubes}
  There exists a collection of subsets $\Delta := \{Q_\omega^k \subset G : k \in \Z, \omega \in I_k\}$, and constant $a_1,a_2 > 0$ and $\tau \in (0,1)$ such that
  \begin{enumerate}[$(a)$]
    \item $\left|G \backslash \bigcup_\omega Q_\omega^k\right| = 0 \qquad \forall k$. \label{christ-cover}
    \item If $j \geq k$ then either $Q_\alpha^j \subset Q_\omega^k$ or $Q_\alpha^j \cap Q_\omega^k = \emptyset$. \label{christ-partition}
    \item For each $(j,\alpha)$ and each $k < j$ there exists a unique $\omega$ such that $Q_\alpha^j \subset Q_\omega^k$. \label{christ-nest}
    \item $\mathrm{diam}(Q_\omega^k) \leq a_1\tau^k$. \label{christ-diam}
    \item Each $Q_\omega^k$ contains some ball $\Bcc(z_\omega,a_0\tau^k)$. \label{christ-ball}
  \end{enumerate}
\end{theorem}
We let
\begin{align*}
  \ell : \Delta &\to \R \\
  Q_\omega^k &\mapsto \tau^k
\end{align*}
denote the scale of each cube in $\Delta$ and $\Delta_k := \{Q_\omega^k : \omega \in I_k\}$ for $k \in \Z$.  If $S \in \Delta_j$, then $\Delta_k(S) = \{Q \in \Delta_{j+k} : Q \subseteq S\}$.  For $Q \in \Delta$, we let $B_Q$ denote the ball with center $z_Q$ contained in $Q$ of size $a_0\tau^k$ as guaranteed by property \eqref{christ-ball}.  For an interval $I \subset \R$, we let $\D(I)$ denote the set of dyadic subintervals of $I$ and $\D^k(I)$ the dyadic subintervals of length $2^{-k}|I|$.  We will use the same notation of dyadic subintervals for horizontal line segments of $G$.

By properties \eqref{christ-diam} and \eqref{christ-ball} we see that there exists a constant $C > 0$ depending only on $G$ such that
\begin{align*}
  \frac{1}{C} \ell(Q)^N \leq |Q| \leq C\ell(Q)^N.
\end{align*}
Here, $N$ is the homogeneous dimension of $G$ as defined by the formula
\begin{align*}
  N = \sum_{k=1}^r k \dim \V_k.
\end{align*}

What follows will be largely inspired from \cite{schul}.  Let $f$ be a Lipschitz map from an interval $[a,b]$ to a metric space $(X,d_X)$.  We fix some $p \geq 1$ and define for $x,y \in [a,b]$
\begin{align}
  \partial_f^{(p)}(x,y) &= \frac{1}{2} \left[ \left( \frac{d_X(f(x), f((x+y)/2))}{|y-x|/2} \right)^p + \left( \frac{d_X(f((x+y)/2),f(y))}{|y-x|/2} \right)^p \right] - \left( \frac{d_X(f(x),f(y))}{|y-x|} \right)^p. \label{partial-defn}
\end{align}
By a simple telescoping sum argument, we have
\begin{align}
  \sum_{k=0}^m \sum_{I \in \D^k([a,b])} |I| \partial_f^{(p)}(a(I),b(I)) \leq 2 (b-a) \Lip_h(2^{-m-1}(b-a))^p. \label{telescope}
\end{align}
Given $\epsilon \in [0,1)$, we then define the quantity
\begin{align*}
  \alpha^{(p)}_f([a,b];\epsilon) = (1-\epsilon)^2 (b-a)^{-2} \underset{\begin{smallmatrix} {a \leq x < y \leq b,} \\ {y - x > \epsilon (b-a)} \end{smallmatrix}}{\iint} \partial_f^{(p)}(x,y) ~dx ~dy.
\end{align*}

\begin{lemma} \label{carleson-1d}
  Let $\epsilon \in (0,1)$, $m \in \N$.  Then
  \begin{align*}
    \sum_{k=0}^m \sum_{I \in \D_k([a,b])} \alpha^{(p)}_f(I;\epsilon) |I| \leq 4(b-a) \Lip_f(\epsilon 2^{-m-1}(b-a))^p.
  \end{align*}
\end{lemma}

\begin{proof}
  We have that
  \begin{align*}
    \sum_{k=0}^m \sum_{I \in \D^k([a,b])} \alpha^{(p)}_f(I;\epsilon)|I| &\leq \sum_{k=0}^m \sum_{I \in \D^k([a,b])} |I|^{-1} \underset{\begin{smallmatrix} {a(I) \leq x < y \leq b(I),} \\ {y-x > \epsilon|I|} \end{smallmatrix}}{\iint} \partial_f^{(p)}(x,y) ~dx ~dy \\
    &= \sum_{k=0}^m \sum_{I \in \D^k([a,b])} |I|^{-1} \int_\epsilon^1 \int_{a(I)(1-r)}^{b(I)(1-r)} \partial_f^{(p)}(v+ra(I),v+rb(I)) ~dv ~|I|dr \\
    &\leq \sum_{k=0}^m \sum_{I \in \D^k([a,b])} 2^{-k} \int_\epsilon^1 \int_{a-rb}^{b-ra} \partial_f^{(p)}(v+ra(I), v+rb(I)) ~dv ~dr = (*).
  \end{align*}
  Here, we've extended the range of $v$ and so, taking the summation into account, overcounted by $2^k$ at each level.  Continuing, we get
  \begin{align*}
    (*) &= \int_\epsilon^1 \int_{a-rb}^{b-ra} \sum_{k=0}^m \sum_{I \in \D^k([v+ra,v+rb])} 2^{-k} \partial_f^{(p)}(a(I),b(I)) ~dv ~dr \\
    &\overset{\eqref{telescope}}{\leq} \int_\epsilon^1 \int_{a-rb}^{b-ra} 2 \Lip_f(2^{-m-1} r (b-a))^p ~dv ~dr \\
    &\leq \int_\epsilon^1 \int_{a-rb}^{b-ra} 2 \Lip_f(2^{-m-1} \epsilon(b-a))^p ~dv ~dr  \\
    &\leq 4(b-a) \Lip_f(2^{-m-1} \epsilon(b-a))^p,
  \end{align*}
\end{proof}

\begin{lemma} \label{stay-in-ball}
  Let $t > 0$.  Suppose $x,y \in \Bcc(0,t)$ such that $z = x^{-1}y$ is a horizontal element.  Then we have that $x \cdot \delta_\lambda(z) \in \Bcc(0,3t)$ for all $\lambda \in (0,1)$.
\end{lemma}

\begin{proof}
  As $x,y \in \Bcc(0,t)$, we have that $\dcc(0,z) \leq \dcc(0,x^{-1}) + \dcc(0,y) = 2t$.  Thus, by the triangle inequality and homogeneity of the CC-metric, we have
  \begin{align*}
    \dcc(0,x \cdot \delta_\lambda (z)) \leq \dcc(0,x) + \lambda \dcc(0,z) \leq 3t.
  \end{align*}
\end{proof}

We now extend the definition of $\alpha$ to Christ cubes.  Given $\epsilon \in [0,1)$, define for a cube $Q \in \Delta$ the quantity
\begin{align}
  \alpha_f^{(p)}(Q;\epsilon) := \ell(Q)^{1-N} \int_{S^{n-1}} \int_{z_Q(G \circleddash v)} \chi_{\{x \cdot \R v \cap 2B_Q \neq \emptyset\}} \alpha_f^{(p)}(x \cdot \R v \cap 6B_Q;\epsilon) ~dx ~d\mu(g).
\end{align}
Here, integration in $x$ is with respect to the $N-1$-dimensional Hausdorff measure $\mathscr{H}^{N-1}$ and $g$ is with respect to the uniform measure on $S^{n-1}$.  We have from \cite{FSS} that the Hausdorff measure is equivalent to all natural notions of measures for the hypersurface $G \circleddash v$, like perimeter measure and spherical Hausdorff measure.  As the Hausdorff measure is left invariant, the translation by $z_Q$ in the domain of integration is actually unnecessary, but it's helpful in keeping things straight.  We can normalize the Hausdorff measure so that $\mathscr{H}^{N-1}(\Bcc(0,1) \cap (G \circleddash v)) = 1$.  We then get by simple homogeneity arguments that $\mathscr{H}^{N-1}(\Bcc(0,\lambda) \cap (G \circleddash v)) = \lambda^{N-1}$.

One point of worry is that there are no guarantees $x \cdot \R v \cap 6B_Q$ is connected.  Thus, we specify that it is the connected subset $I$ containing the subset $x \cdot \R v \cap 2B_Q$, which we know is unique by Lemma \ref{stay-in-ball}.  We first prove that $\alpha^{(p)}_f(Q)$ only evaluates the integrals of horizontal lines that have a significant intersection with $Q$.

\begin{lemma} \label{length-comp}
  Let $\beta \geq 3 \alpha > 0$ and $v \in S^{n-1}$.  Then there exists some constant $C > 0$ depending only on $\alpha$, $\beta$, and $G$ such that if $x \cdot \R v \cap \alpha B_Q \neq \emptyset$, then
  \begin{align*}
    \frac{1}{C} \ell(Q) \leq |x \cdot \R v \cap \beta B_Q| \leq C \ell(Q).
  \end{align*}
\end{lemma}

\begin{proof}
  The lower bound is easy as $x \cdot \R v \cap \beta B_Q$ must go from $\alpha B_Q$ to outside $\beta B_Q$.  The result follows as $\ell(Q)$ is comparable to the radius of $B_Q$ by properties \eqref{christ-diam} and \eqref{christ-ball}.  The upper bound is also easy as the distance between the length of the interval is equal to the distance between endpoints.  As both endpoints are in $\beta B_Q$, their distance apart must be less than the diameter of $\beta B_Q$, which we already know is comparable to $\ell(Q)$.
\end{proof}

We can now prove the main result of this section.

\begin{proposition} \label{carleson-cubes}
  There exist constants $C > 0$ and $\lambda > 0$ depending only on the structure of $G$ such that for any $\epsilon \in [0,1)$, $m \in \N$, and $S \in \Delta$ we have
  \begin{align*}
    \sum_{k=0}^m \sum_{Q \in \Delta_k(S)} \alpha^{(p)}_f(Q;\epsilon) |Q| \leq C |S| \Lip_f(\lambda \epsilon \tau^m \ell(S))^p.
  \end{align*}
\end{proposition}

\begin{proof}
  As there exists a constant $C_0 > 0$ depending only on $G$ such that $|Q| \leq C_0\ell(Q)^N$, it suffices to prove that
  \begin{align*}
    \sum_{k=0}^m \sum_{Q \in \Delta_k, Q \subseteq S} \alpha^{(p)}_f(Q;\epsilon) \ell(Q)^N \leq C |S| \Lip_f(\lambda \epsilon \tau^m \ell(S))^p.
  \end{align*}
  Write $\D_0([a,b])$ to be the standard dyadic decomposition of $[a,b]$ and $\D_\pm([a,b]) = \D_0([a,b]) \pm \frac{1}{3}(b-a)$, where all the intervals of $\D_0([a,b])$ are shifted either to the left or right by $\frac{1}{3}(b-a)$.  We have
  \begin{align*}
    \sum_{k=0}^m &\sum_{Q \in \Delta_k(S)} \alpha^{(p)}_f(Q;\epsilon) \ell(Q)^N \\
    &= \sum_{k=0}^m \sum_{Q \in \Delta_k(S)} \int_{S^{n-1}} \int_{z_Q(G \circleddash v)} \chi_{\{x \cdot \R v \cap 2B_Q \neq \emptyset\}} \alpha^{(p)}_f(x \cdot \R v \cap 6B_Q;\epsilon) \ell(Q) ~dx ~d\mu(v) \\
    &= \int_{S^{n-1}} \int_{z_S(G \circleddash v)} \chi_{\{x \in 12\frac{a_1}{a_0} B_S\}} \sum_{k=0}^m \sum_{Q \in \Delta_k(S)} \chi_{\{x \cdot \R v \cap 2B_Q \neq \emptyset\}} \alpha^{(p)}_f(x \cdot \R v \cap 6B_Q;\epsilon) \ell(Q) ~dx ~d\mu(v) \\
    &= (*).
  \end{align*}
  In the second equality above, we used the fact that, if $x \in z_S(G \circleddash v)$ such that $x \cdot \R v$ intersects $2B_Q$ for some $Q \in \Delta(S)$, then $x \in \frac{12a_1}{a_0} B_S$.  Indeed, let $\eta > 0$ so that $xe^{\eta v} \in x \cdot \R v \cap 2B_Q$.  Remembering that the projection homomorphism $\pi : G \to \R^n$ is 1-Lipschitz and distance-preserving on pairs of points that lie on a horizontal line, we get that
  \begin{align*}
    \dcc(x,xe^{\eta v}) &= |\pi(x) - \pi(xe^{\eta v})| = d_{\R^n}(\pi(z_S(G \circ v)),\pi(xe^{\eta v})) \\
    &\leq |\pi(z_S) - \pi(xe^{\eta v})| \leq \diam_{\R^n}(\pi(S \cup 2B_Q)) \\
    &\leq \diam(S \cup 2B_Q) \leq \diam(S) + 2 \diam(Q).
  \end{align*}
  The last inequality comes from the fact that $B_Q \subseteq S$ as $Q \in \Delta(S)$.  This gives us
  \begin{align*}
    \dcc(z_S,x) &\leq \dcc(z_S,z_Q) + \dcc(z_Q,xe^{\eta v}) + \dcc(xe^{\eta v},x) \\
    &\leq \diam(S) + 2\diam(Q) + \diam(S) + 2 \diam(Q) \\
    &\leq 6\diam(S).
  \end{align*}
  As $\diam(S) \leq \frac{a_1}{a_0} \diam(B_S)$, we then have that
  \begin{align*}
    \dcc(z_S,x) \leq \frac{6a_1}{a_0} \diam(B_S),
  \end{align*}
  which proves the claim.

  Thus, given that $x \cdot \R v \cap 2B_Q \neq \emptyset$ implies that $x \cdot \R v \cap 12a_0^{-1}a_1B_S \neq \emptyset$, we then also have that $I := x \cdot \R v \cap 6B_Q \subset x \cdot \R v \cap 36a_0^{-1}a_1 B_S =: J$.  By the $\frac{1}{3}$-trick (see \cite{Okikiolu}), there exists a universal constant $\gamma \geq 1$ such that $I$ is contained in a subinterval $I' \in \D_0(J) \cup \D_+(J) \cup \D_-(J)$ and $|I'| \leq \gamma|I|$.  It follows that
  \begin{align*}
    \alpha^{(p)}_f(I;\epsilon) &= (1-\epsilon)^2 |I|^{-2} \underset{\begin{smallmatrix} {a(I) \leq x < y \leq b(I),} \\ {y-x > \epsilon|I|} \end{smallmatrix}}{\iint} \partial_f^{(p)}(x,y) ~dx ~dy \\
    &\leq \left(1 - \frac{\epsilon}{\gamma} \right)^2 |I|^{-2} \underset{\begin{smallmatrix} {a(I') \leq x < y \leq b(I'),} \\ {y-x > \epsilon|I'|/\gamma} \end{smallmatrix}}{\iint} \partial_f^{(p)}(x,y) ~dx ~dy \\
    &\leq \gamma^2 \alpha^{(p)}_f\left(I'; \frac{\epsilon}{\gamma} \right).
  \end{align*}

  We have from Lemma \ref{length-comp} that $|J| = |x \cdot \R v \cap 36a_0^{-1}a_1 B_S|$ and $|I| = |x \cdot \R v \cap 6B_Q|$ are comparable to $\ell(S)$ and $\ell(Q) = \tau^k \ell(S)$ for some $k \in \{0,...,m\}$, respectively.  Thus, as $|I| \leq |I'| \leq \gamma |I|$, we have that there is some constant $C_1 \in \N$ so that so that $I' \in \D_0^{k \log 1/\tau + \ell}(J) \cup \D_-^{k \log 1/\tau + \ell}(J) \cup \D_+^{k \log 1/\tau + \ell}(J)$ for some $\ell \leq C_1$.

  For $I \subset J$, we let $I^*$ denote an associated dyadic subinterval from the $\frac{1}{3}$-trick (choosing one arbitrarily if there are multiple associated intervals).  We claim that there is a constant $C_2 > 0$ depending only on $G$ such that, given any $x \in G$, $v \in S^{n-1}$, $k \in \Z$ and $I' \in \D_0^k(J) \cup \D_+^k(J) \cup \D_-^k(J)$, we have
  \begin{align*}
    \left|\{Q \in \Delta : (x \cdot \R v \cap 6B_Q)^* = I' \}\right| \leq C_2.
  \end{align*}
  By Lemma \ref{length-comp}, the length of each interval $I = x \cdot \R v \cap 6B_{Q_\omega^j}$ is comparable to $\ell(Q_\omega^j) = \tau^j$, and so it suffices to only consider $B_{Q_\omega^j}$ when
  \begin{align*}
    \tau^j \in [C_3^{-1}|J|2^{-k}, C_3\gamma |J|2^{-k}]
  \end{align*}
  for some $C_3 > 0$.  The number of $j$ that are possible are only boundedly many over all $k$, and so it suffices to prove the statement in the special case when $Q$ all have the same scale.  We fix such a scale $j$.
  
  It follows from definition that $I'$ is contained in a ball $B$ of radius $|J|2^{-k}$.  Properties \eqref{christ-partition} and \eqref{christ-ball} of Christ cubes give that, for any $k \in \Z$, the center of the balls $B_{Q_\omega^j}$ are $a_0\tau^j$-separated.  As $\tau^j$ is comparable to $|J|2^{-k}$, the number of balls $B_{Q_\omega^j}$ that intersect $B$ is boundedly many by the doubling condition of $G$, which finishes the proof of the claim.
  
  Thus, letting $I_{x,v} := x \cdot \R v \cap 36a_0^{-1}a_1B_S$, there exists constants $C_4,C_5,C_6 > 0$ depending only on $G$ such that
  \begin{align}
    (*) &\leq C_2\gamma^2 \int_{S^{n-1}} \int_{z_S(G \circleddash v)} \chi_{\{x \in 12a_1a_0^{-1} B_S\}} \times \\
    &\qquad \sum_{i \in \{0,+,-\}} \sum_{k=0}^{m \log 1/\tau + C_1} \sum_{I \in \D_i^k(I_{x,v})} |I| \alpha^{(p)}_f\left(I; \frac{\epsilon}{\gamma}\right) ~dx ~d\mu(v) \notag \\
    &\leq C_4 \Lip_f(C_5 \epsilon \tau^m \ell(S))^p \int_{S^{n-1}} \int_{z_S(G \circleddash v)} \chi_{\{x \in 12a_1a_0^{-1} B_S\}} \ell(S) ~dx ~d\mu(v) \label{sum-ineq-1} \\
    &= C_4 \Lip_f(C_5 \epsilon \tau^m \ell(S))^p \int_{S^{n-1}} \mathscr{H}^{N-1}(12 a_1a_0^{-1} B_S \cap z_S(G \circleddash v)) \ell(S) ~d\mu(v) \notag \\
    &\leq C_6 \Lip_f(C_5 \epsilon \tau^m \ell(S))^p |S|. \notag 
  \end{align}
  For \eqref{sum-ineq-1}, we used Lemma \ref{carleson-1d} and the fact that $|I_{x,v}|$ is comparable to $\ell(S)$.  For the last inequality, we used the fact $\mathscr{H}^{N-1}(10 a_1a_0^{-1}B_S \cap z_S (G \circleddash v)) = (10 a_1 \ell(S))^{N-1}$ and $\ell(S)^N \leq C_7|S|$ for some $C_7 > 0$.
\end{proof}

We prove three more preliminary lemmas that will be useful for the sections to come.  All of these lemmas will be concerned with the deviation of lines in Carnot groups.

\begin{lemma} \label{close-parallel}
  Suppose $G$ is a graded nilpotent Lie group of step $r$.  Let $\rho \in (0,1)$ and $\lambda > 0$.  There exists a constant $C > 0$ depending only on $G$ so that if $g,h,u,v \in G$ so that $d_\infty(g,h) \leq \rho \lambda$, $u \in B_\infty(0,1)$, and $d_\infty(u,v) \leq \rho$, then
  \begin{align*}
    \sup_{t \in [0,\lambda]} d_\infty(g\delta_t(u),h\delta_t(v)) \leq C\rho^{1/r} \lambda.
  \end{align*}
\end{lemma}

Here, $\lambda$ is simply the scale at which we are working on.  The more important quantity is $\rho$, which describes how close the ``unit vectors'' $u$ and $v$ are.

\begin{proof}
  We may suppose that $g = 0$.  Then we get from the fact that $d_\infty(g,h) \leq \rho \lambda$ that
  \begin{align}
    |h_i| \leq (\rho \lambda)^i \label{h-infty-bound}
  \end{align}
  for each $i \in \{1,...,r\}$.  Similarly, we have
  \begin{align*}
    \max_{i \in \{1,...,r\}} |u_i| \leq 1.
  \end{align*}
  As $d_\infty(u,v) \leq \rho \leq 1$, we get that $d_\infty(0,v) \leq 2$ and so
  \begin{align*}
    |v_i| \leq 2^i, \qquad \forall i \in \{1,...,r\}.
  \end{align*}
  We then have
  \begin{align}
    d_\infty&(\delta_t(u), h\delta_t(v)) \notag \\
    &= N_\infty((-tu_1,-t^2 u_2,..., -t^ru_r) \cdot (h_1 + tv_1, h_2 + t^2 v_2 + P_2,..., h_n + t^r v_r + P_n)) \notag \\
    &= N_\infty((h_1 + t(v_1 - u_1), h_2 + t^2(v_2 - u_2) + P_2',..., h_r + t^r(v_r - u_r) + P_r')). \label{gh-drift}
  \end{align}
  For each BCH polynomial $P_k'$, let $Q_k$ denote the sum of the Lie brackets that do not have elements of $h$.  Then one sees that
  \begin{align*}
    \delta_t(u)^{-1} \delta_t(v) = (t(v_1 - u_1), t^2(v_2 - u_2) + Q_2,..., t^r(v_r - u_r) + Q_r).
  \end{align*}
  As $d_\infty(u,v) \leq \rho$, we get that $N_\infty(\delta_t(u)^{-1}\delta_t(v)) \leq \rho t \leq \rho \lambda$ and so
  \begin{align}
    |t^k(v_k - u_k) + Q_k| \leq \rho^k \lambda^k, \qquad \forall k \in \{1,...,r\}. \label{sub-polynomial-bound}
  \end{align}
  Thus, referring to \eqref{gh-drift}, it suffices just to bound $|h_k + P_k' - Q_k|$ by some multiple of $\rho \lambda^k$ for any $k$.  By \eqref{h-infty-bound}, it further reduces to bounding just $|P_k' - Q_k|$.  We fix a $k \in \{2,...,r\}$.  By the BCH formula, we know that $P_k' - Q_k$ is a summation of nested Lie brackets of the form $[x_1,[x_2,...,[x_{j-1},x_j]...]]$ where the number of summands depends only on $G$.  Thus, we further reduce to bounding the norm of the maximum of the nested Lie brackets
  \begin{align*}
    |[x_1,[x_2,...,[x_{j-1},x_j]...]]| \leq \prod_{i=1}^j |x_j|
  \end{align*}
  by $\rho \lambda^k$.  We can define the function $i : \{1,...,j\} \to \{1,...,k-1\}$ to satisfy
  \begin{align}
    x_{i(\ell)} \in \{h_{i(\ell)},t^{i(\ell)} u_{i(\ell)},t^{i(\ell)} v_{i(\ell)}\}. \label{index-function}
  \end{align}
  By the stratified nature of $G$,
  \begin{align*}
    \sum_{\ell = 1}^j i(\ell) = k.
  \end{align*}
  We have that there must be some index (say $\ell$) so that $x_{i(\ell)} = h_{i(\ell)}$ as $P_k' - Q_k$ contains only BCH polynomials where each nested Lie bracket has an element of $h$.  As $\max_{k \in \{1,...,r\}} |u_k| \leq 1$ and $|v_k| \leq 2^k$, by \eqref{h-infty-bound}, \eqref{index-function}, and the fact that $t \in [0,\lambda]$, one gets that
  \begin{align*}
    \prod_{i=1}^j |x_j| \leq |h_{i(\ell)}| \cdot 2^r \lambda^{k-i(\ell)} \leq 2^r \rho^{i(\ell)} \lambda^k.
  \end{align*}
  This is the needed bound to finish the proof.
\end{proof}

We will sometimes need the following, slightly different lemma.

\begin{lemma} \label{close-parallel-2}
  Let $\rho \in (0,1)$ and $\lambda > 0$.  There exists a constant $C > 0$ depending only on $G$ so that if $g,h \in G$ so that $d_\infty(g,h) \leq \rho \lambda$ and $u,v \in \V_1$ so that $|u| \leq 1$ and $|u - v| \leq \rho$, then
  \begin{align*}
    \sup_{t \in [0,\lambda]} d_\infty(ge^{tu},he^{tv}) \leq C\rho^{1/r} \lambda.
  \end{align*}
\end{lemma}

\begin{proof}
  The proof is largely the same.  Note that $|v| \leq 1+\rho \leq 2$.  We have that
  \begin{align*}
    d_\infty(e^u,e^v) = N_\infty((-u,0,...,0),(v,0,...,0)) = N_\infty((v-u,Q_2',...,Q_r')).
  \end{align*}
  One can see that each $Q_j'$ is a finite sum of nested Lie bracket of the form $[x_1,[x_2,...,[x_{j-1},x_j]...]]$ where $x_i$ is either $u$ or $v$ and $|[x_{j-1},x_j]| = |[u,v]| \leq \rho$.  Thus, we get that there exists some constant $C_0 > 0$ depending only on $G$ so that for all $j \in \{2,...,r\}$ we have
  \begin{align*}
    |Q_j'| \leq C_0 2^j \rho.
  \end{align*}
  As in the proof of the previous lemma, we can calculate
  \begin{align}
    d_\infty(e^{tu}, he^{tv}) &= N_\infty((-tu,0,..., 0) \cdot (h_1 + tv, P_2,..., P_n)) \notag \\
    &= N_\infty((h_1 + t(v - u), h_2 + P_2',..., h_r + P_r')).
  \end{align}
  Taking $Q_j$ to be the nested Lie brackets of $P_j'$ without elements of $h$, we see then that $Q_j = t^j Q_j'$.  This gives for all $j \in \{2,...,r\}$ that
  \begin{align*}
    |Q_j| \leq C_02^j \rho \lambda^j.
  \end{align*}
  Comparing this with \eqref{sub-polynomial-bound}, we see that this will not change the proof.  The rest of the proof now just follows as before.
\end{proof}

Now we show that two horizontal line segments which are close on the endpoints must be close all throughout.

\begin{lemma} \label{horizontal-lines-bound}
  Let $\rho \in (0,1)$.  There exists a constant $C > 0$ such that if $f,g : [a,b] \to H$ are $\eta$-Lipschitz constant speed horizontal line segments such that
  \begin{align*}
    d_H(f(a),g(a)) \leq \rho |b-a| \eta, \quad d_H(f(b),g(b)) \leq \rho |b-a| \eta,
  \end{align*}
  then
  \begin{align*}
    \sup_{t \in [a,b]} d_H(f(t),g(t)) \leq C\rho^{1/r} |b-a| \eta.
  \end{align*}
\end{lemma}

\begin{proof}
  By translating and rescaling, we can suppose that $a = 0$, $b = 1$, $f(0) = 0$, and $\eta = 1$.  Let $h$ denote $g(0)$.  Then we get that
  \begin{align*}
    \max_j |h_j|^{1/j} \leq \rho.
  \end{align*}
  As $f$ and $g$ are 1-Lipschitz horizontal line segments, there exists horizontal vectors $u,v \in \V_1$ such that $f(t) = e^{tu}$, $g(t) = h e^{tv}$ and $|u_1| \leq 1$, $|v_1| \leq 1$.  Note that
  \begin{align*}
    N_\infty((-u_1,0,...,0) \cdot (h_1 + v_1, h_2 + P_2, ..., h_s + P_s)) = d_\infty(f(1),g(1)) \leq \rho.
  \end{align*}
  This implies that $|h_1 + v_1 - u_1| \leq \rho$.  As $|h_1| \leq \rho$, we get that $|v_1 - u_1| \leq 2\rho$.  One can now use the statement of Lemma \ref{close-parallel-2} to complete the proof.

\end{proof}

\begin{remark}
  Note that the infinity metric in both Lemma \ref{close-parallel} and \ref{close-parallel-2} can be changed to any homogeneous group metric.  This only changes the constant given by the lemmas, which will now also depend on the new metric.  This is a simple consequence of the equivalence of homogeneous norms.
\end{remark}


\section{Coarse differentiation for maps into metric spaces} \label{M-coarse-diff}
We start with a coarse version of a theorem of \cite{Pauls}.  Given $Q \in \Delta$, $\eta > 0$, and a map $h : G \to (X,d_X)$, define
\begin{multline*}
  \qd_h^M(Q,\eta) :=  \frac{1}{\eta a_0\ell(Q)} \inf \sup \{|d_X(h(xe^{sv}),h(xe^{tu})) - (t-s)w(v)| : \\
  x \in \eta B_Q, v \in S^{n-1}, -3 \eta a_0\ell(Q) \leq s \leq t \leq 3\eta a_0 \ell(Q) \}
\end{multline*}
where the infimum is taken over all functions $w : S^{n-1} \to \R^+$.  Recall that the radius of $B_Q$ is $a_0\ell(Q)$.  In this section, we will prove the following theorem.

\begin{theorem} \label{M-UAAP}
  There exist $\alpha,\zeta,\lambda > 0$ depending only on $G$ so that if $\epsilon \in (0,1/2)$, $m \in \N$, $h : G \to (X,d_X)$ is $\psi$-LLD, and $S \in \Delta$ such that
  \begin{align}
    \ell(S) \geq \lambda \epsilon^{-15(r+1)} \tau^{-m} \psi, \label{M-size-constraint}
  \end{align}
  then
  \begin{align*}
    \sum_{k=0}^m \sum_{Q \in \Delta_k(S)} \left\{|Q| : \qd_h^M(Q,\zeta \epsilon^r) > \epsilon \Lip_h(\psi) \right\} \leq \epsilon^{-\alpha} |S|.
  \end{align*}
\end{theorem}

\begin{remark}
  Notice that if $f$ is actually Lipschitz, then $\psi = 0$ and the restriction \eqref{M-size-constraint} becomes empty.  One then can even take $m = \infty$.
\end{remark}

We now show that this implies Theorem \ref{UAAP-M}.  The same proof will hold, {\it mutatis mutadis}, in the case of superreflexive and Carnot valued Lipschitz maps of Theorem \ref{UAAP-rest} and so we will not reprove them in the following sections.

\begin{proof}[Proof of Theorem \ref{UAAP-M}]
  Construct the Christ cubes of $G$ and take a cube $S \in \{Q \in \Delta : Q \subset B_G\}$ so that $\ell(S)$ is maximal.  Thus, there is some constant $C > 0$ depending only on $G$ so that
  \begin{align*}
    \frac{1}{C} \leq \ell(S) \leq C.
  \end{align*}
  As $f$ is Lipschitz, it is $0$-LLD and so the condition on the size of $S$ in Theorem \ref{M-UAAP} is empty.  The same theorem then gives that there exist $\zeta,\alpha > 0$ depending only on $G$ so that for $\epsilon \in (0,1/2)$, we get
  \begin{align}
    \sum \left\{|Q| : Q \in \Delta, Q \subseteq S, \qd_f^M(Q,\zeta \epsilon^r) > \epsilon \right\} \leq \epsilon^{-\alpha} |S|. \label{e:M-carleson-bnd}
  \end{align}
  Let $m = \lceil \epsilon^{-\alpha} \rceil$, $A = \{Q \in \Delta : Q \subseteq S, \qd_f^{UC}(Q,\zeta \epsilon^r) > \epsilon \}$, and $\Delta_k \cap S = \{Q \in \Delta_k : Q \subseteq S\}$.  Suppose $\bigcup_{k=0}^m (\Delta_k \cap S) \subseteq A$.  By the partitioning property of $\Delta$, we get for any $k \geq 0$ that
  \begin{align*}
    \sum_{Q \in \Delta_k \cap S} |Q| = |S|.
  \end{align*}
  Thus, we have
  \begin{align*}
    \sum_{k=0}^m \sum_{Q \in (\Delta_k \cap S) \cap A} |Q| = (m+1)|S|.
  \end{align*}
  We get a contradiction of \eqref{e:M-carleson-bnd} from the definition of $m$.  Thus, we have proven that there exists some $k \in \{0,...,m\}$, $Q \in \Delta_k \cap S$, and $w : S^{n-1} \to \R$ so that for $\rho = \zeta \epsilon^r a_0 \ell(Q)$, we have
  \begin{align*}
    \sup \left\{ \left| d_X(f(xe^{tv}),f(xe^{sv})) - (t-s) w(v) \right| : x \in \zeta \epsilon^r B_Q, v \in S^{n-1}, -3\rho \leq s < t \leq 3\rho \right\} \leq \epsilon \rho \|f\|_{lip}.
  \end{align*}
  
  As $Q \in \Delta_k$ for $k \in \{0,...,m\}$ and $B_Q$ has radius $a_0 \ell(Q)$, we have that $\zeta \epsilon^r B_Q$ has radius at least
  \begin{align*}
    \zeta \epsilon^r a_0 \ell(Q) \geq \frac{a_0}{C} \zeta \epsilon^r \tau^m \geq \frac{a_0}{C} \zeta \epsilon^r \tau^{\epsilon^{-\alpha}} \geq e^{\epsilon^{-\beta}}
  \end{align*}
  for $\beta$ sufficiently large.  This proves the theorem.
\end{proof}

We start off with a numerical lemma, which will act as a uniform convexity condition for general metric spaces.

\begin{lemma} \label{M-numerical}
  Let $\alpha,\beta,\gamma, \epsilon \geq 0$ such that $\gamma \leq \frac{1}{2}(\alpha + \beta)$.  If
  \begin{align}
    \frac{\alpha^2 + \beta^2}{2} - \gamma^2 \leq \epsilon^2, \label{M-num-condition}
  \end{align}
  then
  \begin{align*}
    \max\{|\alpha - \gamma|,|\beta - \gamma|\} \leq 2\epsilon.
  \end{align*}
\end{lemma}

\begin{proof}
  Let $\frac{1}{2}(\alpha + \beta) = \gamma + \eta$.  We can then let $\alpha = \gamma + \eta + \delta$ and $\beta = \gamma + \eta - \delta$ where $\delta \geq 0$.  Then we have
  \begin{align*}
    \frac{\alpha^2 + \beta^2}{2} - \gamma^2 = \left[ \frac{(\gamma + \eta + \delta)^2 + (\gamma + \eta - \delta)^2}{2} - \gamma^2 \right] = (\gamma + \eta)^2 + \delta^2 - \gamma^2 \geq \delta^2 + \eta^2.
  \end{align*}
  By \eqref{M-num-condition}, we then get that $\max\{\delta, \eta\} \leq \epsilon$.  This gives the result.
\end{proof}

We now prove a lemma that controls how far a function $h$ is from being geodesic with $\partial_h^{(2)}$.

\begin{lemma} \label{M-energy}
  Let $h : [a,b] \to (X,d_X)$.  Then
  \begin{align}
    4^{m-1} (b-a)^2 &\max_{k \in \{0,...,m-1\}} \max_{I \in \D^k([a,b])} \partial_h^{(2)}(a(I),b(I)) \notag \\
    &\geq \sup \left\{ | d_X(h(a + s2^{-m}(b-a)),h(a + t2^{-m}(b-a)))\right. \notag\\
    &\qquad\left.- |t-s|2^{-m}d_X(h(a),h(b)) |^2 : 0 \leq s < t \leq 2^m \right\}. \label{M-energy-eq}
  \end{align}
\end{lemma}

\begin{proof}
  We may assume without loss of generality that $[a,b] = [0,1]$.  For the case when $m = 1$, we then have that
  \begin{align*}
    \partial_h^{(2)}(0,1) = \frac{1}{2} \left(\frac{d_X(h(0),h(1/2))^2 + d_X(h(1/2),h(1))^2}{2^{-2}} \right) - d_X(h(0),h(1))^2.
  \end{align*}
  Setting $\alpha = \frac{d_X(h(0),h(1/2))}{1/2}$, $\beta = \frac{d_X(h(1/2),h(1))}{1/2}$, and $\gamma = d_X(h(0),h(1))$, Lemma \ref{M-numerical} gives that if $\partial_h^{(2)}(0,1) \leq \epsilon^2$, then
  \begin{align*}
    \left| d_X(h(0),h(1/2)) - \frac{1}{2} d_X(h(0),h(1)) \right| &\leq \epsilon, \\
    \left| d_X(h(1/2),h(1)) - \frac{1}{2} d_X(h(0),h(1)) \right| &\leq \epsilon.
  \end{align*}
  This completes the proof for $m = 1$.  Now assume that \eqref{M-energy-eq} is satisfied up to $m$ and suppose the next step:
  \begin{align}
    \max_{k \in \{0,...,m\}} \max_{I \in \D^k([0,1])} \partial_h^{(2)}(a(I),b(I)) \leq \epsilon^2. \label{M-step}
  \end{align}
  By the inductive hypothesis, we have that
  \begin{align}
    \left| d_X(h(s2^{-m}),h(t2^{-m})) - \frac{|t-s|}{2^m} d_X(h(0),h(1)) \right| \leq 4^{m-1} \epsilon, \qquad \forall s, t \in \{0,...,2^m\}. \label{M-induct-hyp1}
  \end{align}
  We need to prove for every $s,t \in \{0,...,2^{m+1}\}$ that
  \begin{align*}
    \left| d_X(h(s2^{-m-1}),h(t2^{-m-1})) - \frac{|t-s|}{2^{m+1}} d_X(h(0),h(1)) \right| \leq 4^m \epsilon.
  \end{align*}
  Let $p = 2j + 1$ where $j \in \{0,...,2^m\}$.  Then we have by \eqref{M-induct-hyp1} that
  \begin{align}
    \left| \frac{1}{2} d_X(h(j2^{-m}), h((j+1)2^{-m})) - \frac{1}{2^{m+1}} d_X(h(0), h(1)) \right| \leq \frac{4^{m-1}}{2} \epsilon. \label{M-induct-hyp}
  \end{align}
  Then substituting
  \begin{align*}
    \alpha &= \frac{d_X(h(j2^{-m}),h(p))}{2^{-m-1}},\\
    \beta &= \frac{d_X(h(p),h((j+1)2^{-m})}{2^{-m-1}},\\
    \gamma &= \frac{d_X(h(j2^{-m}),h((j+1)2^{-m}))}{2^{-m}},
  \end{align*}
  into Lemma \ref{M-numerical}, the fact that $\partial_h^{(2)}(j2^{-m},(j+1)2^{-m}) \leq \epsilon^2$ once again gives us that
  \begin{align*}
    \left| d_X(h(j2^{-m}),h(p)) - \frac{1}{2} d_X(h(j2^{-m}),h((j+1)2^{-m})) \right| &\leq 2^{-m}\epsilon, \\
    \left| d_X(h(p),h((j+1)2^{-m})) - \frac{1}{2} d_X(h(j2^{-m}),h((j+1)2^{-m})) \right| &\leq 2^{-m}\epsilon.
  \end{align*}
  This, along with \eqref{M-induct-hyp} gives that
  \begin{align}
    \left| d_X(h(j2^{-m}),h(p)) - \frac{1}{2^{m+1}} d_X(h(0),h(1)) \right| &\leq \left( \frac{4^{m-1}}{2} + 2^{-m} \right) \epsilon, \label{M-odd-1} \\
    \left| d_X(h(p),h((j+1)2^{-m})) - \frac{1}{2^{m+1}} d_X(h(0),h(1)) \right| &\leq \left( \frac{4^{m-1}}{2} + 2^{-m} \right) \epsilon. \label{M-odd-2}
  \end{align}
  Let $s,t \in \{0,...,2^{m+1}\}$ where $s < t$.  If both $s$ and $t$ are even, then the inductive case gives the result.  Thus, we may suppose that one is odd, say $s = 2j + 1$ for some $j \in \{0,...,2^m\}$.  Suppose $t = 2\ell + 1$ is odd for some $\ell \in \{0,...,2^m\}$.  Then we have
  \begin{align*}
    &d_X(h(s2^{-m-1}),h(t2^{-m-1})) - \frac{t-s}{2^{m+1}} d_X(h(0),h(1)) \\
    &\quad \leq d_X(h(s2^{-m-1}),h((j+1)2^{-m})) - \frac{1}{2^{m+1}} d_X(h(0),h(1)) + d_X(h((j+1)2^{-m}),h(\ell2^{-m})) \\
    &\qquad - \frac{\ell - j - 1}{2^m} d_X(h(0),h(1)) + d_X(h(\ell2^{-m}),h(t2^{-m-1})) - \frac{1}{2^{m+1}} d_X(h(0),h(1)) \\
    &\quad \leq 2\left( \frac{4^{m-1}}{2} + 2^{-m} \right)\epsilon + 4^{m-1} \epsilon \leq 4^m \epsilon,
  \end{align*}
  where we've used \eqref{M-induct-hyp1}, \eqref{M-odd-1}, and \eqref{M-odd-2} for the penultimate inequality.  For the other direction, we have
  \begin{align*}
    &d_X(h(s2^{-m-1}),h(t2^{-m-1})) - \frac{t-s}{2^{m+1}} d_X(h(0),h(1)) \\
    &\quad \geq d_X(h(j2^{-m}),h((\ell+1)2^{-m})) - \frac{\ell+1-j}{2^m} d_X(h(0),h(1)) - d_X(h(j2^{-m}),h(s2^{-m-1})) \\
    &\qquad + \frac{1}{2^{m+1}} d_X(h(0),h(1)) - d_X(h((\ell+1)2^{-m}),h(t2^{-m})) + \frac{1}{2^{m+1}} d_X(h(0),h(1)) \\
    &\quad \geq - 4^m \epsilon.
  \end{align*}
  The case when $t$ is even follows similarly.
\end{proof}

We will also use the following lemma many times.  It essentially states that, under small perturbations, $\partial_h^{(2)}$ does not change too much.

\begin{lemma} \label{M-partial-move}
  Let $h : I \to (X,d_X)$ be $\psi$-LLD for some (possibly infinite) interval $I \subseteq \R$ and let $c,d \in I$ so that $\frac{d-c}{4} \geq \psi$ and
  \begin{align}
    \frac{\partial_h^{(2)}(c,d)}{30 \Lip_h(\psi)^2} (d-c) =: \rho \geq \psi. \label{M-large-separation}
  \end{align}
  If we choose $s,t \in I$ so that $|s - c| \leq \rho$ and $|t- d| \leq \rho$, then $\partial_h^{(2)}(s,t) > \frac{1}{5} \partial_h^{(2)}(c,d)$.
\end{lemma}

This lemma may be a little hard to interpret.  One should think of $\rho$ as a threshhold.  The lemma is saying that as long as the points $s$ and $t$ are not perturbed beyond $\rho$, then $\partial_h^{(2)}$ does not decrease too much.  Of course the threshhold should then depend on how large $\partial_h^{(2)}(s,t)$ is as there is less margin for error when $\partial_h^{(2)}(s,t)$ is small.  That $\rho$ is taken to be larger than $\psi$ will allow us to use the $\psi$-LLD bounds.

\begin{proof}
  As $d-c \geq 4\psi$, we get by looking at the definition of $\partial_h^{(2)}$ that 
  \begin{align}
    \frac{\partial_h^{(2)}(c,d)}{\Lip_h(\psi)^2} \leq 1. \label{M-partial-lip}
  \end{align}
  Thus, $\rho \leq \frac{1}{4}(d-c)$, and so we get that $t-s \leq \frac{3}{2}(d-c)$.  The proof is now a direct computation:
  \begin{align*}
    \partial_h^{(2)}(s,t) &= \frac{1}{2} \left[ \left( \frac{d_X(h(s), h((s+t)/2))}{(t-s)/2} \right)^2 + \left( \frac{d_X(h((s+t)/2),h(t))}{(t-s)/2} \right)^2 \right] - \left( \frac{d_X(h(s),h(t))}{t-s} \right)^2 \\
    &\quad \geq \frac{1}{2} \left[ \left( \frac{d_X(h(c), h((c+d)/2)) - 2 \rho \Lip_h(\psi)}{(t-s)/2} \right)^2 + \left( \frac{d_X(h((c+d)/2),h(d)) - 2\rho \Lip_h(\psi)}{(t-s)/2} \right)^2 \right] \\
    &\qquad - \left( \frac{d_X(h(c),h(d)) + 2 \rho \Lip_h(\psi)}{t-s} \right)^2 = (*),
  \end{align*}
  For the inequality, we used $\eqref{M-large-separation}$ and the fact that one can use $\rho \Lip_h(\psi)$ to bound $d_X(h(x),h(y))$ for all $|x-y| \leq \rho$.  Continuing, we get
  \begin{align}
    (*) &\geq \left( \frac{d-c}{t-s} \right)^2 \partial_h^{(2)}(c,d) - \frac{4\rho\Lip_h(\psi) d_X(h(c),h(d)) }{(t-s)^2} - \left(\frac{2\rho\Lip_h(\psi)}{t-s}\right)^2 \notag \\
    &\qquad - \frac{8\rho \Lip_h(\psi) [d_X(h(c),h((c+d)/2)) + d_X(h((c+d)/2),h(d))]}{(t-s)^2} \notag \\
    &\geq \left( \frac{d-c}{t-s} \right)^2 \left[ \partial_h^{(2)}(c,d) - \frac{8\rho \Lip_h(\psi)^2}{d-c} - \frac{4\rho\Lip_h(\psi)^2}{d-c} - \frac{4\rho^2\Lip_h(\psi)^2}{(d-c)^2} \right] \label{M-lipschitz-use} \\
    &\overset{\eqref{M-large-separation} \wedge \eqref{M-partial-lip}}{>} \left( \frac{d-c}{t-s} \right)^2 \frac{1}{2} \partial_h^{(2)}(c,d) \notag \\
    &\geq \frac{1}{5} \partial_h^{(2)}(c,d). \notag
  \end{align}
  In \eqref{M-lipschitz-use}, we used the fact that $\frac{d-c}{2} \geq \psi$ to show that
  \begin{align*}
    d_X(h(c),h(d)) &\leq d_X\left( h(c), h\left( \frac{c+d}{2} \right) \right) + d_X\left( h\left( \frac{c+d}{2} \right), h(d) \right) \leq \Lip_h(\psi)(d-c).
  \end{align*}
\end{proof}

\begin{remark}
  Note that we never used the fact that $s$ and $t$ are on the same line as $c$ and $d$.  We only used their metric relations.  Thus, the lemma holds for maps from metric spaces that contain isometric copies of lines as long as $s$ and $t$ as well as $c$ and $d$ lie on isometric copies of lines and we have that $d(s,c) \leq \rho$, $d(t,d) \leq \rho$, and $d\left( \frac{s+t}{2},\frac{c+d}{2} \right) \leq \rho$.
\end{remark}

We now translate Lemma \ref{M-energy} into a bound concerning $\alpha_h^{(2)}$.

\begin{lemma} \label{M-alpha-geodesic}
  Fix $\epsilon \in (0,1)$ and let $h : [a,b] \to (X,d_X)$ be $\psi$-LLD.  If $\frac{1}{40000} \epsilon^5 (b-a) > \psi$ and $\alpha_h^{(2)}\left([a,b]; \frac{\epsilon}{16} \right) \leq \frac{\epsilon^{14}}{10^{13}} \Lip_h(\psi)^2$, then
  \begin{align*}
    \left| d_X(h(s),h(t)) - |t-s| \frac{d_X(h(a),h(b))}{b-a} \right| \leq \epsilon (b-a) \Lip_h(\psi), \qquad \forall s,t \in [a,b].
  \end{align*}
\end{lemma}

\begin{proof}
  Let
  \begin{align}
    m = \left\lceil \log_2 \frac{1}{\epsilon} \right\rceil + 2 \label{M-log-eps}
  \end{align}
  and suppose for all $j,k \in \{0,...,2^m\}$ that
  \begin{align}
    \left| d_X(h(a + j2^{-m}(b-a)), h(a + k2^{-m}(b-a))) - \frac{|j-k|}{2^m} d_X(h(a),h(b)) \right| \leq \frac{\epsilon}{3} (b-a)\Lip_h(\psi). \label{M-discrete-geo}
  \end{align}
  Let $s,t \in [a,b]$ such that $s < t$.  There exist $j,k \in \{0,...,2^m\}$ so that
  \begin{align}
    0 &\leq a + j2^{-m}(b-a) - s \leq 2^{-m}(b-a), \notag \\
    0 &\leq t - a - k2^{-m}(b-a) \leq 2^{-m}(b-a). \label{M-eps-net}
  \end{align}
  As $2^{-m}(b-a) \geq \frac{1}{40000} \epsilon^5 (b-a) > \psi$, we then have that
  \begin{align*}
    &d_X(h(s),h(t)) \\
    &\quad \leq d_X(h(s),h(a + j2^{-m}(b-a))) + d_X(h(a + j2^{-m}(b-a)),h(a + k2^{-m}(b-a))) \\
    &\qquad+ d_X(h(a + k2^{-m}(b-a)),h(t)) \\
    &\overset{\eqref{M-discrete-geo} \wedge \eqref{M-eps-net}}{\leq} 2^{-m}(b-a) \Lip_h(\psi) + (k-j)2^{-m} d_X(h(a),h(b)) + \frac{\epsilon}{3}(b-a) \Lip_h(\psi) + 2^{-m}(b-a)\Lip_h(\psi) \\
    &\overset{\eqref{M-log-eps} \wedge \eqref{M-eps-net}}{\leq} (a + j2^{-m}(b-a) - s) \frac{d_X(h(a),h(b))}{b-a} + (k-j)2^{-m} d_X(h(a),h(b)) \\
    &\qquad + (t - a - k2^{-m}(b-a)) \frac{d_X(h(a),h(b))}{b-a} + \epsilon (b-a) \Lip_h(\psi) \\
    &\quad \leq (t - s) \frac{d_X(h(a),h(b))}{b-a} + \epsilon (b-a) \Lip_h(\psi).
  \end{align*}
  That $d_X(h(s),h(t)) \geq \frac{t - s}{b-a}  d_X(h(a),h(b)) - \epsilon (b-a) \Lip_h(\psi)$ is proven similarly.  Thus it suffices to prove \eqref{M-discrete-geo}, which, by Lemma \ref{M-energy}, further reduces to proving that
  \begin{align*}
    \max_{k \in \{0,...,{m-1}\}} \max_{I \in \D^k([a,b])} \partial_h^{(2)}(a(I),b(I)) \leq \frac{\epsilon^4}{144} \Lip_h(\psi)^2 \overset{\eqref{M-log-eps}}{\leq} 4^{-m+1} \frac{\epsilon^2}{9} \Lip_h(\psi)^2 .
  \end{align*}
  Suppose this were not the case.  Then there exists some $k \in \{0,...,m-1\}$ and some $[s,t] \in \D^k([a,b])$ so that
  \begin{align*}
    \partial_h^{(2)}(s,t) > \frac{\epsilon^4}{144} \Lip_h(\psi)^2.
  \end{align*}
  Let $s',t' \in [a,b]$ so that $s' \in [s,s + \frac{1}{40000}\epsilon^5(b-a)]$ and $t' \in [t - \frac{1}{40000}\epsilon^5(b-a), t]$.  As
  \begin{align*}
    \frac{\partial_h^{(2)}(s,t)}{30 \Lip_h(\psi)^2} (t-s) > \frac{\epsilon^4}{5000} 2^{-k} (b-a) \geq \frac{\epsilon^5}{40000} (b-a) \geq \psi,
  \end{align*}
  and
  \begin{align*}
    \frac{t-s}{4} \geq 2^{-m-2}(b-a) \geq \frac{\epsilon}{32}(b-a) \geq \psi,
  \end{align*}
  we get by an application of Lemma \ref{M-partial-move} that
  \begin{align*}
    \partial_h^{(2)}(s',t') > \frac{\epsilon^4}{720} \Lip_h(\psi)^2.
  \end{align*}
  We can also bound
  \begin{align*}
    \left| t - s - 2 \cdot \frac{\epsilon^5}{40000}(b-a) \right| \geq \frac{\epsilon}{16} (b-a).
  \end{align*}
  We can now bound
  \begin{align*}
    \alpha_h^{(2)}\left([a,b]; \frac{\epsilon}{16}\right) &\geq \left(1 - \frac{\epsilon}{16} \right)^2 (b-a)^{-2} \int_s^{s+\epsilon^5(b-a)/40000} \int_{t-\epsilon^5(b-a)/40000}^t \partial_h^{(2)}(x,y) ~dy ~dx \\
    &> \frac{1}{2} \left(\frac{\epsilon^5}{40000} \right)^2 \frac{\epsilon^4}{720} \Lip_h(\psi)^2.
  \end{align*}
  This is a contradiction of the hypothesis of the lemma.
\end{proof}

We now return to maps from Carnot groups to metric spaces.  We prove that we can bootstrap the averaging bound of $\alpha_h^{(2)}(Q)$ to a supremum bound.

\begin{lemma} \label{M-sup-bound}
  Let $Q \in \Delta$, $\epsilon \in (0,1/2)$, and $h : G \to (X,d_X)$ be $\psi$-LLD.  There exist constants $\lambda, \alpha_0, \gamma > 0$ depending only on $G$ so that if $\alpha_h^{(2)}\left(Q; \frac{\epsilon}{\gamma} \right) \leq \epsilon^{\alpha_0} \Lip_h(\psi)^2$ and $\ell(Q) \geq \lambda \epsilon^{-15} \psi$, then 
  \begin{align*}
    \sup \left\{ \alpha^{(2)}_h\left(x \cdot \R v \cap 3B_Q; \frac{\epsilon}{16} \right) : v \in S^{n-1}, x \in z_Q(G \circleddash v), x \cdot \R v \cap B_Q \neq \emptyset \right\} \leq \frac{\epsilon^{14}}{10^{13}} \Lip_h(\psi)^2.
  \end{align*}
\end{lemma}

\begin{proof}
  Suppose otherwise.  Then there exists some $w \in S^{n-1}$ and $g \in z_Q(G \circleddash w)$ such that $J = g \cdot \R w \cap 3B_Q$, $J \cap B_Q \neq \emptyset$, and $\alpha_h^{(2)}\left(J; \frac{\epsilon}{16} \right) > \frac{\epsilon^{14}}{10^{13}} \Lip_h(\psi)^2$. We parameterize $J = g \cdot [a,b]w$.  Then there exist some $s,t \in [a,b]$ so that $t - s > \frac{\epsilon}{16} (b-a)$ and
  \begin{align}
    \partial_h^{(2)}(ge^{sw},ge^{tw}) > \frac{1}{2} \frac{\epsilon^{14}}{10^{13}} \Lip_h(\psi)^2. \label{M-big-partial}
  \end{align}
  Letting $u \in J \cap B_Q$, we see that
  \begin{align}
    \dcc(z_Q,g) \leq \dcc(z_Q,u) + \dcc(u,g) \leq 2 \dcc(z_Q,u) \leq 2 a_0 \ell(Q). \label{M-small-u}
  \end{align}

  For each $v \in S^{n-1}$, let $g_v \in z_Q(G \circleddash v)$ be such that there exists some $\lambda$ so that $g_v e^{\lambda v} = g$.  If $|v - w| \leq C_1\epsilon^{15r}$ in addition, where $C_1 > 0$ is some constant to be determined, then as $\dcc(z_Q,g) \leq 2a_0\ell(Q)$, we have by Euclidean geometry and the fact that $\pi$ is distance preserving on horizontal lines that
  \begin{align*}
    \dcc(g,g_v) = d_{\R^n}(\pi(g),\pi(z_Q(G \circ v))) = d_{\R^n}(\pi(z_Q^{-1}g),\pi(G \circ w)) \leq 2a_0C_1\epsilon^{15r} \ell(Q).
  \end{align*}
  Here, we used the fact that $z_Q^{-1}g$ lies $\pi(G \circ w)$ and $|\pi(z_Q^{-1}g)| \leq \dcc(z_Q,g) \leq 2a_0 \ell(Q)$.  We now fix such a $v \in S^{n-1}$.

  Now let $g' \in \Bcc(g_v,C_2\epsilon^{15r} \ell(Q)) \cap z_Q(G \circleddash v)$ for some $C_2 > 0$ to be determined later.  We have
  \begin{align*}
    \dcc(g,g') \leq \dcc(g,g_v) + \dcc(g_v,g') \leq (2a_0C_1 + C_2) \epsilon^{15r} \ell(Q).
  \end{align*}
  As $g \cdot [a,b]w = g \cdot \R w \cap 3B_Q$ and $g \cdot [a,b]w \cap B_Q \neq \emptyset$, by Lemma \ref{length-comp} there is some constant $C_3 > 0$ so that $\ell(Q) \leq C_3(b-a)$.  Thus , by selecting the previous $C_1$ and $C_2$ small enough, Lemma \ref{close-parallel-2} gives that
  \begin{align}
    \dcc(ge^{\lambda w},g' e^{\lambda v}) \leq \frac{1}{10^{16}} \epsilon^{15} \min \{(b-a),a_0\ell(Q)\}, \qquad \forall \lambda \in [a,b]. \label{M-close-lines}
  \end{align}
  Let $I = g' \cdot [a,b]v$.  By \eqref{M-close-lines} and the properties of $J$, we have that $I \subset 6B_Q$ and $I \cap 2B_Q \neq \emptyset$.  Given that $t-s > \frac{\epsilon}{16}(b-a)$ and the fact that $b-a$ is comparable to $\ell(Q)$, if we chose $\lambda$ small enough so that
  \begin{align*}
    \frac{\partial_h^{(2)}(s,t)}{30 \Lip_h(\psi)^2} (t-s) \overset{\eqref{M-big-partial}}{\geq} \frac{1}{10^{16}} \epsilon^{15} (b-a) \geq \lambda \epsilon^{15} \ell(Q) \geq \psi,
  \end{align*}
  then Lemma \ref{M-partial-move} (with its following remark) along with \eqref{M-close-lines} gives
  \begin{align}
    \partial_h^{(2)}(g'e^{sv}, g'e^{tv}) > \frac{1}{10} \frac{\epsilon^{14}}{10^{13}} \Lip_h(\psi)^2. \label{M-big-theta}
  \end{align}
  Now let $I' = g' \cdot \R v \cap 6B_Q$. By Lemma \ref{length-comp}, we have that there exists some constant $C_4 > 0$ such that
  \begin{align}
    |I'| \leq C_4|I| = C_4|J|. \label{M-containing-line}
  \end{align}
  By choosing a large enough $C_4$ depending only on $C_1$ and $C_2$, we can have that \eqref{M-containing-line} holds for all possible choices of $v$ and $g'$.  Note then that
  \begin{align*}
    \frac{t-s}{|I'|} \geq \frac{\epsilon}{16} \frac{|I|}{|I'|} \geq \frac{\epsilon}{16C_4}.
  \end{align*}
  Now, using Lemma \ref{M-partial-move} in the same way as we did in the proof of Lemma \ref{M-alpha-geodesic}, if we set $\nu := C_5(b-a) \epsilon^{15}$ for some small enough $C_5$ depending only on $G$, then $\nu \leq \frac{\epsilon}{64C_4} |I'|$ and we get that there exists some $C_6 > 0$ so that
  \begin{align}
    \alpha_h^{(2)}\left(g' \cdot \R v' \cap 6B_Q; \frac{\epsilon}{32C_4} \right) &= \left(1 - \frac{\epsilon}{32 C_4} \right)^2 \frac{1}{|I'|^2} \underset{\begin{smallmatrix} {a(I') \leq x < y \leq b(I'),} \\ {y-x > \frac{\epsilon}{32 C_4} |I'|} \end{smallmatrix}}{\iint} \partial_h^{(2)}(g'e^{xv},g'e^{yv}) ~dy ~dx \notag \\
    &\overset{\eqref{M-containing-line}}{\geq} \frac{1}{2C^2_4|J|^2} \int_{t- \nu}^t \int_s^{s + \nu} \partial_h^{(2)}(g'e^{xv},g'e^{yv}) ~dx ~dy \notag \\
    &\overset{\eqref{M-big-theta}}{\geq} C_6\epsilon^{44} \Lip_h(\psi)^2. \label{M-large-alpha}
  \end{align}
  We have just proven $\eqref{M-large-alpha}$ for all $v \in S^{n-1}$ so that $|v-w| \leq C_1 \epsilon^{15r}$ and $g' \in \Bcc(g_v,C_2\epsilon^{15r}\ell(Q)) \cap (G \circleddash v)$.  Using the fact that $\mathscr{H}^{N-1}(\Bcc(g_v,C_2\epsilon^{15r} \ell(Q)) \cap (G \circleddash v)) = (C_2\epsilon)^{15r(N-1)} \ell(Q)^{N-1}$, we get that there exists some $C_7 > 0$ depending on the previous constants so that
  \begin{align*}
    \alpha_h^{(2)}&\left(Q, \frac{\epsilon}{32 C_4} \right) \\
    &= \ell(Q)^{1-N} \int_{S^{n-1}} \int_{z_Q(G \circleddash v)} \chi_{\{x \cdot \R v \cap 2B_Q \neq \emptyset\}} \alpha_h^{(2)}\left(x \cdot \R v \cap 6B_Q;\frac{\epsilon}{16 C_4}\right) ~dx ~d\mu(v) \\
    &\geq \ell(Q)^{1-N} \int_{|v-w| \leq C_1\epsilon^{15r}} \int_{\Bcc(g_v,C_2\epsilon^{15r}\ell(Q)) \cap (G \circleddash v)} \chi_{\{x \cdot \R v \cap 2B_Q \neq \emptyset\}} \alpha_h^{(2)}\left(x \cdot \R v \cap 6B_Q; \frac{\epsilon}{16 C_4}\right) ~dx ~d\mu(v) \\
    &\overset{\eqref{M-large-alpha}}{\geq} C_7 \epsilon^{15r(N-1)} \epsilon^{15r(N-1)} \epsilon^{44} \Lip_h(\psi)^2.
  \end{align*}
  By choosing $\gamma = 32C_4$ and $\alpha_0 > 0$ large enough in the hypothesis of the lemma and using the fact that $\epsilon < 1/2$, we have that this is a contradiction.
\end{proof}

We now prove that, if all horizontal lines intersecting a ball are close enough to geodesic on a long interval, then on a smaller controlled scale, all lines have ``slopes'' that depend only on the direction $v \in S^{n-1}$ (\textit{i.e.} are left invariant).  We will actually prove a slightly more general theorem than is need right now (the presence of the $\chi$).  Proving this general form will be useful in the next two sections when we need to expand the ball on which we have almost geodesic behavior.  For now, one can just take $\chi = 1$.

\begin{lemma} \label{M-geodesic-close}
  Let $\epsilon \in (0,1)$, $v \in S^{n-1}$, $x \in G$, $\rho > 0$, $\chi \geq 1$, and $h : G \to (X,d_X)$ be $\psi$-LLD.  Suppose $\rho \geq 24 \chi \epsilon^{-1} \psi$.  There exists a constant $C \in (0,1)$ depending on $G$ and $\chi$ so that if every $g \in \Bcc(x, \rho)$ satisfies
  \begin{align}
    \sup_{-2\rho \leq s < t \leq 2\rho} \left| d_X(h(ge^{sv}),h(ge^{tv})) - \frac{t-s}{4\rho} d_X(h(ge^{-2\rho v}),h(ge^{2\rho v})) \right| \leq \frac{1}{2} C \epsilon^{r+1} \rho \Lip_h(\psi), \label{almost-geodesic}
  \end{align}
  then there exists some $L_v \leq \Lip_h(\psi)$ such that, for all $g \in \Bcc(x, C\chi \epsilon^r \rho)$ we have that
  \begin{align*}
    \sup_{-3C\epsilon^r\rho \leq s < t \leq 3C\epsilon^r\rho} | d_X(h(ge^{sv}),h(ge^{tv})) - (t-s) L_v | \leq C \epsilon^{r+1}\rho \Lip_h(\psi).
  \end{align*}
\end{lemma}

\begin{proof}
  We may suppose without loss of generality that $x = 0$.  We choose $C$ to be small enough (while allowing ourselves to choose $C$ again even smaller) so that Lemma \ref{close-parallel-2} gives that
  \begin{align*}
    \dcc(e^{2 \rho v},ge^{2 \rho v}) \leq \frac{1}{13}  \epsilon \rho, \qquad \forall g \in \Bcc(0, C\chi\epsilon^r\rho).
  \end{align*}
  Using the fact that $\frac{1}{13}  \epsilon \rho \geq \psi$, we get that
  \begin{align}
    d_X(h(e^{2 \rho v}),h(ge^{2 \rho v})) \leq \frac{1}{13}  \epsilon \rho \Lip_h(\psi), \qquad \forall g \in \Bcc(0, C\chi\epsilon^r\rho). \label{sublinear}
  \end{align}
  Let $g \in \Bcc(0, C\chi\epsilon^r\rho)$ and suppose
  \begin{align*}
    d_X(h(e^{-2\rho v}),h(e^{2 \rho v})) - d_X(h(ge^{-2 \rho v}),h(ge^{2 \rho v})) > \frac{1}{3}  \rho \Lip_h(\psi).
  \end{align*}
  Then by \eqref{almost-geodesic} we have that
  \begin{align*}
    d_X(h(0),h(e^{2\rho v})) - d_X(h(g),h(ge^{2\rho v})) > \left( \frac{1}{6}  \epsilon - \frac{1}{2} C  \epsilon^{r+1} \right) \rho \Lip_h(\psi).
  \end{align*}
  We get by the triangle inequality
  \begin{align*}
    d_X(h(e^{2\rho v}),h(ge^{2\rho v})) &\geq d_X(h(e^{2\rho v}),h(0)) - d_X(h(g),h(ge^{2\rho v})) - d_X(h(0),h(g)) \\
    &> \left(\frac{1}{6} \epsilon - \frac{1}{2} C\epsilon^{r+1} - \frac{1}{24} \epsilon\right)  \rho \Lip_h(\psi) \geq \frac{1}{12} \epsilon\rho \Lip_h(\psi).
  \end{align*}
  In the second to last inequality, we used the fact that $\dcc(g,0) \leq C \chi\epsilon^r \rho \leq \frac{\epsilon}{24} \rho$ for sufficiently small $C$.  This is a contradiction of \eqref{sublinear}.  Similarly, we can prove $d_X(h(ge^{-2\rho v}),h(ge^{2\rho v})) - d_X(h(e^{-2\rho v}),h(e^{2\rho v})) \leq \frac{1}{3}  \epsilon \rho \Lip_h(\psi)$ and so we get for all $g \in \Bcc(0,C\chi\epsilon^r\rho)$ that
  \begin{align}
    \frac{1}{4} |d_X(h(ge^{-2\rho v}),h(ge^{2\rho v})) - d_X(h(e^{-2\rho v}),h(e^{2\rho v}))| \leq \frac{1}{12}  \epsilon\rho \Lip_h(\psi). \label{M-close-long-lines}
  \end{align}
  Thus, let $L_v = \frac{1}{4\rho} d_X(h(e^{-2\rho v}),h(e^{2\rho v}))$.  Letting $-3C\epsilon^r\rho \leq s < t \leq 3C \epsilon^r\rho$ and $g \in \Bcc(0,C\chi\epsilon^r \rho)$, we get
  \begin{align*}
    |d_X(h(ge^{sv})&,h(ge^{tv})) - (t-s) L_v| \\
    &\leq \left| d_X(h(ge^{sv}),h(ge^{tv})) - \frac{t-s}{4\rho} d_X(h(ge^{-2\rho v}), h(ge^{2\rho v})) \right| \\
    &\qquad + \frac{t-s}{4\rho} \left| d_X(h(ge^{-2\rho v}),h(ge^{2\rho v})) - d_X(h(e^{-2\rho v}),h(e^{2\rho v})) \right| \\
    &\overset{\eqref{almost-geodesic} \wedge \eqref{M-close-long-lines}}{\leq} \frac{1}{2} C  \epsilon^{r+1} \rho\Lip_h(\psi) + \frac{1}{2} C  \epsilon^{r+1} \rho\Lip_h(\psi) \\
    &\leq C  \epsilon^{r+1} \rho\Lip_h(\psi).
  \end{align*}
\end{proof}

\begin{lemma} \label{M-final}
  Let $\alpha_0,\gamma > 0$ be the constants from Lemma \ref{M-sup-bound} and $C \in (0,1)$ be the constant from Lemma \ref{M-geodesic-close} associated to $\chi = 1$.  There exist $\lambda_0 > 0$ depending only on $G$ so that if $\ell(Q) \geq \lambda_0 \epsilon^{15(r+1)} \psi$ and
  \begin{align*}
    \alpha_h^{(2)}\left(Q; \frac{1}{4\gamma} C a_0\epsilon^{r+1} \right) \leq \left(\frac{1}{4} C a_0\epsilon^{r+1} \right)^{\alpha_0} \Lip_h(\psi)^2,
  \end{align*}
  then there is a function $w : S^{n-1} \to \R^+$ so that, for all $x \in C \epsilon^r B_Q$ and $v \in S^{n-1}$, we have
  \begin{align*}
    | d_X(h(xe^{sv}),h(xe^{tv})) - |t-s| w(v) | \leq \epsilon \cdot C \epsilon^r a_0 \ell(Q) \Lip_h(\psi), \qquad \forall s,t \in [-3C \epsilon^r a_0\ell(Q),3C \epsilon^r a_0\ell(Q)].
  \end{align*}
\end{lemma}

\begin{proof}
  Let $R = a_0\ell(Q)$, the radius of $B_Q$.  By choosing $\lambda_0$ large enough, we get by applying Lemma \ref{M-sup-bound} to the hypothesis that
  \begin{multline*}
    \sup \left\{ \alpha^{(p)}_h\left(x \cdot \R v \cap 3B_Q; \frac{1}{64} C a_0\epsilon^{r+1} \right) : v \in S^{n-1}, x \in G \circleddash v, x \cdot \R v \cap B_Q \neq \emptyset \right\} \\
    \leq \frac{1}{10^{13}} \left( \frac{1}{4} C a_0\epsilon^{r+1} \right)^{14} \Lip_h(\psi)^2.
  \end{multline*}
  Then, by Lemma \ref{M-alpha-geodesic}, choosing $\lambda_0$ sufficiently large again, we get that for all $v \in S^{n-1}$ and $x \in G \circleddash v$ where $x \cdot \R v \cap B_Q \neq \emptyset$ that
  \begin{align}
    \sup_{a \leq s \leq t \leq b} \left| d_X(h(xe^{tv}),h(xe^{sv})) - \frac{t-s}{b-a} d_X(h(xe^{bv}),h(xe^{av})) \right| \leq \frac{1}{4} C \epsilon^{r+1} R \Lip_h(\psi), \label{eq:M-initial-close-geodesic}
  \end{align}
  where $x \cdot [a,b]v = x \cdot \R v \cap 3B_Q$.  Note that for any $x \in B_Q$, we have that $x \cdot [-2R,2R] v \subset y \cdot \R v \cap 3B_Q$ where $y \in G \circleddash v$.  Thus, for all $v \in S^{n-1}$ and $x \in B_Q$, we get that
  \begin{align*}
    \sup_{-2R \leq s \leq t \leq 2R} \left|d_X(h(xe^{tv}),h(xe^{sv})) - \frac{t-s}{4R} d_X(h(xe^{-2Rv}),h(xe^{2Rv})) \right| \leq \frac{1}{2} C \epsilon^{r+1} R \Lip_h(\psi).
  \end{align*}
  Here, we lost a factor of $\frac{1}{2}$ from using \eqref{eq:M-initial-close-geodesic} twice.  We then get from Lemma \ref{M-geodesic-close} that for each $v \in S^{n-1}$ there exists a $w(v) \in \R$ so that for all $x \in C\epsilon^r B_Q$ we have
  \begin{align*}
    \sup_{-3C \epsilon^rR \leq s \leq t \leq 3C \epsilon^rR} \left| d_X(h(xe^{tv}),h(xe^{sv})) - (t-s) w(v) \right| \leq C \epsilon^{r+1} R \Lip_h(\psi).
  \end{align*}
\end{proof}

\begin{proof}[Proof of Theorem \ref{M-UAAP}]
  Lemma \ref{M-final} shows that there exists some $\rho,\lambda_0,\zeta,\alpha_1 > 0$ so that if $\epsilon \in (0,1/2)$, $\qd_h^M(Q,\zeta \epsilon^r) > \epsilon \Lip_h(\psi)$ and $\ell(Q) \geq \lambda_0\epsilon^{-15(r+1)} \psi$, then $\alpha_h^{(p)}\left(Q; \rho \epsilon^{r+1} \right) > \epsilon^{-\alpha_1} \Lip_h(\psi)$.  Thus, if $\ell(S) \geq \lambda \epsilon^{-15(r+1)} \tau^{-m} \psi$ for some $\lambda \geq \lambda_0$ to be determined, then
  \begin{align*}
    \sum_{k=0}^m \sum_{Q \in \Delta_k(S)} \left\{|Q| : \qd_h^M(Q,\zeta \epsilon^r) > \epsilon \Lip_h(\psi) \right\} &\leq \epsilon^{-\alpha_1} \Lip_h(\psi)^{-2} \sum_{k=0}^m \sum_{Q \in \Delta_k(S)} \alpha_h^{(p)}\left( Q; \rho\epsilon^{r+1}\right)|Q| \\
    &\leq \epsilon^{-\alpha} |S|.
  \end{align*}
  For the last inequality, we chose a sufficiently large $\alpha > 0$ and used Proposition \ref{carleson-cubes} for the last inequality and the fact that $\Lip_h(C \rho \epsilon^{r+1} \tau^m \ell(S)) \leq \Lip_h(\psi)$ if we choose $\lambda$ large enough.  Here, $C$ is the constant inside the $\Lip$ from Lemma \ref{carleson-cubes}.
\end{proof}


\section{Coarse differentiation for maps into $p$-convex spaces}
We continue with a generalization of a theorem of \cite{LN}.  The arguments in this section are similar to the arguments in the following section, but are conceptually easier.  We can treat this section as a warm up.  Recall that a Banach space $(Y,\|\cdot\|)$ is uniformly convex if there exists a $p \geq 2$ and $K \geq 1$ such that
\begin{align}
  \forall x,y \in Y, \qquad \frac{\|x-y\|^p + \|y-z\|^p}{2} \geq \left\| \frac{x-z}{2} \right\|^p + \left\| \frac{x+z-2y}{2K} \right\|^p. \label{UC-eqn}
\end{align}

Given $Q \in \Delta$ and $\eta > 0$, define
\begin{align*}
  \qd_h^{UC}(Q,\eta) := \frac{1}{\eta a_0\ell(Q)} \inf \sup \{\|h(x) - T(x) - v\| : x \in \eta B_Q \}
\end{align*}
where the infimum is taken over all homomorphisms $T : G \to Y$ and $v \in Y$.  In this section, we will prove the following theorem.

\begin{theorem} \label{UC-UAAP}
  There exist $\alpha,\zeta,\lambda > 0$ depending only on $G$ and $Y$ so that if $\epsilon \in (0,1/2)$, $m \in \N$, $h : G \to Y$ is $\psi$-LLD, and $S \in \Delta$ so that
  \begin{align*}
    \ell(S) \geq \lambda \epsilon^{-(r+1)(3p+1)} \tau^{-m} \psi,
  \end{align*}
  then
  \begin{align*}
    \sum_{k=0}^m \sum_{Q \in \Delta_k(S)} \left\{|Q| : \qd_h^{UC}(Q,\zeta \epsilon^r) > \epsilon \Lip_h(\psi) \right\} \leq \epsilon^{-\alpha} |S|.
  \end{align*}
\end{theorem}

Given a function $h : \R \to Y$, if we plug in $h(u)$, $h\left( \frac{u+v}{2} \right)$, and $h(v)$ for $x,y,z$ respectively, we have that \eqref{UC-eqn} gives that
\begin{align}
  \frac{|v-u|^p}{2^p} \partial_h^{(p)}(u,v) &= \frac{1}{2} \left( \left\| h(u) - h\left( \frac{u+v}{2} \right) \right\|^p + \left\| h\left( \frac{u+v}{2} \right) - h(v) \right\|^p \right) - \left\| \frac{h(u)-h(v)}{2} \right\|^p \notag \\
  &\geq \frac{1}{K^p} \left\| \frac{h(u)+h(v)}{2} - h\left( \frac{u+v}{2} \right) \right\|^p. \label{partial-midpoint}
\end{align}
We define
\begin{align*}
  \Theta_h(x,y) = \frac{1}{|x-y|} \left\| \frac{h(x) + h(y)}{2} - h\left( \frac{x+y}{2} \right) \right\|.
\end{align*}
From \eqref{partial-midpoint}, we get that
\begin{align*}
  \Theta_h(x,y)^p \leq \frac{K^p}{2^p} \partial_h^{(p)}(x,y).
\end{align*}
We can then similarly define analogues to $\alpha_h^{(p)}$.  For $[a,b] \subset \R$, we set
\begin{align*}
  \beta_h^{(p)}([a,b];\epsilon) &= (1-\epsilon)^2 (b-a)^{-2} \underset{\begin{smallmatrix} {a \leq x < y \leq b,} \\ {y - x > \epsilon(b-a)} \end{smallmatrix}}{\iint} \Theta_h(x,y)^p dy~dx,
\end{align*}
and for $Q \in \Delta$ we set
\begin{align*}
  \beta_h^{(p)}(Q;\epsilon) &= \ell(Q)^{N-1} \int_{S^{n-1}} \int_{z_Q \cdot (G \circleddash v)} \chi_{\{x \cdot \R v \cap 2B_Q \neq \emptyset\}} \beta_h^{(p)}(x \cdot \R v \cap 6B_Q;\epsilon) ~dx ~d\mu(v).
\end{align*}
We also have that
\begin{align}
  \beta_h^{(p)} \leq \frac{K^p}{2^p}\alpha_h^{(p)}. \label{UC-alpha-beta}
\end{align}

We prove the following analogue of Lemma \ref{M-partial-move}.

\begin{lemma} \label{UC-theta-move}
  Let $h : I \to (Y,\|\cdot\|)$ be $\psi$-LLD for some (possibly infinite) interval $I \subseteq \R$, and let $c,d \in I$ so that $\frac{d-c}{4} \geq \psi$ and
  \begin{align}
    \frac{\Theta_h(c,d)}{4 \Lip_h(\psi)} (d-c) =: \rho \geq \psi. \label{UC-large-separation}
  \end{align}
  If we choose $s,t \in I$ so that $|s - c| \leq \rho$ and $|t- d| \leq \rho$, then $\Theta_h(s,t) > \frac{1}{3} \Theta_h(c,d)$.
\end{lemma}

\begin{proof}
  As $d-c \geq 4\psi$, we get that $\Theta_h(c,d) \leq \Lip_h(\psi)$ and so $\rho \leq \frac{1}{4} (d-c)$.  This further gives us that $t-s \leq \frac{3}{2}(d-c)$.  The proof is a direct computation:
  \begin{align*}
    \Theta_h(s,t) &= \frac{1}{t-s} \left\| \frac{h(s) + h(t)}{2} - h\left( \frac{s+t}{2} \right) \right\| \\
    &\geq \frac{1}{t-s} \left( \left\| \frac{h(c) + h(d)}{2} - h\left( \frac{c+d}{2} \right) \right\| - 2 \Lip_h(\psi) \rho \right) \\
    &= \frac{1}{t-s} \left( (d-c) \Theta_h(c,d) - 2 \Lip_h(\psi) \rho \right) \\
    &\overset{\eqref{UC-large-separation}}{\geq} \frac{1}{2} \frac{d-c}{t-s} \Theta_h(c,d) \\
    &\geq \frac{1}{3} \Theta_h(c,d).
  \end{align*}
\end{proof}

By reading the proof of Lemma 2.1 of \cite{LN} (referring to equation (17) in particular), we get the following lemma

\begin{lemma}\label{UC-energy}
  Fix $p\in [2,\infty)$. Suppose that $(Y,\|\cdot\|_Y)$ is a Banach space satisfying the uniform convexity condition \eqref{UC-eqn}. Fix $a,b\in \R$ with $a<b$ and $h:[a,b]\to Y$. Then
  \begin{multline}\label{eq:lower E}
    \sum_{k=0}^{m-1} 2^{-k} \max_{I \in \D^k([a,b])} \Theta_h(a(I),b(I))^p \ge \frac{1}{(4K)^p}\max_{k\in \{0,\ldots,2^m\}}
    \frac{\left\|h\left(a+\frac{k}{2^m}(b-a)\right)-L_h^{a,b}\left(a+\frac{k}{2^m}(b-a)\right)\right\|_Y^p}
    {(b-a)^p},
  \end{multline}
  where $K\in (0,\infty)$ is the constant appearing in $\eqref{UC-eqn}$ and $L_h^{a,b}:[a,b]\to Y$ is the linear interpolation of the values of $h$ on the endpoints of the interval $[a,b]$, \textit{i.e.},
  \begin{equation}\label{eq:def linear interpolation}
    \forall t\in \R,\quad L_h^{a,b}(t) := \frac{t-a}{b-a}h(b)+\frac{b-t}{b-a}h(a).
  \end{equation}
\end{lemma}

We can prove a similar bound using $\beta_h^{(p)}(I;\epsilon)$.  The proof is the same as that of Lemma \ref{M-alpha-geodesic}.  The next few lemmas will also follow from superficial alterations of the proofs of their analogues in Section \ref{M-coarse-diff}.  Thus, we will only provide the full proof for this next lemma as an example, and refer the reader to the previous proofs for the others.

\begin{lemma} \label{UC-beta-affine}
  Suppose $Y$ is a Banach space satisfying the uniform convexity condition $\eqref{UC-eqn}$.  Let $h : [a,b] \to Y$ be $\psi$-LLD.  If $\frac{1}{512K} \epsilon^2 (b-a) \geq \psi$ and $\beta_h^{(p)}\left([a,b]; \frac{\epsilon}{16}\right) \leq \left(\frac{\epsilon}{200K}\right)^{3p} \Lip_h(\psi)^p$ then
  \begin{align*}
    \sup_{t \in [a,b]} \left\|h(t) - L_h^{a,b}(t)\right\|_Y \leq \epsilon (b-a) \Lip_h(\psi).
  \end{align*}
\end{lemma}

\begin{proof}
  Let $m = \lceil\log \frac{1}{\epsilon}\rceil + 2$.  It suffices to prove
  \begin{align*}
    \max_{k\in \{0,\ldots,2^m\}} \left\|h\left(a + \frac{k}{2^m}(b-a)\right) - L_h^{a,b}\left(a + \frac{k}{2^m}(b-a)\right)\right\|_Y \leq \frac{\epsilon}{4}(b-a) \Lip_h(\psi).
  \end{align*}
  Indeed, let $t \in [a,b]$.  Then there exists $k \in \{0,...,2^m\}$ so that $|t - a - k2^{-m}(b-a)| \leq 2^{-m}(b-a)$.  As $h$ is $\psi$-LLD and $b-a \geq 2^{-m} (b-a) \geq \psi$, $L_h^{a,b}$ is $\Lip_h(\psi)$-Lipschitz, and we get
  \begin{align*}
    \|h(t) - L_h^{a,b}(t)\| &\leq \left\|h\left(a + \frac{k}{2^m}(b-a)\right) - L_h^{a,b}\left(a + \frac{k}{2^m}(b-a)\right)\right\| + \left\|h(t) - h\left(a + \frac{k}{2^m}(b-a)\right)\right\| \\
    &\qquad + \left\|L_h^{a,b}\left(a + \frac{k}{2^m}(b-a)\right) - L_h^{a,b}(t)\right\| \\
    &\leq \left(\frac{\epsilon}{4} + 2^{-m} + 2^{-m}\right) (b-a) \Lip_h(\psi) < \epsilon (b-a) \Lip_h(\psi).
  \end{align*}
  If
  \begin{align*}
    \max_{k \in \{0,...,m-1\}} \max_{I \in \D^k([a,b])} \Theta_h(a(I),b(I)) \leq \frac{\epsilon}{32K} \Lip_h(\psi),
  \end{align*}
  then \eqref{eq:lower E} give our result.  Thus, we may assume that there is a subinterval $I = [u,v] \in \bigcup_{k=0}^{m-1} \D^k([a,b])$ where
  \begin{align*}
    \Theta_h(u,v) \geq \frac{\epsilon}{32K} \Lip_h(\psi).
  \end{align*}
  Let $u',v' \in \R$ be such that $|u - u'| < \frac{1}{128K} \epsilon |I|$ and $|v - v'| < \frac{1}{128K} \epsilon |I|$.  Then as $h$ is $\psi$-LLD and
  \begin{align*}
    \frac{\Theta_h(u,v)}{4\Lip_h(\psi)} (v-u) \geq \frac{1}{128} \epsilon |I| \geq \frac{1}{64} \epsilon 2^{-m} (b-a) \geq \psi,
  \end{align*}
  we get by Lemma \ref{UC-theta-move} that
  \begin{align*}
    \Theta_h(u',v') \geq \frac{1}{96K} \epsilon \Lip_h(\psi).
  \end{align*}
  Thus,
  \begin{align*}
    \beta_h^{(p)}\left([a,b]; \frac{\epsilon}{16} \right) &\geq \left(1- \frac{\epsilon}{16} \right)^2 (b-a)^{-2} \int_{v-\epsilon |I|/128K}^v \int_u^{u+ \epsilon |I|/128K} \Theta_h(x,y)^p dy ~dx \\
    &> \frac{1}{2} \left(\frac{\epsilon |I|}{128K(b-a)}\right)^2 \left(\frac{\epsilon}{96K} \Lip_h(\psi)\right)^p \\
    &\geq \frac{1}{(200K)^{2+p}} 2^{-2m-4} \epsilon^{p+2} \Lip_h(\psi)^p.
  \end{align*}
  Remembering that $2^{-m-2} \geq \epsilon$, we get a contradiction of the hypothesis.
\end{proof}

We now return to maps from Carnot groups to uniformly convex Banach spaces.  We prove that we can bootstrap the averaging bound of $\beta_h^{(p)}(Q)$ to a supremum bound.

\begin{lemma} \label{UC-sup-bound}
  Let $Q \in \Delta$ and $h : G \to Y$ be $\psi$-LLD.  There exists $\alpha_0,\gamma,\lambda_0 > 0$ depending only on $G$ and $Y$ so that if $\beta_h^{(p)}\left(Q; \frac{\epsilon}{\gamma}\right) \leq \epsilon^{\alpha_0} \Lip_h(\psi)^p$ and $\ell(Q) > \lambda_0 \epsilon^{3p+1} \psi$, then 
  \begin{align*}
    \sup \left\{ \beta^{(p)}_h\left(x \cdot \R v \cap 3B_Q; \frac{\epsilon}{16} \right) : v \in S^{n-1}, x \in G \circleddash v, x \cdot \R v \cap B_Q \neq \emptyset \right\} \leq \left( \frac{\epsilon}{200K} \right)^{3p} \Lip_h(\psi)^p.
  \end{align*}
\end{lemma}

\begin{proof}
  The proof is largely identical to that of Lemma \ref{M-sup-bound} with mostly superficial modifications ({\it e.g.} $\Theta_h(x,y)^p$ for $\partial_h^{(p)}(x,y)$, $\beta_h^{(p)}$ for $\alpha_h^{(p)}$, Lemma \ref{C-theta-move} for Lemma \ref{M-partial-move}).
\end{proof}

\begin{lemma} \label{UC-affine-close}
  Let $\epsilon \in (0,1)$, $v \in S^{n-1}$, $\rho > 0$, $\chi \geq 1$, $x \in G$, and $h : G \to Y$ be $\psi$-LLD.  Suppose $\frac{\epsilon}{24} \rho \geq \psi$.  There exists a constant $C \in (0,1)$ depending on $\chi$ such that if all $g \in \Bcc(x,\rho)$ satisfies
  \begin{align}
    \sup_{t \in [-2\rho,2\rho]} \left\| h(ge^{tv}) - L_{h|g \cdot \R v}^{-2\rho,2\rho}(ge^{tv}) \right\| \leq \frac{1}{4} C\epsilon^{r+1} \rho\Lip_h(\psi), \label{almost-affine}
  \end{align}
  then there exists $w_v \in Y$ such that $\|w_v\| \leq \Lip_h(\psi)$ and for all $g \in \Bcc(x, C\chi\epsilon^r \rho)$ we have that,
  \begin{align*}
    \sup_{s,t \in [-3 C\epsilon^r \rho,3 C\epsilon^r \rho]} \left\| h(ge^{sv}) - h(ge^{tv}) - (s-t)w_v \right\| \leq C\epsilon^{r+1} \rho\Lip_h(\psi).
  \end{align*}
\end{lemma}

\begin{proof}
  Note that \eqref{almost-affine} implies that, for each element $g \in \Bcc(x,\rho)$, we get
  \begin{align*}
    \sup_{s,t \in [-2\rho,2\rho]} \left\| h(ge^{sv}) - h(ge^{tv}) - \frac{s-t}{4} (h(ge^{2\rho v}) - h(ge^{-2\rho v})) \right\| \leq \frac{1}{2} C\chi\epsilon^{r+1} \rho\Lip_f(\psi).
  \end{align*}
  One then sees that the proof is largely identical to that of Lemma \ref{M-geodesic-close} with superficial modifications.
\end{proof}

\begin{lemma} \label{UC-final}
  Let $\alpha_0,\gamma > 0$ be the constants from Lemma \ref{UC-sup-bound} and $C \in (0,1)$ be the constant from Lemma \ref{UC-affine-close} associated to $\chi = M_G$.  There exist $\lambda > 0$ depending only on $G$ so that if $\epsilon \in (0,1/2)$, $\ell(Q) > \lambda \epsilon^{-(r+1)(3p+1)}\psi$, and
  \begin{align*}
    \beta_h^{(p)}\left(Q; \frac{1}{8\gamma} C a_0 \left( \frac{\epsilon}{M_G} \right)^{r+1} \right) \leq \left( \frac{1}{8} C a_0 \left( \frac{\epsilon}{M_G} \right)^{r+1} \right)^{\alpha_0} \Lip_h(\psi)^p 
  \end{align*}
  then there exists an affine function $A : G \to Y$ so that
  \begin{align*}
    \sup_{x \in C(\epsilon/M_G)^r B_Q} \frac{\| h(x) - A(x)\|}{C (\epsilon/M_G)^r a_0\ell(Q)} \leq \epsilon \Lip_h(\psi).
  \end{align*}
\end{lemma}

\begin{proof}
  By translation, we may suppose $z_Q = 0$ and $h(0) = 0$.  Applying Lemma \ref{UC-sup-bound} to the hypothesis and setting $\lambda$ large enough, we get that for all $v \in S^{n-1}$ and $x \in G \circleddash v$ where $x \cdot \R v \cap B_Q \neq \emptyset$ that
  \begin{align*}
    \beta^{(p)}_h\left(x \cdot \R v \cap 3B_Q; \frac{1}{128} C a_0 \left( \frac{\epsilon}{M_G} \right)^{r+1} \right) \leq \left( \frac{C a_0 (\epsilon/M_G)^{r+1}}{1600K} \right)^{3p} \Lip_h(\psi)^p.
  \end{align*}
  Then, by Lemma \ref{UC-beta-affine} and setting $\lambda$ large enough, we get that for all $v \in S^{n-1}$ and $x \in G \circleddash v$ where $x \cdot \R v \cap B_Q \neq \emptyset$ and $x \cdot [a,b]v = x \cdot \R v \cap 3B_Q$ that
  \begin{align}
    \sup_{t \in [a,b]} \left\|h(xe^{tv}) - L_{h|x \cdot \R v}^{a,b}(t)\right\| \leq \frac{1}{8} C \left( \frac{\epsilon}{M_G} \right)^{r+1} a_0 \ell(Q) \Lip_h(\psi). \label{eq:UC-almost-affine}
  \end{align}
  Notice that if $y \in B_Q$, then $y \cdot [-2a_0\ell(Q),2a_0\ell(Q)] v \subset x \cdot \R v \cap 3B_Q$.  Thus, for all $v \in S^{n-1}$ and $x \in B_Q$, we get
  \begin{align*}
    \sup_{t \in [-2a_0\ell(Q),2a_0\ell(Q)]} \left\|h(xe^{tv}) - L_{h|x \cdot \R v}^{-2a_0\ell(Q),2a_0\ell(Q)}(t)\right\| \leq \frac{1}{4} C \left( \frac{\epsilon}{M_G} \right)^{r+1} a_0 \ell(Q) \Lip_h(\psi),
  \end{align*}
  where we lost a factor of $\frac{1}{2}$ from using \eqref{eq:UC-almost-affine} twice.  We get from Lemma \ref{UC-affine-close} that for each $v \in S^{n-1}$, there exists a $w(v) \in Y$ so that for all $x \in M_G C (\epsilon/M_G)^r B_Q$, we have
  \begin{multline}
    \sup_{s,t \in [-3C (\epsilon/M_G)^r a_0\ell(Q),3C (\epsilon/M_G)^r a_0\ell(Q)]} \left\|h(xe^{tv}) - h(xe^{sv}) - (t-s) w(v)\right\| \\
    \leq C \left( \frac{\epsilon}{M_G} \right)^{r+1} a_0 \ell(Q) \Lip_h(\psi). \label{UC-group-inv}
  \end{multline}
  Take an orthonormal basis $\{v_i\}_{i=1}^n$ of $\V_1$ and define the homomorphism by the action on the generators
  \begin{align}
    T : G &\to Y \notag \\
    e^{t v_i} &\mapsto t w(v_i). \label{UC-affine-def}
  \end{align}
  Take an arbitrary element $g = e^{C_1 v_{i(1)}} e^{C_2 v_{i(2)}} \cdots e^{C_j v_{i(j)}} \in C (\epsilon/M_G)^r B_Q$ written using the Chow theorem.  Let
  \begin{align*}
    g^{(\ell)} = e^{C_1 v_{i(1)}} e^{C_2 v_{i(2)}} \cdots e^{C_\ell v_{i(\ell)}}, \qquad \forall \ell \leq j.
  \end{align*}
  By the assumptions we made in Section \ref{preliminaries}, we have that $j \leq M_G$ and $|C_i| \leq C(\epsilon/M_G)^r a_0\ell(Q)$.  As $g^{(\ell)} \in M_G C (\epsilon/M_G)^r B_Q$, we get that
  \begin{align*}
    \| h(g^{(j)}) &- T(g^{(j)})\| \\
    &= \| h(g^{(j)}) - h(g^{(j-1)}) + h(g^{(j-1)}) - T(g^{(j)}) + T(g^{(j-1)}) - T(g^{(j-1)}) \| \\
    &\leq \| h(g^{(j-1)}) - T(g^{(j-1)}) \| + \| h(g^{(j)}) - h(g^{(j-1)}) - (T(g^{(j)}) - T(g^{(j-1)})) \| \\
    &\overset{\eqref{UC-affine-def}}{=} \| h(g^{(j-1)}) - T(g^{(j-1)}) \| + \| h(g^{(j-1)}e^{C_j v_{i(j)}}) - h(g^{(j-1)}) - C_j w(v_{i(j)}) \| \\
    &\leq \sum_{\ell=2}^j \| h(g^{(\ell-1)}e^{C_\ell v_{i(\ell)}}) - h(g^{(\ell-1)}) - C_\ell w(v_{i(\ell)}) \| \\
    &\overset{\eqref{UC-group-inv}}{\leq} M_G C \left( \frac{\epsilon}{M_G} \right)^{r+1} a_0\ell(Q) \Lip_h(\psi).
  \end{align*}
  Thus, we have that for all $x \in C (\epsilon/M_G)^r B_Q$ that
  \begin{align*}
    \| h(x) - T(x)\| \leq \epsilon \cdot C(\epsilon/M_G)^r a_0 \ell(Q) \Lip_h(\psi).
  \end{align*}
\end{proof}

\begin{proof}[Proof of Theorem \ref{UC-UAAP}]
  Lemma \ref{UC-final} shows that there exists some $\zeta,\rho,\lambda_0,\alpha_1 > 0$ so that if $\epsilon \in (0,1/2)$, $\qd_h^{UC}(Q,\zeta \epsilon^r) > \epsilon \Lip_h(\psi)$ and $\ell(Q) \geq \lambda_0\epsilon^{(r+1)(3p+1)} \psi$, then $\beta_h^{(p)}\left(Q;\rho \epsilon^{r+1} \right) > \epsilon^{\alpha_1} \Lip_h(\psi)^p$.  Thus, if $\ell(S) \geq \lambda\epsilon^{-(r+1)(3p+1)} \tau^{-m} \psi$ for some sufficiently large $\lambda$, then
  \begin{align*}
    \sum_{k=0}^m \sum_{Q \in \Delta_k(S)} \left\{|Q| : \qd_h^{UC}(Q,\zeta \epsilon^r) > \epsilon \right\} &\leq \epsilon^{-\alpha_1} \Lip_h(\psi)^{-p} \sum_{k=0}^m \sum_{Q \in \Delta_k(S)} \beta_h^{(p)}\left(Q; \rho\epsilon^{r+1} \right)|Q| \\
    &\overset{\eqref{UC-alpha-beta}}{\leq} \epsilon^{-\alpha_1} \frac{K^p}{2^p} \Lip_h(\psi)^{-p} \sum_{k=0}^m \sum_{Q \in \Delta_k(S)} \alpha_h^{(p)}\left(Q; \rho\epsilon^{r+1} \right)|Q| \\
    &\leq \epsilon^{-\alpha} |S|,
  \end{align*}
  where we used Proposition \ref{carleson-cubes} as in the proof of Theorem \ref{M-UAAP} for the last inequality and chose some sufficiently large $\alpha > 0$.
\end{proof}


\section{Coarse differentiation for maps into Carnot groups}
In this section, we will study Lipschitz at large distances maps from $G$, a Carnot group of step $r$ endowed with a CC-metric, to $H$, a Carnot group of step $s$ endowed with a specific homogeneous metric $d_H$ to be described.  The Lie algebras of $G$ and $H$ will be $\g$ and $\h$, respectively.  Here, the horizontal layer of $\g$ has dimension $n$ and the horizontal layer of $\h$ has dimension $m$.  Let $\tilde{\pi} : H \to H$ be the function that maps elements of $H$ to their corresponding horizontal elements (\textit{i.e.} $\tilde{\pi}(g_1,...,g_s) = (g_1,0,...,0)$).  We will suppose that that there exist constants $K > 0$ and $p > 1$ so that
\begin{align}
  \frac{d_H(x,y)^p + d_H(y,z)^p}{2} \geq \left( \frac{d_H(x,z)}{2} \right)^p + \frac{1}{K}\left(\left| \frac{x_1 + z_1}{2} - y_1\right|^p + NH(x^{-1} z)^p\right), \label{C-eqn}
\end{align}
where, for an element $h \in H$, we let $h_1$ be shorthand for the image of $h$ under the 1-Lipschitz projection map $\pi : H \to \R^m$.  Here, we defined the map
\begin{align*}
  NH : H &\to \R^+ \\
  g &\mapsto d_H(\tilde{\pi}(g),g)
\end{align*}
to measure how nonhorizontal an element of $H$ is.  In the next section, we will show that such a metric always exists for graded nilpotent Lie groups.  We will not suppose that $d_H$ satisfies the triangle inequality.  Thus, we can only suppose that there exists some $C_Q \geq 1$ so that $d_H$ satisfies the quasi-triangle inequality:
\begin{align*}
  d_H(x,z) \leq C_Q(d_H(x,y) + d_H(y,z)).
\end{align*}
The quasi-triangle inequality constant $C_Q$ will depend only on the group $H$.

Given $Q \in \Delta$ and $\eta > 0$, define
\begin{align*}
  \qd_h^C(Q,\eta) := \frac{1}{\eta a_0\ell(Q)} \inf \sup \{d_H(h(x), g \cdot T(x)) : x \in \eta B_Q \}
\end{align*}
where the infimum is taken over all Lipschitz homomorphisms $T : G \to H$ and $g \in H$.  In this section, we will prove the following theorem.

\begin{theorem} \label{C-UAAP}
  There exist constants $\alpha,\beta,\zeta > 0$ depending only on $G$ and $H$ so that if $\epsilon \in (0,1/2)$, $m \in \N$, $h : G \to H$ is $\psi$-LLD, and $S \in \Delta$ so that
  \begin{align*}
    \ell(S) \geq e^{\epsilon^{-\alpha}} \tau^{-m} \psi,
  \end{align*}
  then
  \begin{align}
    \sum_{k=0}^m \sum_{Q \in \Delta_k(S)} \left\{|Q| : \qd_h^C(Q,\zeta \epsilon^\beta) > \epsilon \Lip_h(\psi) \right\} \leq e^{\epsilon^{-\alpha}} |S|. \label{C-UAAP-eqn}
  \end{align}
\end{theorem}

Given a function $h : \R \to H$, if we plug in $h(u)$, $h\left( \frac{u+v}{2} \right)$, and $h(v)$ for $x,y,z$ respectively, we have that \eqref{C-eqn} gives that
\begin{align}
  \frac{|v-u|^p}{2^p} \partial_h^{(p)}(u,v) &= \frac{1}{2} \left( d_H\left(h(u), h\left( \frac{u+v}{2} \right)\right)^p + d_H\left( h\left( \frac{u+v}{2} \right), h(v)\right)^p \right) - \frac{d_H(h(u),h(v))^p}{2} \notag \\
  &\geq \frac{1}{K} \left( \left| \frac{h(u)_1 + h(v)_1}{2} - h\left( \frac{u+v}{2} \right)_1 \right|^p + NH(h(u)^{-1}h(w))^p\right). \label{C-partial-midpoint}
\end{align}
For $p > 1$, we define
\begin{align}
  \Theta_h(x,y)^p = \frac{1}{|y-x|^p}\left( \left| \frac{h(x)_1 + h(y)_1}{2} - h\left( \frac{x+y}{2} \right)_1 \right|^p + NH(h(x)^{-1}h(y))^p\right). \label{C-theta-defn}
\end{align}
From \eqref{C-partial-midpoint}, we get that
\begin{align*}
  \Theta_h(x,y)^p \leq \frac{K}{2^p} \partial_h^{(p)}(x,y).
\end{align*}
We can then similarly define analogues to $\alpha_h^{(p)}$ as in the previous section.  For $[a,b] \subset \R$, we define
\begin{align*}
  \beta_h^{(p)}([a,b];\epsilon) &= (1-\epsilon)^2 (b-a)^{-2} \underset{\begin{smallmatrix} {a \leq x < y \leq b,} \\ {y - x > \epsilon(b-a)} \end{smallmatrix}}{\iint} \Theta_h(x,y)^p dy~dx,
\end{align*}
and for $Q \in \Delta$, we define
\begin{align*}
  \beta_h^{(p)}(Q;\epsilon) &= \ell(Q)^{N-1} \int_{S^{n-1}} \int_{z_Q \cdot (G \circleddash v)} \chi_{\{x \cdot \R v \cap 2B_Q \neq \emptyset\}} \beta_h^{(p)}(x \cdot \R v \cap 6B_Q;\epsilon) ~dx ~d\mu(v).
\end{align*}
We also have that
\begin{align}
  \beta_h^{(p)} \leq \frac{K}{2^p}\alpha_h^{(p)}. \label{C-alpha-beta}
\end{align}

We first prove some preliminary lemmas.  We first show that, like before, $\Theta_h^p$ does not change much under perturbations.

\begin{lemma} \label{C-theta-move}
  Let $h : I \to H$ be $\psi$-LLD for some (possible infinite) interval $I \subseteq \R$ and let $c,d \in I$ so that $\frac{d-c}{4} \geq \psi$.  There exist constants $\zeta \in (0,1/4)$ and $C \in (0,1)$ so that if
  \begin{align}
    \zeta \left(\frac{\Theta_h(c,d)}{\Lip_h(\psi)}\right)^s (d-c) =: \rho \geq \psi. \label{C-large-separation}
  \end{align}
  then for any $u,v \in I$ so that $|u - c| \leq \rho$ and $|v- d| \leq \rho$ we get $\Theta_h(u,v) > C\Theta_h(c,d)$.
\end{lemma}

\begin{proof}
  As $d-c \geq 4\psi$, we get by looking at the definition of $\Theta_h$ that
  \begin{align*}
    \frac{\Theta_h(c,d)}{\Lip_h(\psi)} \leq 1.
  \end{align*}
  Thus, $\rho \leq \frac{1}{4} (d-c)$ and so we get that $v-u \leq \frac{3}{2}(d-c)$.  We also have that
  \begin{align}
    \left(\frac{\Theta_h(c,d)}{\Lip_h(\psi)}\right)^s \leq \frac{\Theta_h(c,d)}{\Lip_h(\psi)}. \label{C-power-s-bnd}
  \end{align}
  By \eqref{C-theta-defn}, we must have that
  \begin{align*}
    \max \left\{\left| \frac{h(c)_1 + h(d)_1}{2} - h\left( \frac{c+d}{2} \right)_1 \right|^p, NH(h(c)^{-1}h(d))^p \right\} \geq \frac{(d-c)^p}{2} \Theta_h(c,d)^p.
  \end{align*}
  Suppose
  \begin{align}
    \frac{1}{(d-c)^p} \left| \frac{h(c)_1 + h(d)_1}{2} - h\left( \frac{c+d}{2} \right)_1 \right|^p \geq \frac{1}{2} \Theta_h(c,d)^p. \label{theta-1}
  \end{align}
  Then by a direct computation, we get
  \begin{align}
    \Theta_h(u,v) &\geq \frac{1}{v-u} \left| \frac{h(u)_1 + h(v)_1}{2} - h\left( \frac{u+v}{2} \right)_1 \right| \notag \\
    &\geq \frac{1}{v-u} \left( \left| \frac{h(c)_1 + h(d)_1}{2} - h\left( \frac{c+d}{2} \right)_1 \right| - 2 \rho \Lip_h(\psi) \right) \notag \\
    &\overset{\eqref{C-large-separation} \wedge \eqref{C-power-s-bnd}}{\geq} \frac{1}{v-u} \left( \left| \frac{h(c)_1 + h(d)_1}{2} - h\left( \frac{c+d}{2} \right)_1 \right| - \frac{1}{4} \Theta_h(c,d) (d-c) \right) \label{C-NH-theta-use} \\
    &\overset{\eqref{theta-1}}{\geq} \frac{d-c}{2(v-u)} \frac{1}{d-c} \left| \frac{h(c)_1 + h(d)_1}{2} - h\left( \frac{c+d}{2} \right)_1 \right| \notag \\
    &\geq \frac{1}{3(d-c)} \left| \frac{h(c)_1 + h(d)_1}{2} - h\left( \frac{c+d}{2} \right)_1 \right|. \label{C-perturb1}
  \end{align}
  In \eqref{C-NH-theta-use}, we used the fact that $\zeta < 1/4$.  Thus, we get by \eqref{theta-1} that
  \begin{align*}
    \Theta_h(u,v)^p \geq \frac{1}{2 \cdot 3^p} \Theta_h(c,d)^p. 
  \end{align*}
  Now suppose
  \begin{align}
    \frac{1}{(d-c)^p} NH(h(c)^{-1}h(d))^p = \frac{1}{(d-c)^p} d_H(h(d),h(c)\tilde{\pi}(h(c)^{-1}h(d)))^p \geq \frac{1}{2} \Theta_h(c,d)^p. \label{theta-2}
  \end{align}
  As we've only assumed that $d_H$ satisfies a quasi-triangle inequality with constant $C_Q$, we have
  \begin{align*}
    d_H(h&(d),h(c) \tilde{\pi}(h(c)^{-1}h(d))) \\
    &\leq C_Q\left(d_H(h(d),h(v)) + d_H(h(v),h(c) \tilde{\pi}(h(c)^{-1} h(d))) \right) \\
    &\leq C_Qd_H(h(d),h(v)) + C_Q^2 d_H(h(v),h(u) \tilde{\pi}(h(c)^{-1} h(d))) \\
    &\qquad + C_Q^2 d_H(h(u)\tilde{\pi}(h(c)^{-1}h(d)), h(c)\tilde{\pi}(h(c)^{-1}h(d))) \\
    &\leq C_Qd_H(h(d),h(v)) + C_Q^3 d_H(h(v),h(u) \tilde{\pi}(h(u)^{-1} h(v))) \\
    &\qquad + C_Q^3 d_H(\tilde{\pi}(h(u)^{-1} h(v)), \tilde{\pi}(h(c)^{-1} h(d))) \\
    &\qquad + C_Q^2 d_H(h(u)\tilde{\pi}(h(c)^{-1}h(d)), h(c)\tilde{\pi}(h(c)^{-1}h(d))).
  \end{align*}
  Thus, we get
  \begin{align}
    NH(h(u)^{-1}h(v)) &= d_H(h(v),h(u) \tilde{\pi}(h(u)^{-1} h(v))) \notag \\
    &\geq C_Q^{-3} NH(h(c)^{-1}h(d)) - C_Q^{-2} d_H(h(d),h(v)) - d_H(\tilde{\pi}(h(u)^{-1} h(v)), \tilde{\pi}(h(c)^{-1} h(d))) \notag \\
    &\qquad -  C_Q^{-1} d_H(h(u)\tilde{\pi}(h(c)^{-1}h(d)), h(c)\tilde{\pi}(h(c)^{-1}h(d))). \label{C-move-NH}
  \end{align}
  We bound all the negative terms on the right hand side individually.  As $\tilde{\pi}$ maps to horizontal elements, we have that $\tilde{\pi}(h(u)^{-1} h(v)) = e^{\lambda w_0}$ and $\tilde{\pi}(h(c)^{-1} h(d)) = e^{\lambda w_1}$ for some $\lambda \leq \frac{3}{2}(d-c) \Lip_h(\psi)$ and $w_0,w_1 \in \V_1$ so that $w_0 \in S^{m-1}$.  Note that as $|u-c|,|v-d| \leq \rho$, we get 
  \begin{align}
     d_H(h(c),h(u)) \leq \rho \Lip_h(\psi), \notag \\
     d_H(h(d),h(v)) \leq \rho \Lip_h(\psi). \label{C-move-1}
  \end{align}
  As $\pi : (H,d_H) \to \R^m$ is 1-Lipschitz, this gives us that
  \begin{align*}
    |h(c)_1 - h(u)_1| \leq \rho\Lip_h(\psi), \\
    |h(d)_1 - h(v)_1| \leq \rho\Lip_h(\psi).
  \end{align*}
  As $w_0$ and $w_1$ depend only on the first coordinates of $h(u),h(v),h(c)$, and $h(d)$, we get by simple Euclidean geometry that
  \begin{align*}
    |w_0 - w_1| \leq \frac{2\rho\Lip_h(\psi)}{\lambda}.
  \end{align*}
  Remembering that $\lambda \leq \frac{3}{2} (d-c)\Lip_h(\psi)$, we get by Lemma \ref{close-parallel-2} that there exists some $C_1 > 0$ so that
  \begin{multline}
    d_H(\tilde{\pi}(h(u)^{-1} h(v)), \tilde{\pi}(h(c)^{-1} h(d))) \leq C_1 \left( \frac{2\rho\Lip_h(\psi)}{\lambda}\right)^{1/s}\lambda \leq C_1 \rho^{1/s} (d-c)^{1-\frac{1}{s}} \Lip_h(\psi) \\
    \overset{\eqref{C-large-separation}}{\leq} C_1 \zeta^{1/s} (d-c) \Theta_h(c,d). \label{C-move-2}
  \end{multline}
  Using the fact that $d_H(h(u),h(c)) \leq \rho \Lip_h(\psi)$, we can use Lemma \ref{close-parallel-2} again to show that there exists some constants $C_2 > 0$ so that
  \begin{align}
    d_H(h(u)\tilde{\pi}(h(c)^{-1}h(d)), h(c)\tilde{\pi}(h(c)^{-1}h(d))) &\leq C_2 \rho^{1/s} (d-c)^{1-\frac{1}{s}} \Lip_h(\psi) \notag \\
    &\overset{\eqref{C-large-separation}}{\leq} C_2 \zeta^{1/s} (d-c) \Theta_h(c,d). \label{C-move-3}
  \end{align}
  By using \eqref{C-power-s-bnd} and \eqref{C-move-1} for the first term and choosing $\zeta$ sufficiently small, we can force
  \begin{align}
    &C_Q^{-2} d_H(h(d),h(v)) + d_H(\tilde{\pi}(h(u)^{-1} h(v)), \tilde{\pi}(h(c)^{-1} h(d))) \notag \\
    &\qquad + C_Q^{-1} d_H(h(u)\tilde{\pi}(h(c)^{-1}h(d)), h(c)\tilde{\pi}(h(c)^{-1}h(d))) \notag \\
    &\overset{\eqref{C-large-separation} \wedge \eqref{C-power-s-bnd} \wedge \eqref{C-move-1} \wedge \eqref{C-move-2} \wedge \eqref{C-move-3}}\leq \frac{C_Q^{-3}}{4} (d-c) \Theta_h(c,d). \label{C-move-4}
  \end{align}
  We can now bound
  \begin{align*}
    \Theta(u,v) &\geq \frac{1}{v-u} NH(h(u)^{-1}h(v)) \\
    &\overset{\eqref{C-move-NH} \wedge \eqref{C-move-4}}{\geq} \frac{1}{v-u} \left( C_Q^{-3} NH(h(c)^{-1}h(d)) - \frac{C_Q^{-3}}{4} (d-c) \Theta_h(c,d) \right) \\
    &\overset{\eqref{theta-2}}{\geq} \frac{1}{2C_Q^3(v-u)} NH(h(c)^{-1}h(d)) \\
    &\overset{\eqref{theta-2}}{\geq} \frac{d-c}{4C_Q^3(v-u)} \Theta_h(c,d) \\
    &\geq \frac{1}{6C_Q^3} \Theta_h(c,d).
  \end{align*}
\end{proof}

Our next result is to show that if an element has a small non-horizontal amount, then the nonhorizontal coordinates are also small.

\begin{lemma} \label{horizontal-comparison}
  Let $\rho \in (0,1)$ and $\eta > 0$.  There exists a constant $C > 0$ depending only on the structure of the group and the group norm such that if $NH(g) \leq \rho \eta$ and $|g_1| < \eta$, then we have
  \begin{align*}
    \sup_{i \in \{2,...,n\}} |g_i| \leq C\rho^2 \eta^i.
  \end{align*}
\end{lemma}

\begin{proof}
  If we set $P_k$ to be the BCH polynomial of the product $\pi(g)^{-1} g$ for the $k$th level, we have that there exists some $C_0 > 0$ so that
  \begin{align}
    N_H(\tilde{\pi}(g)^{-1} g) = N_H(0,g_2 + P_2,...,g_s + P_s) \geq C_0 \max_{2 \leq i \leq s} |g_i + P_i|^{1/i}. \label{NH-max}
  \end{align}
  This follows from the equivalence of homogeneous norms.  We prove that there exists some sequence of numbers $\lambda_2,...,\lambda_r$ depending only on the group structure and norm such that if $|g_i| \geq \lambda_i \rho^2 \eta^i$ for some $i \in \{2,...,n\}$, then the right hand side of \eqref{NH-max} is greater than $\rho \eta$, which contradicts our assumption.  This is easily seen to be true if $|g_2|^{1/2} \geq C_0^{-2} \rho \eta$ as $P_2 = 0$.  Now assume we have shown that there exist $\lambda_2,...,\lambda_{k-1}$ such that for all $i \in \{2,...,k-1\}$ we have
  \begin{align*}
    |g_i| \leq \lambda_i \rho^2 \eta^i.
  \end{align*}
  Suppose $|g_k| \geq \lambda_k \rho^2 \eta^k$ where $\lambda_k$ is some constant to be determined later.  Note then that the largest possible value for a nested Lie bracket in $P_k$ (modulo already chosen multiplicative coefficients that were dependent only on the group structure) is $|[g_1,[g_1,...,[g_1,g_2]...]]| \leq \lambda_2 \rho^2 \eta^k$ as the presence of any $g_i$ with higher indices would only add to the power of $\rho$.  By the BCH formula there then exists some $C_1 > 0$ depending only on $H$ and $\lambda_2,...,\lambda_{k-1}$ such that
  \begin{align*}
    |P_k| \leq C_1 \rho^2 \eta^k.
  \end{align*}
  Thus, we have
  \begin{align*}
    |g_k + P_k| \geq |g_k| - |P_k| \geq \lambda_k \rho^2 \eta^k - C_1 \rho^2 \eta^k = (\lambda_k-C_1) \rho^2 \eta^k.
  \end{align*}
  We then get a contradiction if we choose $\lambda_k$ large enough.
\end{proof}

Next, we show that, given a product of two elements $gh$, if $g$ has a small non-horizontal amount and $g_1$ and $h_1$ are close, then $g$ is close to the midpoint of the line segment from the origin to $\tilde{\pi}(gh)$.

\begin{lemma} \label{collinear-bound}
  Let $\rho \in (0,1)$ and $\eta > 0$.  There exists a constant $C > 0$ depending only on the group structure such that if $g,h \in H$ such that
  \begin{align*}
    NH(g) \leq \rho \eta, \quad |g_1| \leq \eta, \quad \left| g_1 - h_1 \right| &\leq \rho \eta,
  \end{align*}
  then
  \begin{align*}
    d_H\left(g, \delta_{1/2}\left(\tilde{\pi}(gh)\right)\right) \leq C\rho^{1/s} \eta.
  \end{align*}
\end{lemma}

\begin{proof}
  By the equivalence of norms, we have that there exists a constant $C_0 > 0$ such that
  \begin{align*}
    d_H\left(g,\delta_{1/2}\left(\tilde{\pi}(gh)\right) \right) &= N_H\left( \left( \frac{-g_1-h_1}{2},0,...,0 \right) (g_1,...,g_s) \right) \\
    &= N_H\left( \left( \frac{g_1 - h_1}{2} , g_2 + P_2,..., g_s + P_s \right)\right) \\
    &\leq C_0 \left[ \left| \frac{g_1 - h_1}{2} \right| \vee \max_{i \in \{2,...,s\}} \left(|g_i|^{1/i} + |P_i|^{1/i}\right) \right].
  \end{align*}
  Thus, it suffices to bound each term on the right hand side by some constant multiple of $\rho^{1/s} \eta$.  The term $|g_1 - h_1|$ already satisfies the conclusion.  By Lemma \ref{horizontal-comparison}, we have that there exists some constant $C_1 > 0$ such that for each $i \in \{2,...,s\}$ we have
  \begin{align*}
    |g_i| \leq C_1 \rho^2 \eta^i.
  \end{align*}
  Thus, it suffices to bound the $P_k$, which we will do so by bounding the individual nested Lie brackets that make up its summation (losing only another multiplicative constant).  Let $[x_1,[x_2,...[x_{j-1},x_j]...]]$ be a nested Lie bracket in $P_k$ where $x_\ell$ is either $g_{i(\ell)}$ or $h_1$.  Notice that $|h_1| \leq |h_1 - g_1| + |g_1| \leq (1 + \rho)\eta \leq 2\eta$.  Suppose that $[x_{j-1},x_j] = \pm [g_1,h_1]$.  Then as $[g_1,g_1] = 0$, we get
  \begin{align*}
    |[g_1,h_1]| = |[g_1,g_1-h_1]| \leq |g_1| |g_1 - h_1| \leq \rho \eta^2,
  \end{align*}
  and so we have that there exists some $C_2 > 0$ depending only on the group so that
  \begin{align*}
    |[x_1,[x_2,...[x_{j-1},x_j]...]]| \leq |[x_{j-1},x_j]| \prod_{i=1}^{j-2} |x_i| \leq C_2 \rho \eta^k.
  \end{align*}
  Otherwise, then either $x_{j-1} \in \{g_2,...,g_s\}$ or $x_j \in \{g_2,...,g_s\}$ and so there exists some $C_3 > 0$ depending only on the group so that
  \begin{align*}
    |[x_1,[x_2,...[x_{j-1},x_j]...]]| \leq \prod_{i=1}^j |x_i| \leq C_3 \rho^2 \eta^k.
  \end{align*}
\end{proof}

We can now prove that coarse differentiability of a Lipschitz map $\R \to H$ can be controlled by the $\Theta_h$.

\begin{lemma} \label{C-energy}
  There exists some constant $C > 0$ depending only on $H$ so that, for each $m \in \N \backslash \{0\}$ and map $h : [a,b] \to H$ that is $\psi$-LLD, if $2^{-m-3} (b-a) \geq \psi$ and
  \begin{align}
    \max_{k \in \{0,...,m\}} \max_{I \in \D^k([a,b])} \Theta_h(a(I),b(I))^p \leq \left( \frac{\zeta}{C} \right)^{ps^m} \Lip_h(\psi)^p.
  \end{align}
  then
  \begin{align*}
    \max_{k \in \{0,...,2^m\}} d_H\left( h\left( a + \frac{k}{2^m}(b-a) \right), L_h^{a,b}\left( a + \frac{k}{2^m}(b-a) \right) \right) \leq \zeta (b-a) \Lip_h(\psi).
  \end{align*}
  Here, for $a \leq u < v \leq b$, we define $L^{s,t}$ to be the one sided horizontal interpolant
  \begin{align*}
    L_h^{u,v}(t) := h(u) \delta_{\frac{t-u}{v-u}} \tilde{\pi}(h(v)^{-1}h(u)).
  \end{align*}
\end{lemma}

\begin{proof}
  It will be easier to prove that there exists some $C > 0$ so that if
  \begin{align}
    \max_{k \in \{0,...,m\}} \max_{I \in \D^k([a,b])} \Theta_h(x,y)^p \leq \zeta^p \Lip_h(\psi)^p, \label{step}
  \end{align}
  then
  \begin{align*}
    \max_{k \in \{0,...,2^m\}} d_H\left( h\left( a + \frac{k}{2^m}(b-a) \right), L_h^{a,b}\left( a + \frac{k}{2^m}(b-a) \right) \right) \leq C \zeta^{s^{-m}} (b-a) \Lip_h(\psi).
  \end{align*}
  We may suppose by translation that $[a,b] = [0,L]$ and let $L_h := L_h^{a,b}$ for convenience.  We proceed with induction.  For the first case when $m = 0$, \eqref{step} gives that
  \begin{align*}
    \left| \frac{h(0)_1 + h(L)_1}{2} - h\left( \frac{L}{2} \right)_1 \right|^p + NH(h(0)^{-1}h(L))^p \leq \zeta^p L^p \Lip_h(\psi)^p.
  \end{align*}
  Thus,
  \begin{align*}
    d_H(h(L),L_h(L)) = d(h(L),h(0)\tilde{\pi}(h(0)^{-1}h(L))) = NH(h(0)^{-1}h(L)) \leq \zeta L \Lip_h(\psi),
  \end{align*}
  and
  \begin{align*}
    \left| \frac{h(0)_1 + h(L)_1}{2} - h\left( \frac{L}{2} \right)_1 \right| \leq \zeta L \Lip_h(\psi).
  \end{align*}
  By definition, we have that $h(0) = L_h(0)$.  Now suppose that we have shown for $m \in \N \cup \{0\}$ that
  \begin{align}
    \max_{k \in \{0,...,2^m\}} d_H\left( h\left( \frac{k}{2^m}L  \right), L_h\left( \frac{k}{2^m}L \right) \right) &\leq C\zeta^{s^{-m}} L \Lip_h(\psi), \label{left-bound} \\
    \max_{k \in \{0,...,2^m-1\}} \left| \frac{h\left( \frac{k}{2^m} L \right)_1 + h\left( \frac{k+1}{2^m} L \right)_1}{2} - h\left( \frac{k+1/2}{2^m} L \right)_1 \right| &\leq \zeta 2^{-m} L \Lip_h(\psi), \label{induct-hyp-midpoint} \\
    \max_{k \in \{0,...,2^m-1\}} NH\left( h\left( \frac{k}{2^m} L \right)^{-1} h\left( \frac{k+1}{2^m}L \right) \right) &\leq \zeta 2^{-m}L \Lip_h(\psi). \label{induct-hyp-NH}
  \end{align}
  Applying \eqref{step} again gives us
  \begin{align*}
    \max_{k \in \{0,...,2^{m+1}-1\}} \Theta_h\left( \frac{k}{2^{m+1}}L, \frac{k+1}{2^{m+1}}L \right)^p \leq \zeta^p \Lip_h(\psi)^p,
  \end{align*}
  which, taking into account the definition of $\Theta_h(x,y)^p$, also gives
  \begin{align}
    \max_{k \in \{0,...,2^{m+1}-1\}} NH\left( h\left( \frac{k}{2^{m+1}} L \right)^{-1} h\left( \frac{k+1}{2^{m+1}} L \right) \right) &\leq \zeta 2^{-m-1} L \Lip_h(\psi), \label{induct-NH} \\
    \max_{k \in \{0,...,2^{m+1}-1\}} \left| \frac{h\left( \frac{k}{2^{m+1}} L \right)_1 + h\left( \frac{k+1}{2^{m+1}} L \right)_1}{2} - h\left( \frac{k + 1/2}{2^{m+1}} L \right)_1 \right| &\leq \zeta 2^{-m-1} L \Lip_h(\psi). \label{induct-midpoint}
  \end{align}
  Applying Lemma \ref{collinear-bound} to \eqref{induct-hyp-midpoint} and \eqref{induct-NH} with $\eta = 2^{-m+1} L \Lip_h(\psi)$, $\rho = \zeta$, $g = h\left(\frac{k}{2^m}L\right)^{-1} h\left( \frac{k+1/2}{2^m}L\right)$, and $h = h\left(\frac{k+1/2}{2^m} L\right)^{-1}h\left( \frac{k+1}{2^m} L \right)$, we get that there exists some $C_0 > 0$ so that
  \begin{align*}
    \max_{k \in \{0,...,2^m-1\}} d_H\left( h\left( \frac{k+1/2}{2^m} L \right), L_h^{k2^{-m}L,(k+1)2^{-m}L}\left( \frac{k+1/2}{2^m} L \right) \right) \leq C_0\zeta^{1/s} 2^{-m+1} L \Lip_h(\psi).
  \end{align*}
  We now show that $L_h^{k2^{-m}L,(k+1)2^{-m}L}((k+1)2^{-m}L)$ is close to $L_h((k+1)2^{-m}L)$ in order to set up Lemma \ref{horizontal-lines-bound}.  By \eqref{induct-hyp-NH} we have that for each $k \in \{0,...,2^m-1\}$
  \begin{align*}
    d_H\left( h\left( \frac{k+1}{2^m} L\right),L_h^{k2^{-m}L,(k+1)2^{-m}L}\left( \frac{k+1}{2^m} L \right) \right) &= NH\left(h\left( \frac{k}{2^m} L \right)^{-1} h\left( \frac{k+1}{2^m} L \right) \right) \\
    &\leq \zeta 2^{-m} L \Lip_h(\psi).
  \end{align*}
  Thus, the quasi-triangle inequality of $H$ gives
  \begin{align}
    d_H&\left( L_h^{k2^{-m}L,(k+1)2^{-m}L}\left( \frac{k+1}{2^m} L \right), L_h\left( \frac{k+1}{2^m} L \right) \right) \notag \\
    &\quad \leq C_Qd_H\left( L_h^{k2^{-m}L,(k+1)2^{-m}L}\left( \frac{k+1}{2^m} L \right),h\left( \frac{k+1}{2^m} L \right) \right) + C_Qd_H\left( h\left( \frac{k+1}{2^m} L \right), L_h\left( \frac{k+1}{2^m} L \right) \right) \notag \\
    &\quad\leq C_Q\left(\zeta 2^{-m} + C\zeta^{s^{-m}}\right) L \Lip_h(\psi). \label{right-bound}
  \end{align}
  Now taking \eqref{left-bound} and \eqref{right-bound} in consideration, applying Lemma \ref{horizontal-lines-bound} with $\rho= C_Q(\zeta + C2^m \zeta^{s^{-m}})$, $\eta = \Lip_h(\psi)$, and $[a,b] = [k2^{-m}L,(k+1)2^{-m}L]$ gives that there exists some $C_1 > 0$ so that
  \begin{align*}
    d_H\left( L_h^{k2^{-m}L,(k+1)2^{-m}L} \left( \frac{k+1/2}{2^m} L \right), L_h\left( \frac{k+1/2}{2^m} L \right) \right) \leq C_1 \left( \zeta + C2^m \zeta^{s^{-m}} \right)^{1/s} 2^{-m}L \Lip_h(\psi).
  \end{align*}
  Finally, applying to quasi-triangle inequality again, we have for all $k \in \{0,...,2^m-1\}$ that
  \begin{align*}
    d_H\left(h \left( \frac{k+1/2}{2^m} L\right), L_h\left( \frac{k+1/2}{2^m} L \right) \right) &\leq C_Q d_H\left( L_h^{k2^{-m}L,(k+1)2^{-m}K} \left( \frac{k+1/2}{2^m} L \right), L_h\left( \frac{k+1/2}{2^m} L \right) \right) \\
    &\qquad +C_Qd_H\left( h\left( \frac{k+1/2}{2^m} L \right), L_h^{k2^{-m}L,(k+1)2^{-m+1}L)}\left( \frac{k+1/2}{2^m} L \right) \right) \\
    &\quad\leq C_Q\left( C_1 \left( \zeta + C 2^m \zeta^{s^{-m}} \right)^{1/s} 2^{-m} + C_0\zeta^{1/s} 2^{-m+1}\right) L \Lip_h(\psi) \\
    &\quad\leq C\zeta^{s^{-m}} L \Lip_h(\psi).
  \end{align*}
  In the last inequality, we require that $C$ be large enough.
\end{proof}

As in the previous sections, we will show that this bound can be translated into a statment concerning bounds of $\beta_h^{(p)}$.  As before, this proof is similar to Lemma \ref{M-alpha-geodesic}, and we will give the proof as an example for the next few lemmas, whose proofs we will not give.  The only tricky point is that $d_H$ is guaranteed only to be a semimetric.  But this will not change the proofs, just the constants.

\begin{lemma} \label{C-beta-horizontal}
  Suppose $H$ is a Carnot group that satisfies the convexity condition.  Let $h : [a,b] \to H$ be $\psi$-LLD and $\epsilon \in (0,1/2)$.  There exists a constant $\alpha_0 > 0$ depending only on $H$ so that if $\beta_h^{(p)}\left([a,b]; \frac{\epsilon}{8} \right) \leq e^{-\epsilon^{-\alpha_0}} \Lip_h(\psi)^p$ and $b-a \geq e^{\epsilon^{-\alpha_0}} \psi$, then
  \begin{align*}
    \max_{t \in [a,b]} d_H(h(t),L_h^{a,b}(t)) \leq \epsilon (b-a) \Lip_h(\psi).
  \end{align*}
\end{lemma}

\begin{proof}
  Let $m = \lceil \log \frac{C_Q}{\epsilon} \rceil + 2$.  It suffices to prove
  \begin{align*}
    \max_{k\in \{0,\ldots,2^m\}} d_H\left(h\left(a + \frac{k}{2^m}(b-a)\right), L_h^{a,b}\left(a + \frac{k}{2^m}(b-a)\right)\right) \leq \frac{\epsilon}{4C_Q} (b-a) \Lip_h(\psi).
  \end{align*}
  where $C_Q > 0$ is the quasitriangle constant.  Indeed, let $t \in [a,b]$.  Then there exists $k \in \{0,...,2^m\}$ so that $|t - a + k2^{-m}(b-a)| \leq 2^{-m}(b-a)$.  As $h$ is $\psi$-LDD and $b-a \geq \psi$, $L_h^{a,b}$ is $\Lip_h(\psi)$-Lipschitz.  Thus, as $2^{-m}(b-a) \geq \frac{1}{8} \epsilon(b-a) \geq \psi$, we get
  \begin{align*}
    d_H(h(t), L_h^{a,b}(t)) &\leq C_Qd_H\left(h(t), h\left(a + \frac{k}{2^m}(b-a)\right)\right) + C_Qd_H\left(L_h^{a,b}\left(a + \frac{k}{2^m}(b-a) \right), L_h^{a,b}(t)\right) \\
    &\qquad + C_Qd_H\left(h\left(a + \frac{k}{2^m}(b-a)\right), L_h^{a,b}\left(a + \frac{k}{2^m}(b-a)\right)\right) \\
    &\leq C_Q\left(2^{-m} + 2^{-m} + \frac{\epsilon}{4C_Q} \right) (b-a) \Lip_h(\psi) \\
    &< \epsilon(b-a)\Lip_h(\psi).
  \end{align*}

  If
  \begin{align*}
    \max_{k \in \{0,...,m\}} \max_{I \in \D^k([a,b])} \Theta_h(a(I),b(I))^p \leq \left( \frac{\epsilon}{4C} \right)^{ps^m} \Lip_h(\psi)^p,
  \end{align*}
  for some sufficiently large $C > 0$ depending only on $H$, then Lemma \ref{C-energy} gives the result.  Thus, we may assume that there is a subinterval $I = [u,v] \in \bigcup_{k=0}^m \D^k([a,b])$ such that
  \begin{align}
     \Theta_h(u,v)^p > \left( \frac{\epsilon}{4C} \right)^{ps^m} \Lip_h(\psi)^p. \label{e:large-theta-h}
  \end{align}
  Remembering that $m$ is approximately $\log \frac{1}{\epsilon}$, as $\epsilon < 1/2$, we get that there exists some constant $\kappa > 0$ so that we have
  \begin{align}
    \left( \frac{\epsilon}{4C} \right)^{s^m} \geq \epsilon^{\epsilon^{-\kappa}}. \label{gamma-m-simplify}
  \end{align}
  Let $\delta := \zeta \epsilon^{s\epsilon^{-\kappa}} 2^{-m}(b-a)$ and $s,t \in [a,b]$ so that $s \in [u,u + \delta]$ and $t \in [v - \delta,v]$ where $\zeta$ is the constant from Lemma \ref{C-theta-move}.  As $v-u \geq 2^{-m}(b-a)$, we then have that
  \begin{align*}
    \left( \frac{\Theta_h(u,v)}{\Lip_h(\psi)} \right)^s \zeta (v-u) \overset{\eqref{e:large-theta-h} \wedge \eqref{gamma-m-simplify}}{\geq} \zeta \epsilon^{s\epsilon^{-\kappa}} (v-u) \geq \delta.
  \end{align*}
  If we require $\alpha_0$ to be sufficiently large (while allowing us to define it even larger later), then $\zeta \epsilon^{s\epsilon^{-\kappa}} 2^{-m} \geq e^{-\epsilon^{-\alpha_0}}$.  As $e^{-\epsilon^{-\alpha_0}}(b-a) \geq \psi$, the same lemma gives that there exists some constant $C_0 > 0$ so that
  \begin{align}
    \Theta_h(s,t) \geq C_0 \Theta_h(u,v). \label{e:theta-preserved}
  \end{align}

  Thus,
  \begin{align*}
    \beta_h^{(p)}\left([a,b]; \frac{\epsilon}{8} \right) &\geq \left(1 - \frac{\epsilon}{8} \right)^2 (b-a)^{-2} \int_{v-\delta}^v \int_u^{u+ \delta} \Theta_h(x,y)^p dy ~dx \\
    &\overset{\eqref{e:large-theta-h} \wedge \eqref{e:theta-preserved}}{>} \frac{C_0^p}{2} \left( \frac{\delta}{b-a} \right)^2 \left( \frac{\epsilon}{4C} \right)^{ps^m} \Lip_h(\psi)^p  \\
    &\geq \frac{C_0^p}{2} \zeta^2 2^{-2m} \epsilon^{2s\epsilon^{-\kappa} + p\epsilon^{-\kappa}} \Lip_h(\psi)^p.
  \end{align*}
  Remembering that $C_Q2^{-m+2} \geq \epsilon$, we see that if we define $\alpha_0$ sufficiently large enough one last time, we get a contradiction of the hypothesis.
\end{proof}

We now return to maps from Carnot groups to Carnot groups.  We prove that we can bootstrap the averaging bound of $\beta_h^{(p)}(Q)$ to a supremum bound.

\begin{lemma} \label{C-sup-bound}
  Let $Q \in \Delta$, $\epsilon \in (0,1/2)$, $\alpha_0 > 0$ be the constant from Lemma \ref{C-beta-horizontal}, and $h : G \to H$ be $\psi$-LLD.  There exist constants $\gamma > 0$ and $\alpha_1 > 0$ so that if $\beta_h^{(p)}\left(Q; \frac{\epsilon}{\gamma}\right) \leq e^{-\epsilon^{-\alpha_1}} \Lip_h(\psi)^p$ and $\ell(Q) \geq e^{\epsilon^{-\alpha_1}} \psi$, then 
  \begin{align*}
    \sup \left\{ \beta^{(p)}_h\left(x \cdot \R v \cap 3B_Q; \frac{\epsilon}{8}\right) : v \in S^{n-1}, x \in G \circleddash v, x \cdot \R v \cap B_Q \neq \emptyset \right\} \leq e^{-\epsilon^{-\alpha_0}} \Lip_h(\psi)^p.
  \end{align*}
\end{lemma}

\begin{proof}
  The proof is largely identical to that of Lemma \ref{M-sup-bound} with superficial modifications ($\Theta_h(x,y)^p$ for $\partial_h^{(p)}(x,y)$, $\beta_h^{(p)}$ for $\alpha_h^{(p)}$, Lemma \ref{C-theta-move} for Lemma \ref{M-partial-move}).
\end{proof}

As before, we also need to prove that, if we have a subball on which $h$ acts approximately as a horizontal homomorphism on all horizontal lines, then one can choose a representative ``slope'' in each direction so that, on a controlled subball, $h$ is approximately a horizontal homomorphism of this slope for all horizontal lines of the corresponding direction.

\begin{lemma} \label{C-horizontal-close}
  Let $\epsilon \in (0,1/2)$, $v \in S^{n-1}$, $\rho > 0$, $\chi \geq 1$, $x \in G$, and $h : G \to H$ be $\psi$-LLD.  There exists a constant $\Gamma \in (0,1)$ depending on $\chi$ and $\zeta \in (0,1)$ depending on $G$ and $H$ such that if $\frac{\epsilon}{24} \rho \geq \psi$ and all $g \in \Bcc(x,\rho)$ satisfy
  \begin{align}
    \sup_{t \in [-2\rho,2\rho]} d_H(h(ge^{tv}), L_{h|g \cdot \R v}^{-2\rho,2\rho}(ge^{tv})) \leq \zeta \Gamma \epsilon^{r+1} \rho\Lip_h(\psi), \label{almost-horizontal}
  \end{align}
  then there exists a horizontal element $w_v \in H$ such that $N_H(w_v) \leq \Lip_h(\psi)$ and for all $g \in \Bcc(x, \Gamma \chi \epsilon^r\rho)$ we have that,
  \begin{align*}
    \sup_{s,t \in [-3\lambda\epsilon^r\rho, 3\lambda\epsilon^r\rho]} d_H(h(ge^{sv})^{-1} h(ge^{tv}), \delta_{t-s}w_v ) \leq \lambda \epsilon^{r+1} \rho \Lip_h(\psi).
  \end{align*}
\end{lemma}

\begin{proof}
  Note that \eqref{almost-horizontal} implies that, for each element $g \in \Bcc(x,\rho)$, we get
  \begin{align*}
    \sup_{s,t \in [-2\rho,2\rho]} d_H(h(ge^{sv})^{-1} h(ge^{tv}), \delta_{(t-s)/4} \tilde{\pi}(h(ge^{-2v})^{-1} h(ge^{2v}))) \leq \frac{\zeta}{2} \Gamma \epsilon^{r+1} \rho\Lip_h(\psi).
  \end{align*}
  One then sees that the proof is largely identical to that of Lemma \ref{M-geodesic-close} with mostly superficial modifications as in Lemma \ref{UC-affine-close}.  The $\zeta$ is used to account for the quasi-triangle inequality in the proof.
\end{proof}

Before we prove the final lemma relating all these bounds on $\beta_h^{(p)}$, we need to prove another lemma stating that, if we show that there is a subball on which along horizontal lines $h$ is close to being horizontal of the same ``slope'', then $h$ is close to a real homomorphism.  We will first need the following lemma, which is a group-independent variant of Chow's theorem.  We omit the proof as it essentially follows from applying Chow's theorem to the ``free'' $r$-step stratified nilpotent Lie group where there are no non-axiomatic Lie bracket relations.

\begin{lemma} \label{bracket-independent}
  Let $\alpha,n,r > 0$.  There exists constants $N = N(n,r) > 0$ and $\lambda = \lambda(\alpha,n,r) > 0$ so that, for any $k > 0$, if
  \begin{align*}
    P(X_1,...,X_n) = \sum_{j=1}^k \alpha_j [X_{i(j,1)},[X_{i(j,2)},...,[X_{i(j,k_j-1)},X_{i(j,k_j)}]...]],
  \end{align*}
  is an abstract polynomial of Lie brackets where for each $j$, $|\alpha_j| \leq \alpha$, $|k_j| \leq r$, and $i(j,\cdot) : \{1,...,k_j\} \to \{1,...,n\}$, then there exist two sequences $\{\lambda_\ell\}_{\ell=1}^N$ and $\{i_\ell\}_{\ell=1}^N$ so that, for any set of horizontal elements $v_1,...,v_n$ of any graded nilpotent Lie algebra $\g$,
  \begin{align*}
    e^{\lambda_1 v_{i_1}} e^{\lambda_2 v_{i_2}} \cdots e^{\lambda_N v_{i_N}} = e^{P(v_1,...,v_n) + Z}.
  \end{align*}
  where $Z \in \bigoplus_{j > r} \V_j(\g)$.  Furthermore, we have $|\lambda_i| \leq \lambda(\alpha,n,r)$ for all $i$.
\end{lemma}

Such sequences are clearly not necessarily unique.  We will call such paths bracket-independent as no Lie bracket relationships are used.

We will also need the following lemma, which easily follows from an inequality of \L{}ojasiewicz.

\begin{lemma} \label{loja}
  Let $\{f_i : \R^n \to \R\}_{i=1}^m$ be a finite collection of polynomials and $Z = \bigcap_i f^{-1}(0)$.  Then for any compact set $K$, there exist constants $C > 0$ and $\alpha > 0$ such that
  \begin{align*}
    d_{\R^n}(x,Z)^\alpha \leq C\max_i |f_i(x)|, \qquad \forall x \in K.
  \end{align*}
\end{lemma}

\begin{proof}
  The case when $m = 1$ is a special case of the original \L{}ojasiewicz inequality \cite{Lojasiewicz,malgrange}.  To get the general case, define $P(x) = \sum_{i=1}^m f(x)^2$ and apply the original \L{}ojasiewicz inequality to $P$.  It is clear that $P^{-1}(0) = Z$, and so there exists some $\alpha,C > 0$ so that $d_{\R^n}(x,Z)^\alpha \leq C|P(x)|$ for $x \in K$.  The lemma follows as $P(x)^{1/2} \leq \sqrt{m} \max_i |f_i(x)|$.
\end{proof}

We are now ready to prove the following lemma.

\begin{lemma} \label{C-homomorphism}
  Let $v_1,...,v_n$ be an orthonormal basis of $\V_1(\g)$ and $h : G \to T$ be $\psi$-LLD.  There exist constants $\nu,\beta_0,\Lambda,M > 0$ depending only on $G$ and $H$ so that, if there exist some $x \in G$, $\rho > 0$, and $w_1,...,w_n \in \V_1(\h)$ so that $|w_i| \leq \Lip_h(\psi)$, $\Lambda \rho \geq \psi$, and
  \begin{align}
    \sup\left\{ \sup_{s,t \in [-\Lambda\rho,\Lambda\rho]} d_H(h(ge^{sv_i})^{-1}h(ge^{tv_i}), e^{(t-s)w_i}) : g \in \Lambda M \Bcc(x,\rho) \right\} \leq (\nu \epsilon)^{\beta_0} \Lambda \rho \Lip_h(\psi), \label{close-horizontal-assumption}
  \end{align}
  for every $i \in \{1,...,n\}$, then there exists a Lipschitz homomorphism $T : G \to H$ so that
  \begin{align*}
    \sup_{z \in \Bcc(x,\rho)} \frac{d_H(h(z),h(x) \cdot T(z))}{\rho} \leq \epsilon \Lip_h(\psi).
  \end{align*}
\end{lemma}

\begin{proof}
  We will suppose without loss of generality that $x = 0$ and $h(0) = 0$.  As the $\{v_i\}$ vectors generate the group under the exponential map, it reduces to defining $T(e^{\lambda v_i})$ for each $i$.  As both $G$ and $H$ are simply connected Lie groups, we can instead pass to homomorphisms of their Lie algebras (which we still denote $T$ by abuse of notation).  If we define $T(v_i) = u_i \in \V_1(\h)$, then for this to pass to a homomorphism of Lie groups, we must have that $\{u_i\}$ satisfy all the same Lie bracket relations as $\{v_i\}$.  Note that we only need to verify the relations up to $s$, the rank of $\h$ as all Lie brackets of higher order vanish and are satisfied vacously.  Thus, we may as well suppose that $r \leq s$.  If $r < s$, then we also have to add the Lie bracket relations that any nested Lie bracket of depth larger than $r$ is 0.

  Each Lie bracket relation can be expressed as a sum of nested Lie bracket that must add up to zero, which we can think of as polynomials.  As the Lie algebras are graded and we are only dealing with horizontal vectors, we get that, for each equation, all the Lie brackets must be nested of the same depth, giving us that the polynomials are homogeneous.  Indeed, given any layer $\V_k(\g)$, the number of possible nested Lie bracket monomials composed of $\{v_1,...,v_n\}$ of degree $k$ is, with overcounting, at most $n^k$.  After choosing a maximal linearly independent subset of these monomials, one can then define the polynomials as just the linear dependence relations of the other monomials with respect to this basis.
  
  Let $\{P_i\}$ denote these polynomials of Lie brackets as defined over $\V_1(\h)^n$.  Thus, all feasible configuration of targets for $T(v_i)$ must evalute each $P_i$ to 0.  We first claim that there exists some $C_0 > 0$ (to be determined) depending on $G$ and $H$ so that
  \begin{align}
    |P_i(w_1,...,w_n)| \leq C_0 \nu (\nu\epsilon)^{\beta_0/s^{M-1}} \Lip_h(\psi)^{\deg(P_i)}, \qquad \forall i. \label{small-P}
  \end{align}
  Suppose that $\deg(P) = k$ and express $P(X_1,...,X_n)$ as the homogeneous nested Lie bracket polynomial
  \begin{align*}
    P(X_1,...,X_n) = \sum_{j=1}^\ell \alpha_j [X_{i(j,1)},[X_{i(j,2)},...,[X_{i(j,k-1)},X_{i(j,k)}]...]]
  \end{align*}
  where $i : \{1,...,\ell\} \times \{1,...,k\} \to \{1,...,n\}$ and $\alpha_j \in \R$.  Then using Lemma \ref{bracket-independent} and subsequently rescaling by $\rho$, we can find $M,\Lambda > 0$ and a minimal bracket-independent path $\{\lambda_j\}_{j=1}^M$, $\{i_j\}_{j=1}^M$ so that
  \begin{align*}
    e^{\lambda_1 X_{i_1}} e^{\lambda_2 X_{i_2}} \cdots e^{\lambda_M X_{i_M}} = e^{\rho^k P(X_1,...,X_n)}, \qquad \max_{i \in \{1,...,M\}} |\lambda_i| \leq \rho\Lambda,
  \end{align*}
  for any set of horizontal vectors $X_1,...,X_n$ in any graded Lie algebra of step at most $r$.  Note that as there exist only finitely many polynomials, all of which are determined by $G$, we can then choose a single $M \geq M_G$ to work for all $P_i$.  We will also take a single $\Lambda \geq 1$ to work for all $P_i$ as all the coefficients of the polynomials depend only on $G$.  These will be the $M$ and $\Lambda$ of the statement of the lemma.
  
  Let $S_\ell = e^{\lambda_1 v_1} \cdots e^{\lambda_\ell v_\ell}$ denote the partial path up to $\ell$ in $G$ with the understanding that $S_0 = 0$.  Note that $S_\ell \in \Lambda M \rho B_G$ for all $\ell \in \{1,...,M\}$.  We have
  \begin{align*}
    h(e^{\rho^k P(v_1,...,v_n)}) = \left[h(S_0)^{-1}h(S_1)\right] \cdots \left[h(S_{M-1})^{-1}h(S_M)\right].
  \end{align*}
  Define for $\ell \in \{1,...,n\}$
  \begin{align*}
    h^{(\ell)} = e^{\lambda_1 w_1} \cdots e^{\lambda_\ell w_\ell}.
  \end{align*}
  By the bracket-independent nature of $\{\lambda_j\}$ and $\{i_j\}$, we then that we have
  \begin{align}
    h^{(M)} = e^{\rho^k P(w_1,...,w_n) + Z} \label{PZ-return}
  \end{align}
  where $Z \in \bigoplus_{j > r} \V_j(\h)$.  We have by the hypothesis that
  \begin{align*}
    d_H(h(S_1),h^{(1)}) = d_H(h(e^0 e^{\lambda_1 v_1}),e^{\lambda_1 w_{i_1}}) \overset{\eqref{close-horizontal-assumption}}{\leq} (\nu \epsilon)^{\beta_0} \Lambda \rho \Lip_h(\psi).
  \end{align*}
  Now suppose by induction that for up to $\ell-1$ we have
  \begin{align*}
    d_H(h(S_{\ell-1}),h^{(\ell-1)}) \leq C^{1 + s^{-1} + ... + s^{-\ell+2}} (\nu \epsilon)^{\beta_0/s^{\ell-2}} \Lambda \rho \Lip_h(\psi).
  \end{align*}
  where $C > 0$ is the constant from Lemma \ref{close-parallel}.  Using the notation of the same lemma, we define $\lambda = \Lambda \rho \Lip_h(\psi)$, $u = \delta_{\lambda^{-1}}(h(S_{\ell-1})^{-1}h(S_\ell))$, $v = e^{\lambda^{-1}\lambda_\ell w_{i_\ell}}$, and $g = h(S_{\ell-1})$, $h = h^{(\ell-1)}$.  As $d_{cc}(S_{\ell-1},S_\ell) \leq \Lambda \rho$ and $\Lambda \rho \geq \psi$, we get that
  \begin{align*}
    d_H(u,0) = \lambda^{-1} d_H(h(S_{\ell-1}),h(S_\ell)) \leq \lambda^{-1} \Lambda \rho \Lip_h(\psi) \leq 1.
  \end{align*}
  We also have
  \begin{align*}
    d_H(u,v) = \lambda^{-1} d_H(h(S_{j-1})^{-1}h(S_j),e^{\lambda_\ell w_{i_\ell}}) \overset{\eqref{close-horizontal-assumption}}{\leq} C^{1+s^{-1} + ... + s^{-\ell+2}} (\nu\epsilon)^{\beta_0/s^{\ell-2}}.
  \end{align*}
  Lemma \ref{close-parallel} then gives
  \begin{align*}
    d_H(h(S_\ell),h^{(\ell)}) = d_H(h(S_{\ell-1}) \cdot \delta_\lambda(u),h^{(\ell-1)} \cdot \delta_\lambda(v)) \leq C^{1+s^{-1} + ... + s^{-\ell+1}} (\nu\epsilon)^{\beta_0/s^{\ell-1}} \Lambda \rho \Lip_h(\psi),
  \end{align*}
  and so we see that
  \begin{align}
    d_H(0,e^{\rho^k P(w_1,...,w_n) + Z}) \overset{\eqref{PZ-return}}{=} d_H(h(S_M),h^{(M)}) \leq C^{1+s^{-1} + ... + s^{-M+1}} (\nu \epsilon)^{\beta_0/s^{M-1}} \Lambda \rho \Lip_h(\psi). \label{P-return}
  \end{align}
  Note that there exist some constant $C_1 > 0$ so that $d_H(0,e^{\rho^k P(w_1,...,w_n)} + Z) \geq C_1 \rho |P(w_1,...,w_n)|^{1/k}$.  Combining this with \eqref{P-return} gives the needed inequality
  \begin{align*}
    |P(w_1,...,w_n)| \leq C_0 \nu (\nu\epsilon)^{\beta_0/s^{M-1}} \Lip_h(\psi)^k.
  \end{align*}
  Here, we've shed the $\deg(P_i)$ in the exponent to account for the $\nu$ term.

  Let $\{f_i\}_{i=1}^m$ be an orthonormal basis of the horizontal layer of $\h$.  If, $u = \sum_{i=1}^m \alpha_i f_i$ and $v = \sum_{i=1}^m \beta_i f_i$, we can compute the Lie bracket
  \begin{align*}
    [u,v] = \sum_{i < j} (\alpha_i \beta_j - \alpha_j \beta_i) [f_i,f_j].
  \end{align*}
  Note that this is a quadratic homogeneous vector value polynomial.  It easily follows that a nested Lie bracket of horizontal elements $[u_{i_1},[u_{i_2},...,[u_{i_{j-1}},u_{i_j}]...]]$ can be expressed as a homogeneous vector-valued polynomial of degree $k$.  Thus, each of the Lie bracket polynomials $P_i$ can be thought of as a vector of real valued polynomials.  From this, we get a collection of real valued polynomials $\{Q_i\}$ on $\R^{mn}$, thought of as $n$ copies of $\R^m$, such that any collection of $n$ horizontal vectors of $H$ are in the zero locus of $\{P_i\}$ if and only if its coordinates are in the zero locus of $\{Q_i\}$.

  As $P_i(w_1,...,w_n) = v_i$ such that $|v_i| \leq C_0\nu(\nu\epsilon)^{\beta_0/s^{M-1}} \Lip_h(\psi)^{\deg(P_i)}$ for all $i$, we get that
  \begin{align*}
    |Q_i(\tilde{w}_1,...,\tilde{w}_n)| \leq C_0 \nu (\nu\epsilon)^{\beta_0/s^{M-1}}, \qquad \forall i,
  \end{align*}
  where $\tilde{w}_i = \Lip_h(\psi)^{-1}w_i$.  Lemma \ref{loja} with $K = B(0,n)$ then gives a set of vectors $(\tilde{w}_1',...,\tilde{w}_n') \in \R^{mn}$ so that for all $i$ we have $Q_i(\tilde{w}_1',...,\tilde{w}_n') = 0$ and
  \begin{align}
    |\tilde{w}_i - \tilde{w}_i'| \leq C_2 \nu^{1/\alpha} (\nu\epsilon)^{\beta_0/(\alpha s^{M-1})}
  \end{align}
  for some $\alpha > 0$ depending on $G$ and $C_2 > 0$ depending on $C_0$.  Setting $w_i' = \Lip_h(\psi) \tilde{w}_i'$, we get that $Q_i(w_1,...,w_n) = 0$ and
  \begin{align}
    |w_i - w_i'| \leq C_2 \nu^{1/\alpha} (\nu\epsilon)^{\beta_0/(\alpha s^{M-1})} \Lip_h(\psi). \label{close-vectors}
  \end{align}
  Thus, the homomorphism defined via generators
  \begin{align*}
    T : G &\to H \\
        e^{\lambda v_i} &\mapsto e^{\lambda w_i'}
  \end{align*}
  is well defined.  That $T$ is Lipschitz follows from the fact that $\{e^{v_i}\}_i$ generate $G$ and $T$ is Lipschitz on this set of generators.  By \eqref{close-vectors} and Lemma \ref{close-parallel-2} and remembering that $|w_i| \leq \Lip_h(\psi)$, we get that there exist some $C_3$ depending on $C_2$ so that
  \begin{align*}
    \sup_{t \in [-\Lambda\rho,\Lambda\rho]} d_H(e^{\lambda w_i},e^{\lambda w_i'}) \leq C_3 \nu^{1/(\alpha s)} (\nu\epsilon)^{\beta_0/(\alpha s^M)} \Lambda \rho \Lip_h(\psi).
  \end{align*}
  This, together with \eqref{close-horizontal-assumption}, the quasitriangle inequality of $d_H$, and specifying $\nu$ to be sufficiently small and $\beta_0 = \alpha s^{2M}$, we get that
  \begin{align}
    \sup\left\{ \sup_{s,t \in [-\Lambda\rho,\Lambda\rho]} d_H(h(ge^{sv_i})^{-1}h(ge^{tv_i}), e^{(t-s)w_i'}) : g \in \Lambda M \Bcc(x,\rho) \right\} \leq (\nu\epsilon)^{s^M} \rho \Lip_h(\psi) \label{close-horizontal-assumption-2}
  \end{align}
  for every $i \in \{1,...,n\}$.
  
  Take $g \in \Bcc(x,\rho)$.  By Chow's theorem, we can write it as a product
  \begin{align*}
    g = e^{\lambda_1 v_{i_1}} \cdots e^{\lambda_{M_G} v_{i_{M_G}}}
  \end{align*}
  so that $|\lambda_1| \leq \rho$.  We can form the partial products $S_j = e^{\lambda_1 v_{i_1}} \cdots e^{\lambda_j v_{i_j}}$ as before.  Note then that $T(S_{j-1}^{-1} S_j) = e^{\lambda_j w_{i_j}'}$.  As $M_G \leq M$, we can repeat the process before that yielded \eqref{P-return} to get that there exists some $C_4 > 0$ so that
  \begin{align*}
    d_H(h(g),T(g)) \leq C_4 \nu \epsilon \rho \Lip_h(\psi).
  \end{align*}
  By specifying $\nu$ small enough, we finish the proof of the lemma.
\end{proof}

We are now ready to prove our final lemma, which resembles the final lemmas from the previous sections.

\begin{lemma} \label{C-final}
  Let $\alpha_1,\gamma > 0$ be the constants from Lemma \ref{C-sup-bound}, $\Lambda,M,\nu,\beta_0 > 0$ be the constants from Lemma \ref{C-homomorphism}, and $\lambda \in (0,1)$ be the constant from Lemma \ref{C-horizontal-close} associated to $M$.  Suppose $h : G \to H$ is $\psi$-LLD.  There exist $C > 0$ depending only on $G$ so that if $\epsilon \in (0,1/2)$, $\eta = C\epsilon^{\beta_0(r+1)s}$, $\ell(Q) \geq e^{\eta^{-\alpha_1}} \psi$ and
  \begin{align*}
    \beta_h^{(p)}\left(Q; \frac{\eta}{\gamma} \right) \leq e^{-\eta^{-\alpha_1}} \Lip_h(\psi)^p
  \end{align*}
  then there exists a Lipschitz homomorphism $T : G \to H$ and $g \in H$ so that for all $x \in \Lambda^{-1} \lambda(\nu\epsilon)^{\beta_0 r} B_Q$ we have
  \begin{align*}
    d_H(h(x), g \cdot T(x)) \leq \epsilon \cdot \Lambda^{-1} \lambda \left( \nu\epsilon \right)^{\beta_0 r} a_0 \ell(Q) \Lip_h(\psi).
  \end{align*}
\end{lemma}

\begin{proof}
  We will define $\eta$ as $\zeta \lambda^s a_0 \left( \nu \epsilon \right)^{\beta_0(r+1)s}$ and prove the statement for sufficiently small enough $\zeta$.  By translation, we may suppose $z_Q = 0$ and $h(0) = 0$.  Applying Lemma \ref{C-sup-bound} to the hypothesis, we get that
  \begin{align*}
    \sup \left\{ \beta^{(p)}_h\left(x \cdot \R v \cap 3B_Q; \frac{\eta}{8} \right) : v \in S^{n-1}, x \in G \circleddash v, x \cdot \R v \cap B_Q \neq \emptyset \right\} \leq e^{-\eta^{-\alpha_0}} \Lip_h(\psi)^p.
  \end{align*}
  Then, by Lemma \ref{C-beta-horizontal}, we get that for all $v \in S^{n-1}$ and $x \in G \circleddash v$ where $x \cdot \R v \cap B_Q \neq \emptyset$ and $x \cdot [a,b]v = x \cdot \R v \cap 3B_Q$, we have
  \begin{align}
    \sup_{t \in [a,b]} d_H(h(xe^{tv}), L_{h|x \cdot \R v}^{a,b}(t)) \leq \zeta \lambda^s \left( \nu \epsilon \right)^{\beta_0(r+1)s} a_0 \ell(Q) \Lip_h(\psi). \label{C-final-init}
  \end{align}
  Note that if $x \in B_Q$ then $x \cdot [-2a_0\ell(Q),2a_0\ell(Q)]v \subseteq x \cdot \R v \cap 3B_Q$.  Thus, we have for $x \in B_Q$ that
  \begin{align*}
    d_H(L_{h|x \cdot \R v}^{-2a_0\ell(Q),2a_0\ell(Q)}(-2a_0\ell(Q)), L_{h|x \cdot \R v}^{a,b}(-2a_0\ell(Q))) &= d_H(h(xe^{-2a_0\ell(Q)v}), L_{h|x \cdot \R v}^{a,b}(-2a_0\ell(Q))) \\
    &\overset{\eqref{C-final-init}}{\leq} \zeta \lambda^s \left( \nu \epsilon \right)^{\beta_0(r+1)s} a_0 \ell(Q) \Lip_h(\psi).
  \end{align*}
  Here, we've shifted the domain of the $L^{a,b}$ term so that the endpoint matches up.  The other endpoint can be bounded similarly.  Thus, as $L|_{x \cdot \R v}$ is $\Lip_h(\psi)$-Lipschitz, we get from Lemma \ref{horizontal-lines-bound} that there exist some $C_0 > 0$ so that
  \begin{align*}
    \sup_{t \in [-2a_0 \ell(Q),2a_0\ell(Q)]} d_H(L_{h|x \cdot \R v}^{-2a_0\ell(Q),2a_0\ell(Q)}(t), L_{h|x \cdot \R v}^{a,b}(t)) \leq C_0 \zeta^{1/s} \lambda \left( \nu \epsilon \right)^{\beta_0(r+1)} a_0 \ell(Q) \Lip_h(\psi).
  \end{align*}
  Here, we require $\zeta \leq 1$ so that $\zeta \lambda^s (\nu \epsilon)^{\beta_0(r+1)s} \leq 1$.  This, along with \eqref{C-final-init} and the quasi-triangle inequality gives that there exist some $C_1 > 0$ so that
  \begin{multline*}
    \sup \left\{ \sup_{t \in [-2a_0\ell(Q),2a_0\ell(Q)]} d_H(h(xe^{tv}), L_{h|x \cdot \R v}^{-2a_0\ell(Q),2a_0\ell(Q)}(t)) : v \in S^{n-1}, x \in B_Q \right\} \\
    \leq C_1\zeta^{1/s} \lambda \left( \nu \epsilon \right)^{\beta_0(r+1)} a_0 \ell(Q) \Lip_h(\psi).
  \end{multline*}
  Choosing $\zeta$ to be small enough, we get from Lemma \ref{C-horizontal-close} that for each $v \in S^{n-1}$, there exists a horizontal element $w(v) \in H$ so that for all $x \in M \lambda \left( \nu\epsilon \right)^{\beta_0 r} B_Q$ we have
  \begin{align}
    \sup_{s,t \in [-3\lambda (\nu\epsilon)^{\beta_0 r} a_0\ell(Q),3\lambda (\nu\epsilon)^{\beta_0 r} a_0\ell(Q)]} d_H(h(xe^{sv})^{-1}h(xe^{tv}), \delta_{t-s}w(v)) \leq \lambda \left( \nu \epsilon \right)^{\beta_0 (r+1)} a_0 \ell(Q) \Lip_h(\psi). \label{C-group-inv}
  \end{align}
  We finish the proof by applying Lemma \ref{C-homomorphism}.
\end{proof}

\begin{proof}[Proof of Theorem \ref{C-UAAP}]
  Lemma \ref{C-final} shows that there exists some $\alpha_2,\gamma,\zeta,\beta_0 > 0$ so that, setting $\eta = \epsilon^{\beta_0(r+1)s}$, if $\qd_h^{C}(Q,\zeta \epsilon^{\beta_0 r}) > \epsilon \Lip_h(\psi)$ and $\ell(Q) \geq e^{\eta^{-\alpha_2}} \psi$ then $\beta_h^{(p)}\left(Q; \eta/\gamma \right) > e^{-\eta^{-\alpha_2}} \Lip_h(\psi)$.  Thus, if $\ell(S) \geq \lambda e^{\eta^{-\alpha_2}} \tau^{-m} \psi$ for some sufficiently large $\lambda > 0$, then
  \begin{align*}
    \sum_{k=0}^m \sum_{Q \in \Delta_k(S)} \left\{|Q| : \qd_h^{UC}(Q,\zeta \epsilon^{\beta_0 r}) > \epsilon \Lip_h(\psi) \right\} &\leq e^{\eta^{-\alpha_2}} \Lip_h(\psi)^{-p} \sum_{k=0}^m \sum_{Q \in \Delta_k(S)} \beta_h^{(p)}(Q; \eta/\gamma)|Q| \\
    &\overset{\eqref{C-alpha-beta}}{\leq} C_0 e^{\eta^{-\alpha_2}} \Lip_h(\psi)^{-p} \sum_{k=0}^m \sum_{Q \in \Delta_k(S)} \alpha_h^{(p)}(Q; \eta/\gamma)|Q| \\
    &\leq C_1 e^{\eta^{-\alpha_2}} |S|,
  \end{align*}
  where we used Proposition \ref{carleson-cubes} for the last inequality as in the proof of Lemma \ref{M-UAAP}.  Setting $\beta = \beta_0r$ and setting $\alpha = \alpha_2 \beta_0 (r+1)s$, we get the statement of the theorem.
\end{proof}


\section{Uniform convexity of graded nilpotent Lie groups}

We now show that every graded nilpotent Lie group has a homogeneous metric (possibly a semimetric) that satisfies \eqref{C-eqn}.  We will let $H$ be a graded nilpotent Lie group of step $s$.  We also remind the reader that, for convenience, we have supposed that the constants in the Lie bracket structure of $H$ are all 1.  All the proofs that follow go through with superficial modifications in the general case and the results all differ by constants depending on the structure of the group.  Given a sequence of numbers $\lambda_2,...,\lambda_s > 0$, we can inductively construct a group norm via a sequence of group seminorms as follows.  Let
\begin{align*}
  N_2(x) = \left( |x_1|^4 + \lambda_2 |x_2|^2 \right)^{1/4}.
\end{align*}
Having defined $N_{k-1} : G \to \R$, we define
\begin{align*}
  N_k(x) = \left( N_{k-1}(x)^{2k!} + \lambda_k |x_k|^{2(k-1)!} \right)^{1/2k!}
\end{align*}
We take $N = N_s$ to be the group norm from this construction and $d(g,h) = N(g^{-1}h)$ to be its metric.  Note that $N_j \leq N_k$ if $j \leq k$.  We see that it is indeed a homogeneous norm by the anisotropic scaling of the norm.  As before, we let $\tilde{\pi} : H \to H$ denote the projection of $H$ onto its horizontal component and
\begin{align*}
  NH(x) = d(\tilde{\pi}(x),x) = N(\tilde{\pi}(x)^{-1} x)
\end{align*}
denote how non-horizontal $x$ is.

We will first need the following numerical lemma.

\begin{lemma} \label{num-bound-1}
  For each $C > 0$ and $k \geq 2$ there exists a $\lambda = \lambda(C,k)$ such that the following inequality holds for all $a,b \geq 0$:
  \begin{align*}
      \frac{\lambda}{2^k} (a + b)^{2(k-1)!} - C b^{2(k-1)!} \leq \frac{\lambda}{2} a^{2(k-1)!}
    \end{align*}
\end{lemma}

\begin{proof}
  Choose $\epsilon$ small enough so that $2^{-k} (1 + \epsilon)^{2(k-1)!} \leq 1/2$.  Then if $b \leq \epsilon a$ we have
  \begin{align*}
    \frac{\lambda}{2^k} (a + b)^{2(k-1)!} - Cb^{2(k-1)!} \leq \frac{\lambda}{2^k} (1+\epsilon)^{2(k-1)!} a^{2(k-1)!} \leq \frac{\lambda}{2} a^{2(k-1)!}.
  \end{align*}
  So far we haven't chosen $\lambda$.  Now let $b = \eta a > \epsilon a$.  Choose $\lambda$ so that
  \begin{align*}
    \frac{\lambda}{2^k} \left( \frac{1}{\epsilon} + 1 \right)^{2(k-1)!} \leq C.
  \end{align*}
  We have that
  \begin{align*}
    \frac{\lambda}{2^k} (a + b)^{2(k-1)!} - Cb^{2(k-1)!} = \left( \frac{\lambda}{2^k} (1+\eta)^{2(k-1)!} - C\eta^{2(k-1)!} \right) a^{2(k-1)!}.
  \end{align*}
  Dividing through by $\eta^{2(k-1)!}$, we get that the right hand side is
  \begin{align*}
    \left( \frac{\lambda}{2^k} \left( \frac{1}{\eta} + 1 \right)^{2(k-1)!} - C \right) a^{2(k-1)!} \leq \left( \frac{\lambda}{2^k} \left( \frac{1}{\epsilon} + 1 \right)^{2(k-1)!} - C\right) a^{2(k-1)!} \leq 0.
  \end{align*}
\end{proof}

We can now prove our uniform convexity result.

\begin{proposition} \label{carnot-convexity}
  Let $H$ be a graded nilpotent Lie group of step $s$.  Then there exists $\lambda_2,...,\lambda_s > 0$ and $C > 0$ depending only on $H$ such that, if we construct the group norm as above, then for all $g,h \in H$ we have
  \begin{align}
    \frac{1}{2} \left( N(g)^{2s!} + N(h)^{2s!} \right) \geq \left( \frac{N(gh)}{2} \right)^{2s!} + C 2^{-2s!} \left( | g_1 - h_1|^{2s!} + NH(gh)^{2s!} \right).
  \end{align}
\end{proposition}

\begin{proof}
  For brevity, we write
  \begin{align*}
    \alpha_{j,k} = |[g_1,h_1]|^{k!} + |g_1 - h_1|^{2k!} + \sum_{i=2}^j \left(|g_i|^{2k!/i} + |h_i|^{2k!/i} \right).
  \end{align*}
  We first prove that for all $k \in \{2,...,s\}$ we can find a sequence $\lambda_2,...,\lambda_k$ and a $C_0 > 0$ depending on the structure of $H$ such that
  \begin{align}
    \frac{N_k(g)^{2k!} + N_k(h)^{2k!}}{2} \geq \left( \frac{N_k(gh)}{2} \right)^{2k!} + \sum_{i=3}^k \left[ \sum_{j=1}^i \binom{i}{j} C_0^j \alpha_{i-1,i-1}^j |g_1 + h_1|^{2(i-1)! \cdot (i-j)}\right]^{k!/i!} + C_0\alpha_{k,k}, \label{induct-result}
  \end{align}
  where the summation doesn't appear for $k = 2$.  The proof will be by induction on the steps.  The constant $C_0$ may change from induction, but as there are only finitely many steps, we can just choose the minimal constant.  Thus, we will not have to be too careful with $C_0$.  For the base case, we have
  \begin{align*}
    \left( \frac{N_2(gh)}{2} \right)^4 = \left| \frac{g_1 + h_1}{2} \right|^4 + \frac{\lambda_2}{2^4} \left| g_2 + h_2 + \frac{1}{2} [g_1,h_1] \right|^2.
  \end{align*}
  By the parallelogram law of $|\cdot|$, we have that
  \begin{align*}
    \left| \frac{g_1+h_1}{2} \right|^4 &= \left( \frac{|g_1|^2 + |h_1|^2}{2} - \left| \frac{g_1 - h_1}{2} \right|^2 \right)^2 \\
    &= \left(\frac{|g_1|^2 + |h_1|^2}{2}\right)^2 - 2 \left( \frac{|g_1|^2 + |h_1|^2}{2} - \left| \frac{g_1 - h_1}{2} \right|^2 \right) \left| \frac{g_1 - h_1}{2} \right|^2 - \left| \frac{g_1 - h_1}{2} \right|^4 \\
    &= \left(\frac{|g_1|^2 + |h_1|^2}{2}\right)^2 - 2 \left| \frac{g_1 + h_1}{2} \right|^2 \left| \frac{g_1 - h_1}{2} \right|^2 - \left| \frac{g_1 - h_1}{2} \right|^4
  \end{align*}
  Note that
  \begin{align*}
    |[g_1,h_1]|^2 = \frac{1}{4} |[g_1 - h_1, g_1 + h_1]|^2 \leq \frac{1}{4} |g_1+h_1|^2 |g_1-h_1|^2.
  \end{align*}
  Thus,
  \begin{align}
    \left( \frac{N_2(gh)}{2} \right)^4  \leq \frac{|g_1|^4 + |h_1|^4}{2} - \left| \frac{g_1 - h_1}{2} \right|^4 - \frac{1}{2} |[g_1,h_1]|^2  + \frac{\lambda_2}{4} \left| \frac{g_2 + h_2}{2} + \frac{1}{4}[g_1,h_1]\right|^2 \label{1st-step}
  \end{align}
  Using the inequality $|a + b|^p - 2^{p-1} |b|^p \leq 2^{p-1} |a|^p$, we see that if we set $\lambda_2 = 1$, we have that
  \begin{align*}
    - 8 \left| \frac{[g_1,h_1]}{4} \right|^2 + \frac{1}{4} \left| \frac{g_2 + h_2}{2} + \frac{1}{4}[g_1,h_1]\right|^2 &\leq -\frac{15}{32} |[g_1,h_1]|^2  + \frac{1}{2} \left| \frac{g_2 + h_2}{2} \right|^2 \\
    &= -\frac{15}{32} |[g_1,h_1]|^2  + \frac{1}{2} \left(\frac{|g_2|^2 + |h_2|^2}{2} - \left| \frac{g_2 - h_2}{2} \right|^2 \right).
  \end{align*}
  Now combining with \eqref{1st-step} we have
  \begin{align*}
    \left( \frac{N_2(gh)}{2} \right)^4 \leq \frac{N_2(g)^4 + N_2(h)^4}{2} - \left| \frac{g_1 - h_1}{2} \right|^4 - \frac{15}{32} |[g_1,h_1]|^2 - \frac{|g_2|^2 + |h_2|^2}{4},
  \end{align*}
  which finishes the base case of the induction.  By the inductive hypothesis, we have that
  \begin{multline}
    \left( \frac{N_{k-1}(gh)}{2} \right)^{2(k-1)!} \leq \frac{N_{k-1}(g)^{2(k-1)!} + N_{k-1}(h)^{2(k-1)!}}{2} - C_0 \alpha_{k-1,k-1} \\
    - \sum_{i=3}^{k-1} \left[ \sum_{j=1}^i \binom{i}{j} C_0^j \alpha_{i-1,i-1}^j |g_1 + h_1|^{2(i-1)! \cdot (i-j)}\right]^{(k-1)!/i!}. \label{induct-hyp}
  \end{multline}
  By the construction of the norm, we have that
  \begin{align}
    \left( \frac{N_k(gh)}{2} \right)^{2k!} = \left( \frac{N_{k-1}(gh)}{2} \right)^{2k!} + \frac{\lambda_k}{2^k} \left| \frac{g_k + h_k}{2} + P_k \right|^{2(k-1)!} \label{norm-step}
  \end{align}
  where $P_k$ is a BCH polynomial of $gh$ at level $k$.  Suppose there exists a constant $C_1 > 0$ depending only on the group structure such that
  \begin{multline}
    \left( \frac{N_{k-1}(gh)}{2} \right)^{2k!} \leq \frac{N_{k-1}(g)^{2k!} + N_{k-1}(h)^{2k!}}{2} - C_1 \left( |P_k|^{2(k-1)!} + \alpha_{k-1,k} \right) \\
    - \sum_{i=3}^k \left[ \sum_{j=1}^i \binom{i}{j} C_1^j \alpha_{i-1,i-1}^j |g_1 + h_1|^{2(i-1)! \cdot (i-j)}\right]^{k!/i!}. \label{induct-1}
  \end{multline}
  We first finish the induction.  Combining \eqref{norm-step} and \eqref{induct-1} gives
  \begin{multline}
    \left( \frac{N_k(gh)}{2} \right)^{2k!} \leq \frac{N_{k-1}(g)^{2k!} + N_{k-1}(h)^{2k!}}{2} - C_1 \alpha_{k-1,k} + \frac{\lambda_k}{2^k} \left| \frac{g_k + h_k}{2} + P_k \right|^{2(k-1)!} \\
    - C_1|P_k|^{2(k-1)!} - \sum_{i=3}^k \left[ \sum_{j=1}^i \binom{i}{j} C_1^j \alpha_{i-1,i-1}^j |g_1 + h_1|^{2(i-1)! \cdot (i-j)}\right]^{k!/i!}. \label{induct-2}
  \end{multline}
  Having fixed $C_1$, an immediate consequence of Lemma \ref{num-bound-1} when $a = |g_k+h_1|/2$ and $b = |P_k|$ is that there exists a $\lambda_k > 0$ such that
  \begin{align}
    \frac{\lambda_k}{2^k} \left| \frac{g_k + h_k}{2} + P_k \right|^{2(k-1)!} - C_1|P_k|^{2(k-1)!} &\leq \frac{\lambda_k}{2} \left| \frac{g_k+h_k}{2} \right|^{2(k-1)!} \notag \\
    &\leq \lambda_k \frac{|g_k|^{2(k-1)!} + |h_k|^{2(k-1)!}}{4}. \label{induct-3}
  \end{align}
  Using \eqref{induct-2} and \eqref{induct-3}, we have that there exists a constant $C_2$ depending on $C_1$ and $\lambda_k$ such that
  \begin{align*}
    \left( \frac{N_k(gh)}{2} \right)^{2k!} &\leq \frac{N_k(g)^{2k!} + N_k(h)^{2k!}}{2} - C_1\alpha_{k-1,k}  - \frac{\lambda_k(|g_k|^{2(k-1)!} + |h_k|^{2(k-1)!})}{4} \\
    &\qquad- \sum_{i=3}^k \left[ \sum_{j=1}^i \binom{i}{j} C_1^j \alpha_{i-1,i-1}^j |g_1 + h_1|^{2(i-1)! \cdot (i-j)}\right]^{k!/i!}\\
    &\leq \frac{N_k(g)^{2k!} + N_k(h)^{2k!}}{2} - C_2\alpha_{k,k} \\
    &\qquad- \sum_{i=3}^k \left[ \sum_{j=1}^i \binom{i}{j} C_2^j \alpha_{i-1,i-1}^j |g_1 + h_1|^{2(i-1)! \cdot (i-j)}\right]^{k!/i!}.
  \end{align*}
  This completes the proof of \eqref{induct-result} with $C_0 = C_2$.
  
  We now prove \eqref{induct-1} before we finish the proof of the Proposition.  By $2(k-1)!$-convexity of the Euclidean norm, we have
  \begin{align*}
    | g_1+h_1 |^{2(k-1)!} + | g_1-h_1 |^{2(k-1)!} \geq 2 |g_1|^{2(k-1)!} + 2|h_1|^{2(k-1)!},
  \end{align*}
  which gives
  \begin{align}
    |g_1 + h_1|^{2(k-1)!} \vee |g_1-h_1|^{2(k-1)!} \geq |g_1|^{2(k-1)!} + |h_1|^{2(k-1)!}. \label{sum-diff-ineq}
  \end{align}
  Let
  \begin{align*}
    \beta := N_{k-1}(gh)^{2(k-1)!}.
  \end{align*}
  By construction of $N_{k-1}(gh)$, we then have that
  \begin{align*}
    \beta \geq N_2(gh)^{2(k-1)!} \geq |g_1 + h_1|^{2(k-1)!}.
  \end{align*}
  We list the following properties which are straightforward from the definition of $\alpha_{k-1,k-1}$ and $\beta$.
  \begin{align}
    \alpha_{k-1,k-1} \vee \beta &\overset{\eqref{sum-diff-ineq}}{\geq} |g_1|^{2(k-1)!} \vee |h_1|^{2(k-1)!}, \label{alpha-v-beta} \\
    \alpha_{k-1,k-1} &\geq |[g_1,h_1]|^{(k-1)!}, \label{alpha1}\\
    \alpha_{k-1,k-1} &\geq |g_i|^{2(k-1)!/i} \vee |h_i|^{2(k-1)!/i}, \label{alpha-geq-1} ~~~~~2 \leq i \leq k-1.
  \end{align}
  Rearranging \eqref{induct-hyp}, we have
  \begin{multline}
    \frac{\beta}{2^{2(k-1)!}} + C_0 \alpha_{k-1,k-1} \leq \frac{N_{k-1}(g)^{2(k-1)!} + N_{k-1}(h)^{2(k-1)!}}{2} \\
    - \sum_{i=3}^{k-1} \left[ \sum_{j=1}^i \binom{i}{j} C_0^j \alpha_{i-1,i-1}^j |g_1 + h_1|^{2(i-1)! \cdot (i-j)}\right]^{(k-1)!/i!}. \label{induct-hyp-rearrange}
  \end{multline}
  Note that the right hand side of \eqref{induct-hyp-rearrange} must be positive as the left hand side is obviously so.  Thus, raising both sides to the power of $k$, and using the fact that $(a - b)^k \leq a^k - b^k$ when $a \geq b$ on the right hand side, we get
  \begin{multline*}
    \left(\frac{\beta}{2^{2(k-1)!}}\right)^k + \sum_{j=1}^k \binom{k}{j} C_0^k \alpha_{k-1,k-1}^k \left( \frac{\beta}{2^{2(k-1)!}} \right)^{k-j} \leq \left(\frac{N_{k-1}(g)^{2(k-1)!} + N_{k-1}(h)^{2(k-1)!}}{2}\right)^k \\
    - \left( \sum_{i=3}^{k-1} \left( \sum_{j=1}^i \binom{i}{j} C_0^j \alpha_{i-1,i-1}^j |g_1 + h_1|^{2(i-1)! \cdot (i-j)}\right)^{(k-1)!/i!} \right)^k.
  \end{multline*}
  We then have that there exists some constant $C_3 > 0$ depending on $C_0$ such that
  \begin{align*}
    \left( \frac{N_{k-1}(gh)}{2} \right)^{2k!} &\leq \frac{N_{k-1}(g)^{2k!} + N_{k-1}(h)^{2k!}}{2} - \sum_{j=1}^k \binom{k}{j} C_0^j \alpha_{k-1,k-1}^j \left( \frac{\beta}{2^{2(k-1)!}} \right)^{k-j} \\
    &\qquad- \sum_{i=3}^{k-1} \left( \sum_{j=1}^i \binom{i}{j} C_0^j \alpha_{i-1,i-1}^j |g_1 + h_1|^{2(i-1)! \cdot (i-j)}\right)^{k!/i!} \\
    &\leq \frac{N_{k-1}(g)^{2k!} + N_{k-1}(h)^{2k!}}{2} - \sum_{j=1}^k \binom{k}{j} C_3^j \alpha_{k-1,k-1}^j \beta^{k-j} \\
    &\qquad- \sum_{i=3}^k \left( \sum_{j=1}^i \binom{i}{j} C_3^j \alpha_{i-1,i-1}^j |g_1 + h_1|^{2(i-1)! \cdot (i-j)}\right)^{k!/i!}. \\
  \end{align*}
  The last inequality (in which the range of the second summation changes) comes from the fact that $\beta \geq |g_1 + h_1|^{2(k-1)!}$ and the $2^{-2(k-1)!(k-j)}$ terms.  It now suffices to show that there exists some $C_4 > 0$ such that
  \begin{align*}
    \alpha_{k-1,k} + |P_k|^{2(k-1)!} \leq \sum_{j=1}^k \binom{k}{j} C_4^j \alpha_{k-1,k-1}^j \beta^{k-j}
  \end{align*}
  That $\alpha_{k-1,k} \leq \alpha_{k-1,k-1}^k$ is straightforward from the definition and so it remains to bound $|P_k|^{2(k-1)!}$.  By the BCH formula, we have that $P_k$ is finite summation of nested Lie brackets $[x_1,[x_2,...,[x_{j-1},x_j]...]]$ where $x_l$ is either $g_{i(l)}$ or $h_{i(l)}$ and
  \begin{align*}
    \sum_{l=1}^j i(l) = k.
  \end{align*}
  At the loss of some multiplicative constant depending only on $H$, it then suffices to bound each Lie bracket raised to the power $2(k-1)!$.  We have two cases.  Suppose that $[x_{j-1},x_j] = \pm [g_1,h_1]$.  Then $|[x_{j-1},x_j]|^{2(k-1)!} \leq \alpha_{k-1,k-1}^2$ by \eqref{alpha1}.  We also have $|x_l|^{2(k-1)!} \leq \alpha_{k-1,k-1}^{i(l)}$ if $i(l) > 1$ or $|x_l|^{2(k-1)!} \leq \alpha_{k-1,k-1} \vee \beta$ if $i(l) = 1$ by \eqref{alpha-v-beta} and \eqref{alpha1}, respectively.  Putting everything together, we get that there exists a constant $C_5 > 0$ depending on the group structure such that
  \begin{align*}
    |[x_1,[x_2,...,[x_{j-1},x_j]...]]|^{2(k-1)!} &\leq C_5 \alpha_{k-1,k-1}^2 \prod_{l=1}^{j-2} (\alpha_{k-1,k-1} \vee \beta)^{i(l)} \\
    &\leq C_5 \sum_{j=1}^k \binom{k}{j} \alpha_{k-1,k-1}^j \beta^{k-j}.
  \end{align*}
  If $[x_{j-1},x_j] \neq \pm [g_1,h_1]$ then $i(j-1) \vee i(j) > 1$, and so we have the simple bound
  \begin{align*}
    |[x_1,[x_2,...,[x_{j-1},x_j]...]]|^{2(k-1)!} &\leq C_5 \prod_{l : i(l) > 1} \alpha_{k-1,k-1}^{i(l)} \cdot \prod_{l : i(l) = 1} (\alpha_{k-1,k-1} \vee \beta) \\
    &\leq C_5 \sum_{j=1}^k \binom{k}{j} \alpha_{k-1,k-1}^j \beta^{k-j}.
  \end{align*}
  This ends the proof of \eqref{induct-1} and completes the proof of \eqref{induct-result}.

  Having shown \eqref{induct-result} for $k = s$, to prove the statement of the proposition, we need to show that there exists a $C_6 > 0$ such that
  \begin{align*}
    C_6\left(|g_1 - h_1|^{2s!} + NH(gh)^{2s!}\right) \leq \sum_{k=3}^s \left[ \sum_{j=1}^k \binom{k}{j} C_0^j \alpha_{k-1,k-1}^j |g_1 + h_1|^{2(i-1)! \cdot (k-j)}\right]^{s!/k!} + C_0\alpha_{s,s}.
  \end{align*}
  That $|g_1 - h_1|^{2s!} \leq \alpha_{s,s}$ follows from definition.  We have
  \begin{align*}
    NH(gh) = N((-g_1-h_1,0,...,0)(g_1+h_1,g_2+h_2+P_2,...,g_s+h_s+P_s)).
  \end{align*}
  By expanding out all the BCH Lie brackets, we see that, to bound $NH(gh)$, it suffices to bound for each $k \in \{1,...,s\}$ Lie brackets of the form
  \begin{align*}
    |[x_1,[x_2,...,[x_{j-1},x_j]...]]|^{2s!/k}
  \end{align*}
  where $x_l$ is either $g_{i(l)}$ or $h_{i(l)}$ and
  \begin{align*}
    \sum_{l=1}^j i(l) = k.
  \end{align*}
  We first suppose $j \geq 2$.  If $k \geq 3$ then by the argument before, we have for some $C_7 > 0$ that
  \begin{align*}
    |[x_1,[x_2,...,[x_{j-1},x_j]...]]|^{2s!/k} \leq \left[\sum_{j=1}^k \binom{k}{j} C_7^j \alpha_{k-1,k-1}^j |g_1 + h_1|^{2(i-1)! \cdot (k-j)}\right]^{s!/k!}.
  \end{align*}
  If $k = 2$ then the only Lie bracket is $[x_1,x_2] = [g_1,h_1]$ and so
  \begin{align*}
    |[x_1,x_2]|^{2s!/2} \leq \alpha_{s,s}.
  \end{align*}
  Now if $j = 1$ then we have that $x_1 = g_k$ for some $k \geq 2$.  Then we have the bound
  \begin{align*}
    |g_i|^{2s!/k} \leq \alpha_{s,s}.
  \end{align*}
  This finishes the proof.
\end{proof}

\begin{remark}
  This proposition says that the norm is midpoint ``uniformly convex''.  Indeed, by taking $N(g) = 1$ and $N(h) = 1$, we get that
  \begin{align*}
    N(\delta_{1/2}(g) \delta_{1/2}(h)) \leq \left[1 - C2^{-p} \left( |g_1 - h_1|^p + NH(gh)^p \right) \right]^{1/p}.
  \end{align*}
  However, as graded nilpotent Lie groups of step greater than 1 are nonabelian, iterating midpoints do not produce the dyadic points between $g$ and $h$ with respect to $\delta$.  Thus, the group norm is not guaranteed to be a true norm, but only a quasinorm.
\end{remark}

\subsection{Carnot groups and Markov convexity}

We can show that Proposition \ref{carnot-convexity} is an actual convexity result in the following sense: graded nilpotent Lie groups with homogeneous semimetrics that satisfy \eqref{C-eqn} have nontrivial Markov convexity.  We recall the definition of Markov convexity which was introduced in \cite{LNP}.

Let $\{X_t\}_{t \in \Z}$ be a Markov chain on a state space $\Omega$.  Given some integer $k \geq 0$, we denote $\{\tilde{X}_t(k)\}_{t \in \Z}$ to be the process that equals $X_t$ for $t \leq k$ and then evolves independently (with respect to the same transition probabilities as $X_t$) for $t > k$.  Let $p > 0$.  We then say that a metric space $(X,d_X)$ is Markov $p$-convex if there exists some constant $\Pi$ so that, for every Markov chain $\{X_t\}_{t \in \Z}$ on $\Omega$ and every $f : \Omega \to X$, we have that
\begin{align*}
  \sum_{k = 0}^\infty \sum_{t \in \Z} \frac{\mathbb{E} \left[ d_X\left( f(X_t),f\left( \tilde{X}_t(t-2^k) \right) \right)^p \right]}{2^{kp}} \leq \Pi^p \sum_{t \in \Z} \mathbb{E}\left[ d_X(f(X_{t-1}),f(X_t))^p \right].
\end{align*}
A metric space has nontrivial Markov convexity if it is Markov $p$-convex for some $p < \infty$.

It was proven in \cite{LNP,MN:08} that a Banach space $X$ is Markov $p$-convex if and only if it has an equivalent norm $\|\cdot\|$ so that $(X,\|\cdot\|)$ satisfies \eqref{UC-eqn} for the same $p$ and some $K > 0$ that can be controlled by $\Pi$.  Thus, it follows that the linear invariant of being isomorphic to a uniformly convex space can be expressed as the metric invariant of having nontrivial Markov convexity.  This can be thought of as a sharpening of Bourgain's metrical characterization of superreflexive spaces \cite{bourgain-superreflexivity}.  Both Markov convexity and Bourgain's characterization are parts of a larger research program, the Ribe program, that we now briefly describe.

Recall that a Banach space $X$ is said to be finitely representable in another Banach space $Y$ if there exists some $K \geq 1$ so that, for every finite dimensional subspace $Z \subset X$, there exists an isomorphic embedding $T : Z \to Y$ so that $\|T\|_{lip}\|T^{-1}\|_{lip} \leq K$.  In \cite{ribe}, Ribe proved that Banach spaces that are uniformly homeomorphic are also mutually finitely representable.  This says that a quantitative metric equivalence between Banach spaces (uniformly homeomorphism) induces a finite dimensional type of linear equivalence (finite representability).  Thus, it may be possible to characterize linear properties of Banach spaces that depend only on their finite dimensional substructure in purely metric terms.  This is the Ribe program, an active line of research that has seen many advances over the past three decades For more information on the Ribe program, see the surveys \cite{ball,naor}.  It should be noted that Bourgain's discretization theorem is a quantitative reformulation of Ribe's theorem.

We can prove that all graded nilpotent Lie groups have are Markov $p$-convex for some $p < \infty$.

\begin{theorem} \label{markov-convex}
  Let $(H,d)$ be an $s$-step graded nilpotent Lie group that satisfies \eqref{C-eqn} for some $p \in [1,\infty)$ and $K \geq 1$.   Then $H$ is Markov $(p \cdot s!)$-convex.
\end{theorem}

\begin{proof}
  By the proof of Proposition 2.1 of \cite{MN:08}, it suffices to prove that there exists some constant $C \geq 1$ so that
  \begin{align}
    \frac{d(x,w)^{ps!} + d(x,z)^{ps!}}{2^{{ps!}-1}} + \frac{d(z,w)^{ps!}}{C^{ps!}} \leq d(y,w)^{ps!} + d(z,y)^{ps!} + 2d(y,x)^{ps!}. \label{markov-4pt-ineq}
  \end{align}
  Raising \eqref{C-eqn} to the power $s!$ and using Jensen's inequality, we get that there exist some constant $C_0 > 0$ so that
  \begin{align*}
    d(x,y)^{ps!} + d(y,w)^{ps!} \geq \frac{d(x,w)^{ps!}}{2^{ps!-1}} + C_0 \sum_{j=1}^{s!} d(x,w)^{p(s!-j)} \left(\left| \frac{x_1 + w_1}{2} - y_1\right|^p + NH(x^{-1} w)^p\right)^j.
  \end{align*}
  Doing the same thing with $z$ in place of $w$, setting $x = 0$, and adding the two inequalities together, we see that it suffices to show that there exists some $C_1 > 0$ so that for any $y \in H$ we have
  \begin{multline}
     d(z,w)^{ps!} \leq C_1 \sum_{j=1}^{s!} N(w)^{p(s!-j)} \left( \left| \frac{w_1}{2} - y_1\right|^p + NH(w)^p \right)^j \\
     + C_1 \sum_{j=1}^{s!} N(z)^{p(s!-j)} \left( \left| \frac{z_1}{2} - y_1\right|^p + NH(z)^p \right)^j. \label{convex-reduction1}
  \end{multline}
  By the quasi-triangle inequality and Jensen's inequality, we know that there exists some constant $C_2 > 0$ so that
  \begin{align}
    d(z,w)^{ps!} \leq C_2(NH(z)^{ps!} + d(\tilde{\pi}(z),\tilde{\pi}(w))^{ps!} + NH(w)^{ps!}). \label{convex-upper}
  \end{align}
  The terms in summations in \eqref{convex-reduction1} corresponding to the index $j = s!$ take care of the $NH$ terms in \eqref{convex-upper}.  Dropping these $NH$ terms from \eqref{convex-reduction1} and using the fact that there exists some $C_3 > 0$ so that $N(z) \geq C_3 |z_1|$ and $N(w) \geq C_3|w_1|$, we see that we only need to find a $C_4 > 0$ so that for any $y \in H$ we have
  \begin{align}
    d(\tilde{\pi}(z),\tilde{\pi}(w))^{ps!} \leq C_4 \sum_{j=1}^{s!} \left( |w_1|^{p(s!-j)} \left| \frac{w_1}{2} - y_1 \right|^{jp} + |z_1|^{p(s!-j)} \left| \frac{z_1}{2} - y_1 \right|^{jp} \right). \label{convex-reduction2}
  \end{align}
  Note that we have that there exists some $C_5 > 0$ so that
  \begin{align*}
    d(\tilde{\pi}(z),\tilde{\pi}(w))^{ps!} = d(0,(-z_1,0,...,0)(w_1,0,...,0))^{ps!} \leq C_5 \left[ |w_1 - z_1|^{ps!} \vee \left( \max_{k \in \{2,...,s\}} |P_k|^{\frac{ps!}{k}} \right) \right].
  \end{align*}
  where $P_k$ is the BCH polynomial of level $k$.  It is clear by the BCH polynomial that each nested Lie bracket of $P_k$ are of depth $k$ and compose only of $w_1$ or $z_1$.  By setting $j = s!$ on the right hand side of \eqref{convex-reduction2}, we see that the summation contains the summand
  \begin{align*}
    \left| \frac{w_1}{2} - y_1 \right|^{ps!} + \left|y_1 - \frac{z_1}{2} \right|^{ps!} \geq 2^{1-ps!} \left(\left| \frac{w_1}{2} - y_1\right| + \left|y_1 - \frac{z_1}{2} \right|\right)^{ps!} \geq 2^{1-2ps!} |w_1-z_1|^{ps!}.
  \end{align*}
  Thus, it suffices to bound each $|P_k|^{1/k}$ by some multiple of the right hand side of \eqref{convex-reduction2}.  As mentioned many times now, it suffices to bound each nested Lie bracket $[x_1,[x_2,...,[x_{k-1},x_k]...]]$ where each $x_j$ is either $w_1$ or $z_1$.  We can also suppose $[x_{k-1},x_k] = [w_1,z_1] = [w_1,z_1 - w_1]$ and so there exists some $C_6 > 0$ so that
  \begin{align*}
    |[x_1,[x_2,...,[x_{k-1},x_k]...]]|^{ps!/k} \leq C_6 |z_1 - w_1|^{ps!/k} (|w_1| \vee |z_1|)^{\frac{k-1}{k}ps!}.
  \end{align*}
  Suppose $\frac{1}{2}|z_1| \leq |w_1| \leq 2|z_1|$.  Then we have that there exists a $C_7 > 0$ so that
  \begin{align}
    \sum_{j=1}^{s!} &\left( |w_1|^{p(s!-j)} \left| \frac{w_1}{2} - y_1 \right|^{jp} + |z_1|^{p(s!-j)} \left| \frac{z_1}{2} - y_1 \right|^{jp} \right) \\
    &\geq C_7 \sum_{j=1}^{s!} (|w_1| \vee |z_1|)^{p(s!-j)} \left( \left| \frac{w_1}{2} - y_1\right|^{jp} + \left| \frac{z_1}{2} - y_1 \right|^{jp} \right) \notag \\
    &\geq C_7 2^{1-2ps!} \sum_{j=1}^{s!} (|w_1| \vee |z_1|)^{p(s!-j)} |w_1-z_1|^{jp}. \label{alternative-1}
  \end{align}
  Take the summand corresponding to the index $j = s!/k$.  Then we have that the summation of \eqref{alternative-1} has the summand
  \begin{align*}
    (|w_1| \vee |z_1|)^{\frac{k-1}{k} ps!} |w_1-z_1|^{\frac{ps!}{k}},
  \end{align*}
  which finishes the proof for the case $\frac{1}{2}|z_1| \leq |w_1| \leq 2|z_1|$.  Now suppose that this is not the case, and without loss of generality assume that $|w_1| > 2|z_1|$.  Then there exists some constant $C_8 > 0$ so that
  \begin{align}
    |[x_1,[x_2,...,[x_{k-1},x_k]...]]|^{ps!} \leq C_8 |w_1|^{\frac{k-1}{k}ps!} |z_1|^{\frac{ps!}{k}}, \label{max-w}
  \end{align}
  as we have that $|x_j| \leq |w_1|$ for every $j$, although either $x_{k-1}$ or $x_k$ has to be $z_1$ so that the Lie bracket is nonzero.  Thus, for any $y \in H$ we have by the triangle inequality that
  \begin{align*}
    \left| \frac{w_1}{2} - y_1\right| + \left| \frac{z_1}{2} - y_1\right| \geq \frac{1}{2}|w_1 - z_1| \geq \frac{1}{4} |w_1|,
  \end{align*}
  and so
  \begin{align*}
    \left| \frac{w_1}{2} - y_1 \right|^{jp} \vee \left| \frac{z_1}{2} - y_1 \right|^{jp} \geq 8^{-jp} |w_1|^{jp}.
  \end{align*}
  Now we get that there exists some $C_9 > 0$ so that
  \begin{align}
    \sum_{j=1}^{s!} \left( |w_1|^{p(s!-j)} \left| \frac{w_1}{2} - y_1 \right|^{jp} + |z_1|^{p(s!-j)} \left| \frac{z_1}{2} - y_1 \right|^{jp} \right) &\geq C_9 \sum_{j=1}^{s!} |z_1|^{p(s!-j)} |w_1|^{jp}. \label{alternative-2}
  \end{align}
  Looking at the summand corresponding to $j = \frac{k-1}{k} s!$ in \eqref{alternative-2}, we get that the summation is greater than a multiple of $|w_1|^{\frac{k-1}{k}ps!} |z_1|^{ps!}$, as required by \eqref{max-w}.
\end{proof}

As each Carnot group satisfies \eqref{C-eqn} for some $p < \infty$ by Proposition \ref{carnot-convexity}, we get that all Carnot groups have nontrivial Markov convexity.  Note that Carnot groups do not biLipschitzly embed into uniformly convex Banach spaces  and thus any Banach space with nontrivial Markov convexity.  Indeed, this follows from Theorem \ref{full-nonembed} (which will be proven in the next section) or from \cite{CK3,LN:06} where it was observed that Pansu differentiation, and thus Semmes's argument, extends easily to the case of uniformly convex targets.  Thus, we have proven the following theorem:

\begin{theorem}
  There exists a metric space of nontrivial Markov convexity that does not biLipschitzly embed into any Banach space of nontrivial Markov convexity.
\end{theorem}

This corollary can be contrasted with \cite{LNP} where it was shown that trees with nontrivial Markov convexity can always be embedded into some uniformly convex $L_p$ (and thus have nontrivial Markov convexity).

\section{Some applications}
We can now use coarse differentiation to prove some results for quantitative nonembeddability and an analogue of Bourgain's discretization theorem.  We start with the simplest case of the nonembeddability results as a warmup.

\subsection{Nonembeddability into $p$-convex spaces}

\begin{theorem} \label{UC-nonembed}
  Let $(X,\|\cdot\|)$ be a uniformly convex Banach space.  Then there exist $c,C > 0$ such that for every $f : B_G \to X$ which is 1-Lipschitz with respect to the Carnot-Carath\'{e}odory metric there exist $x,y \in G$ with $\dcc(x,y)$ arbitrarily small so that
  \begin{align*}
    \frac{\|f(x) - f(y)\|}{\dcc(x,y)} \leq C\left( \log \frac{1}{\dcc(x,y)} \right)^{-c}.
  \end{align*}
\end{theorem}

\begin{proof}
  Construct the Christ cubes of $G$ take a cube $S \in \{Q \in \Delta : Q \subset B_G\}$ so that $\ell(S)$ is maximal.  Thus, there is some constant $C_0 > 0$ depending only on $G$ so that
  \begin{align}
    \frac{1}{C_0} \leq \ell(S) \leq C_0. \label{maximal-cube}
  \end{align}
  As $f$ is Lipschitz, it is $0$-LLD and so the condition on the size of $S$ in Theorem \ref{UC-UAAP} is empty.  The same theorem then gives that there exist $\zeta,\alpha > 0$ depending only on $G$ so that for $\epsilon \in (0,1/2)$, we get
  \begin{align}
    \sum \left\{|Q| : Q \in \Delta, Q \subseteq S, \qd_f^{UC}(Q,\zeta \epsilon^r) > \epsilon \right\} \leq \epsilon^{-\alpha} |S|. \label{UC-carleson-bnd}
  \end{align}
  Let $m = \lceil \epsilon^{-\alpha} \rceil$.  There exists some $C_1 > 0$ depending only on $G$ so that
  \begin{align}
     \zeta a_0 \ell(S) \tau^{3C_1 \epsilon^{-\alpha}} \leq \zeta \epsilon^r \tau^{2m} a_0 \ell(S) \leq \zeta \epsilon^r \tau^m a_0 \ell(S) \leq \zeta a_0 \ell(S) \tau^{C_1 \epsilon^{-\alpha}}. \label{dist-bound}
  \end{align}
  Here, we needed to specify that $\epsilon$ be smaller than some constant depending only on $r$, $\alpha$, and $\tau$ so that $\epsilon^r \geq \tau^{C_1 \epsilon^{-\alpha}}$.

  Let $A = \{Q \in \Delta : Q \subseteq S, \qd_f^{UC}(Q,\zeta \epsilon^r) > \epsilon \}$ and $\Delta_k \cap S = \{Q \in \Delta_k : Q \subseteq S\}$.  Suppose $\bigcup_{k=m}^{2m} (\Delta_k \cap S) \subseteq A$.  By the partitioning property of $\Delta$, we get for any $k \geq 0$ that
  \begin{align*}
    \sum_{Q \in \Delta_k \cap S} |Q| = |S|.
  \end{align*}
  Thus, we have
  \begin{align*}
    \sum_{k=m}^{2m} \sum_{Q \in (\Delta_k \cap S) \cap A} |Q| = (m+1)|S|.
  \end{align*}
  We get a contradiction of \eqref{UC-carleson-bnd} from the definition of $m$.  Thus, we have proven that there exists some $k \in \{m,...,2m\}$, $Q \in \Delta_k \cap S$, $v \in X$, and homomorphism $A : G \to Y$ so that
  \begin{align*}
    \frac{1}{\zeta \epsilon^r a_0\ell(Q)} \sup \{\|f(x) - A(x) - v\| : x \in \zeta \epsilon^r B_Q \} \leq \epsilon.
  \end{align*}
  Remembering that $a_0\ell(Q)$ is the radius of $B_Q$, we see that if we choose $x = z_Q$ and $y$ to be $xg \in \partial (\zeta \epsilon^r B_Q)$ where $\pi(g) = 0$, then   \begin{align*}
    \frac{\|f(x) - f(y)\|}{\dcc(x,xg)} \leq \frac{\|f(x) - A(x) - v\|}{\dcc(x,xg)} + \frac{\|A(x) - A(xg)\|}{\dcc(x,xg)} + \frac{\|f(y) - A(x) - v\|}{\dcc(x,xg)}.
  \end{align*}
  As $A : G \to X$ is a homomorphism from a nonabelian group to an abelian group, we get that it has a kernel, which is easily seen to be the exponential image of the subspace orthogonal to the horizontal subspace.  Thus, $\|A(x) - A(xg)\| = 0$ and so
  \begin{align}
    \frac{\|f(x) - f(xg)\|}{\dcc(x,xg)} \leq 2\epsilon. \label{UC-embed-bad}
  \end{align}
  Note that
  \begin{align*}
    \dcc(x,xg) \in \left[ \zeta \epsilon^r a_0 \tau^{2m} \ell(S), \zeta \epsilon^r a_0 \tau^m \ell(S) \right] \overset{\eqref{maximal-cube} \wedge \eqref{dist-bound}}{\subseteq} \left[ \frac{1}{C_0} \zeta a_0 \tau^{3C_1 \epsilon^{-\alpha}}, C_0 \zeta a_0 \tau^{C_1 \epsilon^{-\alpha}}\right].
  \end{align*}
  Thus, taking \eqref{UC-embed-bad} into account, we have that
  \begin{align}
    \frac{\|f(x) - f(xg)\|}{\dcc(x,xg)} \leq 2\left( 3C_1 \log \frac{1}{\tau} \right)^{1/\alpha} \left( \log \frac{1}{\dcc(x,y)} - \log \frac{C_0}{\zeta a_0}\right)^{-1/\alpha}.
  \end{align}
  As $\epsilon$ can be made arbitrarily small, we can then make $\dcc(x,y)$ arbitrarily small by the upper bound \eqref{dist-bound}.  Thus, $\log \frac{1}{\dcc(x,y)}$ becomes much larger than the constant $\log \frac{C_0}{\zeta a_0}$, which proves the statement of the theorem.
\end{proof}

\begin{remark}
  The same argument can be used to prove Theorem \ref{full-nonembed} in the case of embeddings into Carnot groups with obvious modifications.  The double log rate comes from the double exponential rate of Theorem \ref{UAAP-rest}.

  In this context, it should be noted that not every homomorphism between Carnot groups is Lipschitz.  For example, the homomorphism that maps $\R$ to the Heisenberg group via the $z$-axis is not Lipschitz, but $\frac{1}{2}$-H\"{o}lder continuous.  Thus, that $H$ does not admit a biLipschitz homomorphic embedding of $G$ can be a stronger condition than $H$ not admitting a homomorphic embedding of $G$.
\end{remark}

\subsection{Nonembeddability into $CBB(0)$ spaces}
We now move to Alexandrov spaces.  The arguments for the next two sections are inspired from those in \cite{Pauls}, but we have to make all the infinitesimal arguments quantitative (and in the case of $CBB(0)$ spaces use a slightly different argument).

\begin{theorem} \label{CBB-nonembed}
  Let $(X,d_X)$ be a $CBB(0)$ space.  Then there exist $c,C > 0$ such that for every $f : B_G \to X$ which is 1-Lipschitz with respect to the Carnot-Carath\'{e}odory metric there exist $x,y \in G$ with $\dcc(x,y)$ arbitrarily small so that
  \begin{align*}
    \frac{d_X(f(x),f(y))}{\dcc(x,y)} \leq C\left( \log \frac{1}{\dcc(x,y)} \right)^{-c}.
  \end{align*}
\end{theorem}

We will need the following lemma, which is essentially Lemma 7.1 of \cite{Pauls}.

\begin{lemma} \label{sublinear-diverge}
  Let $\lambda > 0$.  There exist constant $C > 1$, $0 < \beta_1 < \beta_2 < 1$, and $u,v \in S^{n-1}$ depending only on $G$ so that if we define the two lines
  \begin{align*}
    \gamma_0(t) &= e^{tv}, \\
    \gamma_1(t) &= e^{\lambda u} e^{tv}.
  \end{align*}
  then we have for all $|t| > \lambda$ that
  \begin{align}
    \frac{\lambda^{1-\beta_1}}{C} |t|^{\beta_1} < \dcc(\gamma_0(t),\gamma_1(t)) < \lambda^{1-\beta_2} C|t|^{\beta_2}. \label{sublinear-diverge-eqn}
  \end{align}
\end{lemma}

For $CBB(0)$ spaces, we can prove that almost minimizing curves are close to actual minimizing geodesics on a subinterval.  To do so, we first need the following theorem, which is a special case of Theorem 3.2 of \cite{BGP}.

\begin{theorem} \label{global-4pts}
  Let $X$ be a $CBB(0)$ space.  Then for any four points $a,b,c,d \in X$, we have the inequality
  \begin{align*}
    \widetilde{\angle} bac + \widetilde{\angle} bad + \widetilde{\angle}cad \leq 2\pi.
  \end{align*}
\end{theorem}

\begin{lemma} \label{CBB-close-geodesic}
  Let $(X,d_X)$ be a $CBB(0)$ space.  Suppose $\gamma : [a,b] \to X$ is Lipschitz and there exists some $L \geq L' > 0$ so that
  \begin{align*}
    \left| d_X(\gamma(s),\gamma(t)) - |t - s| L' \right| \leq \epsilon (b-a) L, \qquad \forall s,t \in [a,b].
  \end{align*}
  Here, $\epsilon \in (0,L'/L)$.  Let $\tilde{\gamma} : \left[a + \frac{b-a}{4},b\right] \to X$ denote the constant speed minimal geodesic from $\gamma\left(a + \frac{1}{4}(b-a)\right)$ to $\gamma(b)$.  Then there exists some universal constant $C \geq 1$ so that
  \begin{align*}
    \sup_{\lambda \in \left[a + \frac{1}{3}(b-a), a + \frac{2}{3}(b-a)\right]} d_X(\gamma(\lambda),\tilde{\gamma}(\lambda)) \leq C \epsilon^{1/2} (b-a)L.
  \end{align*}
\end{lemma}

\begin{proof}
  We will first suppose that $L' = L$.  We may assume without loss of generality that $[a,b] = [0,1]$.  Let $\lambda \in [1/3,2/3]$.  By hypothesis, we have that
  \begin{align*}
    d_X(\gamma(0),\gamma(1/4)) &\leq \frac{L}{4} + \epsilon L, \\
    d_X(\gamma(1/4),\gamma(\lambda)) &\leq \left( \lambda - \frac{1}{4} \right)L + \epsilon L, \\
    d_X(\gamma(1/4),\tilde{\gamma}(\lambda)) &\leq \left( \frac{4}{3} \lambda - \frac{1}{3} \right) \left( \frac{3}{4} L + \epsilon L \right), \\
    d_X(\gamma(1/4),\gamma(1)) &\leq \frac{3}{4} L + \epsilon L, \\
    d_X(\gamma(0),\gamma(\lambda)) &\geq \lambda L - \epsilon L, \\
    d_X(\gamma(0),\gamma(1)) &\geq L - \epsilon L.
  \end{align*}
  From these bounds, we can use the law of cosines to bound the comparison angles
  \begin{align*}
    \widetilde{\angle} \gamma(0)\gamma(1/4)\gamma(\lambda) &\geq \pi - (240\epsilon)^{1/2}, \\
    \widetilde{\angle} \gamma(0)\gamma(1/4)\tilde{\gamma}(\lambda) &\geq \widetilde{\angle} \gamma(0)\gamma(1/4)\gamma(1) \geq \pi - (32\epsilon)^{1/2}.
  \end{align*}
  The last inequality comes from the monotonicity of angles condition for $CBB(0)$ spaces.  Letting $a = \gamma(1/4)$, $b = \gamma(0)$, $c = \gamma(\lambda)$ and $d = \tilde{\gamma}(\lambda)$, we get, by Theorem \ref{global-4pts}, that
  \begin{align*}
    \widetilde{\angle} \gamma(\lambda) \gamma(1/4) \tilde{\gamma}(\lambda) \leq (240 \epsilon)^{1/2} + (32 \epsilon)^{1/2} \leq 24 \epsilon^{1/2}.
  \end{align*}
  Applying the law of cosines once more, we get that there exists some universal constant $C > 0$ so that
  \begin{align*}
    d_X(\gamma(\lambda),\tilde{\gamma}(\lambda)) \leq C \epsilon^{1/2} L.
  \end{align*}

  Now suppose $L' < L$.  Then we have for all $s,t \in [a,b]$ that
  \begin{align*}
    \left| d_X(\gamma(s),\gamma(t)) - |t - s| L' \right| \leq \epsilon (b-a) L = \epsilon \frac{L}{L'} (b-a) L'.
  \end{align*}
  As $\epsilon \frac{L}{L'} < 1$, applying the lemma gives us that 
  \begin{align*}
    \sup_{\lambda \in \left[a + \frac{1}{3}(b-a), a + \frac{2}{3}(b-a)\right]} d_X(\gamma(\lambda),\tilde{\gamma}(\lambda)) \leq C \epsilon^{1/2} (b-a) \left( \frac{L}{L'} \right)^{1/2} L' \leq C \epsilon^{1/2} (b-a) L.
  \end{align*}
\end{proof}

We now prove a quantitative version of Lemma 10.5.4 of \cite{BBI} for line segments.

\begin{lemma} \label{CBB-angle-control}
  Let $\epsilon > 0$ and $\gamma : [-L,L] \to X$ be a unit speed minimal geodesic.  Let $x \in X$ so that $d_X(x,\gamma(0)) \leq L \cos\left( \frac{\pi}{2} - \frac{\epsilon}{2} \right)$.  Then $\angle x \gamma(0) \gamma(L) \leq \widetilde{\angle} x \gamma(0) \gamma(L) + \epsilon$
\end{lemma}

\begin{proof}
  If $x = \gamma(0)$, then the statement is trivial.  Thus, we may assume $x \neq \gamma(0)$.  Suppose that
  \begin{align}
    \theta := \widetilde{\angle} x\gamma(0)\gamma(L) < \angle x \gamma(0) \gamma(L) - \epsilon. \label{CBB-angle-small}
  \end{align}
  Construct comparison triangles in $\R^2$ of $\widetilde{\triangle} x \gamma(0) \gamma(L)$ and $\widetilde{\triangle} x \gamma(0) \gamma(-L)$ so that $\overline{\gamma(L)}$ and $\overline{\gamma(-L)}$ are on opposite sides of the line spanned by $x\gamma(0)$.  Note that the distance between $\overline{\gamma(L)}$ and $\overline{\gamma(-L)}$ in $\R^2$ may not necessarily be $d_X(\gamma(0),\gamma(-L))$.  We have that
  \begin{align*}
    \angle_{\R^2} \gamma(-L)\gamma(0)\gamma(L) &= \widetilde{\angle} \gamma(-L)\gamma(0)x + \widetilde{\angle} x\gamma(0)\gamma(L) \\
    &\overset{\eqref{CBB-angle-small}}{\leq} \angle \gamma(-L)\gamma(0)x + \angle x \gamma(0)\gamma(L) - \epsilon \\
    &= \pi - \epsilon.
  \end{align*}
  Here, $\angle_{\R^2} \gamma(-L)\gamma(0)\gamma(L)$ denotes the angle of the hinge in the above two-triangle construction, \textit{not the comparison triangle}.  Setting $\angle {\R^2} \gamma(-L)\gamma(0)\gamma(L) = \pi - \delta$, by the law of cosines, we then get that
  \begin{align}
    d_X(x,\gamma(L))^2 &= L^2 + d_X(x,\gamma(0))^2 - 2Ld_X(x,\gamma(0)) \cos \theta, \label{cos-dist-1} \\
    d_X(x,\gamma(-L))^2 &= L^2 + d_X(x,\gamma(0))^2 - 2Ld_X(x,\gamma(0)) \cos (\pi - \delta - \theta). \label{cos-dist-2}
  \end{align}
  As $\gamma$ is a minimizing geodesic between its endspoints, we must have that
  \begin{align}
    d_X(\gamma(L),x) + d_X(x,\gamma(-L)) \geq d_X(\gamma(L),\gamma(0)) + d_X(\gamma(0),\gamma(-L)) = 2L. \label{line-triangle-ineq}
  \end{align}
  To derive a contradiction, we bound from above by Cauchy-Schwarz
  \begin{align*}
    d_X(\gamma(L),x) + d(x,\gamma(-L)) \leq \left( 2 d_X(\gamma(L),x)^2 + 2 d_X(x,\gamma(-L))^2 \right)^{1/2}.
  \end{align*}
  Thus, it suffices to maximize the term inside the square root with respect to $\theta$.  Taking \eqref{cos-dist-1} and \eqref{cos-dist-2} into account, we get by calculus that the maximum is achieved when $\theta = \frac{\pi - \delta}{2}$.  Plugging this in, we get
  \begin{align}
    d_X(\gamma(L),x) + d(x,\gamma(-L)) \leq 2\left( L^2 + d_X(x,\gamma(0))^2 - 2 Ld_X(x,\gamma(0)) \cos \left( \frac{\pi}{2} - \frac{\delta}{2} \right) \right)^{1/2}. \label{ineq-aab}
  \end{align}
  Remembering that $0 < d_X(x,\gamma(0)) < L \cos\left( \frac{\pi}{2} - \frac{\epsilon}{2} \right) \leq L \cos\left( \frac{\pi}{2} - \frac{\delta}{2} \right)$, we have that
  \begin{align}
    d_X(x,\gamma(0))^2 - 2Ld_X(x,\gamma(0)) \cos \left( \frac{\pi}{2} - \frac{\delta}{2} \right) < 0. \label{ineq-aaa}
  \end{align}
  giving us
  \begin{align*}
    d_X(\gamma(L),x) + d(x,\gamma(-L)) \overset{\eqref{ineq-aab} \wedge \eqref{ineq-aaa}}{<} 2L,
  \end{align*}
  a contradiction of \eqref{line-triangle-ineq}.
\end{proof}

We can now prove that geodesic segments that start off close either diverge linearly or stay bounded for some time.

\begin{lemma} \label{CBB-linear-diverge}
  Let $\gamma_0 : [-L,L] \to X$ and $\gamma_1 : [-L,L] \to X$ be unit speed minimal geodesics.  If $d_X(\gamma_0(0),\gamma_1(0)) \leq \epsilon L$ then there is a constant $\alpha \geq 0$ so that
  \begin{align*}
    | d_X(\gamma_0(x),\gamma_1(x)) - \alpha x | \leq 30 \epsilon L, \qquad \forall x \in [0,\epsilon^{1/2}L].
  \end{align*}
\end{lemma}

\begin{proof}
  Let $\tilde{\gamma} : [0,\epsilon^{1/2} L] \to X$ be the constant speed minimizing geodesic from $\gamma_1(0)$ to $\gamma_0(\epsilon^{1/2} L)$.  By the triangle inequality, we have that
  \begin{align}
    \delta L := d_X(\gamma_0(\epsilon^{1/2} L),\gamma_1(0)) \in [(\epsilon^{1/2} - \epsilon)L,(\epsilon^{1/2} + \epsilon)L]. \label{thetaL-range}
  \end{align}
  Let $\theta_0 = \widetilde{\angle} \gamma_0(0)\gamma_0(\epsilon^{1/2}L)\gamma_1(0)$.  We have that
  \begin{align*}
    d_X(\gamma_0(\epsilon^{1/2} L),\gamma_1(0)) \leq (\epsilon^{1/2} + \epsilon) L \leq \cos\left(\frac{\pi}{2} - 4\epsilon^{1/2} \right) \frac{L}{2}.
  \end{align*}
  Thus, Lemma \ref{CBB-angle-control} tells us that
  \begin{align}
    \angle \gamma_0(0) \gamma_0(\epsilon^{1/2} L) \gamma_1(0) \leq \widetilde{\angle} \gamma_0(0) \gamma_0(\epsilon^{1/2} L) \gamma_1(0) + 8\epsilon^{1/2}. \label{angle-bound0}
  \end{align}
  Let $x \in [0,\epsilon^{1/2} L]$.  As $\gamma_0$ and $\tilde{\gamma}$ are both constant speed geodesics, we get by the law of cosines that
  \begin{align*}
    d_X(\gamma_0(x),\tilde{\gamma}(x))^2 &\leq \left(\frac{\epsilon^{1/2}L - x}{\epsilon^{1/2}L}\right)^2 \left(\epsilon L^2 + \delta^2 L^2 - 2L^2 \epsilon^{1/2} \delta \cos \angle \gamma_0(0) \gamma_0(\epsilon^{1/2} L) \gamma_1(0)\right) \\
    &\overset{\eqref{angle-bound0}}{\leq} \epsilon L^2 + \delta^2 L^2 - 2L^2 \epsilon^{1/2} \delta \cos (\theta_0 + 8 \epsilon^{1/2}) \\
    &\leq \epsilon L^2 + \delta^2 L^2 - 2L^2 \epsilon^{1/2} \delta \cos \theta_0 + 20 L^2 \epsilon \delta \sin \theta_0  \\
    &\leq \epsilon^2 L^2 + 40 L^2 \epsilon^2.
  \end{align*}
  For the last inequality, we have used the fact that $\theta_0$ is the comparison angle $\widetilde{\angle} \gamma_0(0) \gamma_0(\epsilon^{1/2}L) \gamma_1(0)$, which we already knows has side lengths $\epsilon^{1/2}L$, $\delta L$ and $\epsilon L$.  We also used the fact that $\delta \leq 2\epsilon^{1/2}$ to get
  \begin{align*}
    \epsilon L = d_X(\gamma_0(0),\gamma_1(0)) \geq \min \{ \delta L \sin \theta_0, \epsilon^{1/2} L \sin \theta_0\} \geq \frac{1}{2} \delta L \sin \theta_0,
  \end{align*}
  Thus, 
  \begin{align}
    \sup_{x \in [0,\epsilon^{1/2} L]} d_X(\gamma_0(x),\tilde{\gamma}(x)) \leq 7 \epsilon L. \label{close-tilde-geodesic}
  \end{align}
  We now use the technique above to show that $\tilde{\gamma}$ and $\gamma_1$ diverge linearly.  As before, Lemma \ref{CBB-angle-control} gives that
  \begin{align}
    \angle \gamma_0(\epsilon^{1/2} L) \gamma_1(0) \gamma_1(\epsilon^{1/2}L) \leq \widetilde{\angle} \gamma_0(\epsilon^{1/2} L) \gamma_1(0) \gamma_1(\epsilon^{1/2} L) + 8\epsilon^{1/2}. \label{angle-bound1}
  \end{align}
  Let $x \in [0,\epsilon^{1/2} L]$.  Then remembering that $\delta L = d_X(\tilde{\gamma}(\epsilon^{1/2} L),\gamma_0(0))$ and setting $\theta_1 := \widetilde{\angle} \gamma_0(\epsilon^{1/2} L) \gamma_1(0) \gamma_1(\epsilon^{1/2} L)$, we get
  \begin{align}
    d_X(\tilde{\gamma}(x),\gamma_0(x))^2 &\leq \left( \frac{x}{\epsilon^{1/2}L} \right)^2 \left( \epsilon L^2 + \delta^2 L^2 - 2L^2 \delta \epsilon^{1/2} \cos \angle \gamma_0(\epsilon^{1/2} L) \gamma_1(0) \gamma_1(\epsilon^{1/2}L) \right) \notag \\
    &\overset{\eqref{angle-bound1}}{\leq} \left( \frac{x}{\epsilon^{1/2}L} \right)^2 \left( \epsilon L^2 + \delta^2 L^2 - 2L^2 \delta \epsilon^{1/2} \cos (\theta_1 + 8\epsilon^{1/2}) \right) \notag \\
    &\leq \left( \frac{x}{\epsilon^{1/2}L} \right)^2 \left( d_X(\gamma_0(\epsilon^{1/2}L),\gamma_1(\epsilon^{1/2}L))^2 + 20 L^2 \delta \epsilon \sin \theta_1 \right). \label{2lines-gap-bound}
  \end{align}
  As $\theta_1$ is the comparison angle $\widetilde{\angle} \gamma_0(\epsilon^{1/2} L) \gamma_1(0) \gamma_1(\epsilon^{1/2} L)$, it follows from Euclidean geometry that
  \begin{align}
    d_X(\gamma_1(\epsilon^{1/2} L), \gamma_0(\epsilon^{1/2}L)) &\geq \min \left\{ d_X(\gamma_1(0),\gamma_1(\epsilon^{1/2}L)) \sin \theta_1, d_X(\gamma_1(0),\gamma_0(\epsilon^{1/2}L)) \sin\theta_1 \right\} \notag \\
    &\overset{\eqref{thetaL-range}}{\geq} \frac{1}{2} \delta L \sin \theta_1. \label{2lines-end-gap}
  \end{align}
  Now we can bound
  \begin{align}
    d_X(\tilde{\gamma}(x),\gamma_0(x)) &\overset{\eqref{2lines-gap-bound}}{\leq} \left( \frac{x}{\epsilon^{1/2}L} \right) \left( d_X(\gamma_0(\epsilon^{1/2}L),\gamma_1(\epsilon^{1/2}L)) + \frac{10 L^2 \delta \epsilon \sin \theta_1}{d_X(\gamma_0(\epsilon^{1/2} L),\gamma_1(\epsilon^{1/2}L))} \right) \notag \\
    &\overset{\eqref{2lines-end-gap}}{\leq} \frac{x}{\epsilon^{1/2}L} d_X(\gamma_0(\epsilon^{1/2}L),\gamma_1(\epsilon^{1/2}L)) + 20 L \epsilon. \label{two-lines-lower-bound}
  \end{align}
  where we used \eqref{thetaL-range} to give that $\delta \leq 2\epsilon^{1/2}$ in the last inequality.  By the triangle comparison condition of $CBB(0)$ spaces, we also have
  \begin{align}
    d_X(\tilde{\gamma}(x),\gamma_0(x)) &\geq \frac{x}{\epsilon^{1/2}L} d_X(\gamma_0(\epsilon^{1/2}L),\gamma_1(\epsilon^{1/2}L)). \label{two-lines-upper-bound}
  \end{align}
  Combining \eqref{close-tilde-geodesic}, \eqref{two-lines-lower-bound}, and \eqref{two-lines-upper-bound}, we get
  \begin{align*}
    \left| d_X(\gamma_0(x),\gamma_1(x)) - x \frac{d_X(\gamma_0(\epsilon^{1/2}L),\gamma_1(\epsilon^{1/2}L))}{\epsilon^{1/2}L} \right| \leq 28 \epsilon L
  \end{align*}
\end{proof}

\begin{proof}[Proof of Theorem \ref{CBB-nonembed}]
  We, as in the proof of Theorem \ref{UC-nonembed}, take a cube $S \in \Delta$ so that $S \subseteq B_G$ and $\ell(S)$ is maximal among all such cubes.  We get that there exist some $C_0 > 0$ depending only on $G$ so that
  \begin{align*}
    \frac{1}{C_0} \leq \ell(S) \leq C_0.
  \end{align*}
  Let $\zeta$ and $\alpha$ be the constants from Theorem \ref{M-UAAP}.  Let $\epsilon \in (0,1/2)$ and choose $D > 1$ so that
  \begin{align}
    \epsilon := \left[ C(124+C) D \right]^{\frac{2}{\beta_1(\beta_2-1)}}, \label{CBB-eps-defn}
  \end{align}
  where $C > 1$, $\beta_1$, and $\beta_2$ are the constants in Lemma \ref{sublinear-diverge}.  Note then that $\epsilon \leq D^{-2}$.  Let $C_1 \geq 1$ be the constant of Lemma \ref{CBB-close-geodesic} and set $m = \lceil (\epsilon/6C_1)^{-2\alpha} \rceil$.  We can define some $C_2 > 0$ depending only on $G$ so that
  \begin{align*}
    \zeta a_0 \ell(S) \tau^{3C_2 \epsilon^{-2\alpha}} \leq \zeta \epsilon^r a_0 \tau^{2m} \ell(S) \leq \zeta \epsilon^r a_0 \tau^m \ell(S) \leq \zeta a_0 \ell(S) \tau^{C_2 \epsilon^{-2\alpha}}.
  \end{align*}
  Here, we needed to specify that $\epsilon$ be smaller than some constant depending only on $r$, $\alpha$, and $\tau$ so that $\epsilon^r \geq \tau^{C_2m}$.  For convenience, we define
  \begin{align}
    \eta := \tau^{C_2 \epsilon^{-2\alpha}}. \label{CBB-eta-defn}
  \end{align}
  
  We can use Theorem \ref{M-UAAP} as in the proof of Theorem \ref{UC-nonembed} to find a ball $B = \zeta\epsilon^rB_Q$ of radius $r = \zeta \epsilon^r a_0 \tau^k \ell(S)$ for some $k \in \{m,...,2m\}$ and a function $w : S^{n-1} \to \R^+$ so that
  \begin{align}
    \sup_{-3r \leq s \leq t \leq 3r} \left| d_X(h(xe^{sv}),h(xe^{tv})) - |t-s| w(v) \right| \leq \left( \frac{\epsilon}{6^{3/2} C_1} \right)^2 6r, \qquad \forall x \in B. \label{CBB-coarse-diff}
  \end{align}
  Note that $r \in [\zeta a_0 \ell(S) \eta^3, \zeta a_0 \ell(S) \eta]$.
  
  Suppose for all $x,y \in G$ so that $\frac{\zeta a_0 \ell(S)}{C} \epsilon \eta^3 \leq \dcc(x,y) \leq C \zeta a_0 \ell(S) \eta$, we have
  \begin{align}
    \frac{\dcc(x,y)}{D} \leq d_X(h(x),h(y)). \label{distortion-assumption}
  \end{align}
  Then we must have for all $v \in S^{n-1}$ that
  \begin{align*}
    w(v) \geq \frac{1}{2D}.
  \end{align*}
  Indeed, otherwise we can choose any $x \in B$ and get from \eqref{distortion-assumption} that
  \begin{align*}
    d_X(h(xe^{-rv}),h(xe^{rv})) \geq \frac{\dcc(xe^{-rv},xe^{rv})}{D} \geq \frac{2r}{D},
  \end{align*}
  and so
  \begin{align*}
    d_X(h(xe^{-rv}),h(xe^{rv})) - 2r w(v) \geq \frac{r}{D},
  \end{align*}
  which contradicts \eqref{CBB-coarse-diff} as $\epsilon \leq D^{-2}$.
  
  As $(\epsilon/6C_1)^2 \leq 1/2D$, applying Lemma \ref{CBB-close-geodesic} to \eqref{CBB-coarse-diff} with $L = 1$ and $L' = w(v)$ shows that, for each $x \in B$, $v \in S^{n-1}$, there exists a constant speed geodesic $\gamma : [-r,r] \to X$ so that
  \begin{align}
    d_X(h(xe^{tv}),\gamma(t)) \leq C_1 \left( \frac{\epsilon}{6^{3/2} C_1} \right) 6 r \leq \frac{1}{2} \epsilon r, \qquad \forall t \in [-r,r]. \label{CBB-almost-geodesic}
  \end{align}
  By Lemma \ref{sublinear-diverge}, there exist $u,v \in S^{n-1}$ with the following property.  Setting $g = z_Q e^{\epsilon r u}$, we have for $|t| > \epsilon r$ 
  \begin{align*}
    \frac{(\epsilon r)^{1-\beta_1}}{C} |t|^{\beta_1} \leq \dcc(z_Qe^{tv},ge^{tv}) \leq C (\epsilon r)^{1-\beta_2} |t|^{\beta_2}.
  \end{align*}
  For $|t| \in [\epsilon r, \epsilon^{1/2} r]$, we then have that
  \begin{align*}
    \dcc(z_Q e^{tv},ge^{tv}) \in \left[ \frac{1}{C} \epsilon r, C \epsilon^{1 - \frac{\beta_2}{2}} r \right] \subseteq \left[ \frac{1}{C} \zeta a_0 \ell(S) \eta^4, C \zeta a_0 \ell(S) \eta\right].
  \end{align*}
  Then by \eqref{distortion-assumption} and the fact that $h$ is 1-Lipschitz, for $t \in [\epsilon r, \epsilon^{1/2} r]$, we get that
  \begin{align}
    L(t) := \frac{(\epsilon r)^{1-\beta_1}}{CD} t^{\beta_1} \leq d_X(h(ge^{tv}),h(z_Qe^{tv})) \leq C (\epsilon r)^{1-\beta_2} t^{\beta_2} =: U(t). \label{CBB-sublinear-bound}
  \end{align}
  Let $\gamma_0$ and $\gamma_1$ denote the constant speed geodesics of $X$ associated with $h(z_Qe^{tv})$ and $h(ge^{tv})$ on the domain $[-r,r]$ using Lemma \ref{CBB-close-geodesic}, respectively.  Note that
  \begin{align*}
    d_X(\gamma_0(0),\gamma_1(0)) \leq d_X(\gamma_0(0),h(z_Q)) + d_X(h(z_Q),h(g)) + d_X(h(g),\gamma_1(0)) \overset{\eqref{CBB-almost-geodesic}}{\leq} 2\epsilon r.
  \end{align*}
  Thus, Lemma \ref{CBB-linear-diverge} gives us that there exist some affine function $A : \R \to \R$ so that
  \begin{align*}
    | d_X(\gamma_0(t),\gamma_1(t)) - A(t) | \leq 60 \epsilon r, \qquad \forall t \in [0,\epsilon^{1/2}r],
  \end{align*}
  which, by \eqref{CBB-almost-geodesic}, further gives
  \begin{align}
    | d_X(h(z_Qe^{tv}),h(ge^{tv})) - A(t) | \leq 61 \epsilon r, \qquad \forall t \in [0,\epsilon^{1/2}r]. \label{CBB-linear-bound}
  \end{align}
  Let $f(t) = d_X(h(z_Qe^{tv}),h(ge^{tv}))$.  As
  \begin{align*}
    U(\epsilon r) &= C\epsilon r, \\
    U(\epsilon^{1/2} r) &= C\epsilon^{1-\frac{\beta_2}{2}} r,
  \end{align*}
  by \eqref{CBB-sublinear-bound}, we get from \eqref{CBB-linear-bound} that
  \begin{align*}
    A(\epsilon r) &\leq U(\epsilon r) + 61\epsilon r = (61 + C)\epsilon r, \\
    A(\epsilon^{1/2} r) &\leq U(\epsilon^{1/2} r) + 61\epsilon r = 61\epsilon r + C\epsilon^{1-\frac{\beta_2}{2}} r.
  \end{align*}
  Thus, as $A$ is affine, we get (using just a slightly worse bound) that
  \begin{align*}
    A(t) \leq (61 + C)\epsilon r + C\epsilon^{\frac{1-\beta_2}{2}} t, \qquad \forall t \in [\epsilon r,\epsilon^{1/2} r].
  \end{align*}
  This then, by \eqref{CBB-linear-bound} gives that
  \begin{align*}
    f(t) \leq (122 + C)\epsilon r + C\epsilon^{\frac{1-\beta_2}{2}} t, \qquad \forall t \in [\epsilon r,\epsilon^{1/2} r].
  \end{align*}
  Let $s = [C(124 + C)D]^{1/\beta_1} \epsilon r$.  One can see that \eqref{CBB-eps-defn} gives that $s \in [\epsilon r,\epsilon^{1/2}r]$.  Then we have
  \begin{align*}
    f(s) &\leq (122 + C) \epsilon r + [CD(124 + C)]^{1/\beta_1} \epsilon^{\frac{3-\beta_2}{2}} r \overset{\eqref{CBB-eps-defn}}{\leq} (123+C)\epsilon r, \\
    L(s) &= (124 + C) \epsilon r.
  \end{align*}
  This is a contradiction of the fact that $L(t) \leq f(t)$ on $[\epsilon r,\epsilon^{1/2}r]$.  Thus, \eqref{distortion-assumption} is false and so there exists some $x,y \in G$ with $\frac{1}{C} \zeta a_0 \ell(S) \eta^4 \leq \dcc(x,y) \leq C \zeta a_0 \ell(S) \eta$ so that
  \begin{align*}
    d_X(f(x),f(y)) \leq \frac{\dcc(x,y)}{D}.
  \end{align*}
  Recalling the definition of $D$ and $\eta$, this gives us that
  \begin{align*}
    \frac{d_X(f(x),f(y))}{\dcc(x,y)} &\overset{\eqref{CBB-eps-defn}}{\leq} C(124 + C)\epsilon^{\frac{\beta_1(1-\beta_2)}{2}} \\
    &\overset{\eqref{CBB-eta-defn}}{\leq} C(124+C) \left[ \left( 4C_2 \log \frac{1}{\tau} \right)^{\frac{1}{2\alpha}} \left( \log \frac{1}{\dcc(x,y)} - \log \frac{C_0 C}{\zeta a_0} \right)^{-\frac{1}{2\alpha}} \right]^{\frac{\beta_1(1-\beta_2)}{2}}.
  \end{align*}
  As $\epsilon$ can be made arbitrarily small by making $D$ arbitrarily big, we get by the fact that $\dcc(x,y) \leq \eta$ and \eqref{CBB-eta-defn} that $\dcc(x,y)$ can be made arbitrarily small.
\end{proof}

\subsection{Nonembeddability into $CAT(0)$ spaces}

\begin{theorem} \label{CAT-nonembed}
  Let $(X,d_X)$ be a $CAT(0)$ space.  Then there exist $c,C > 0$ such that for every $f : B_G \to X$ which is 1-Lipschitz with respect to the Carnot-Carath\'{e}odory metric there exist $x,y \in B_G$ where $\dcc(x,y)$ is arbitrarily small so that
  \begin{align*}
    \frac{d_X(f(x),f(y))}{\dcc(x,y)} \leq C\left( \log \frac{1}{\dcc(x,y)} \right)^{-c}.
  \end{align*}
\end{theorem}

For $CAT(0)$ spaces, we can show that almost minimizing curves are close to a minimizing geodesic.  We first need the following lemma, which can be found as Proposition 5.1 of \cite{Ballman}.

\begin{lemma} \label{CAT-midpoint-ineq}
  Let $x,y \in X$, a $CAT(0)$ space, and suppose $m$ is the midpoint.  Then
  \begin{align*}
    d_X(z,m)^2 \leq \frac{d_X(x,z)^2 + d_X(z,y)^2}{2} + \frac{d_X(x,y)^2}{4}.
  \end{align*}
\end{lemma}

\begin{lemma} \label{CAT-close-geodesic}
  Let $\epsilon \in (0,1)$ and $(X,d_X)$ be a $CAT(0)$ space.  Suppose $\gamma : [0,1] \to X$ is Lipschitz and there exists some $L > 0$ so that
  \begin{align*}
    \left| d_X(\gamma(s),\gamma(t)) - |t - s| L \right| \leq \epsilon L, \qquad \forall s,t \in [0,L].
  \end{align*}
  Let $\gamma_0 : [0,1] \to X$ be the constant speed minimal geodesic from $\gamma(0)$ to $\gamma(1)$.  There exists some universal constant $C > 0$ so that
  \begin{align*}
    \sup_{t \in [0,1]} d_X(\gamma(t),\gamma_0(t)) \leq C\sqrt{\epsilon} L.
  \end{align*}
\end{lemma}

\begin{proof}
  For each $k \in \N$, let $\gamma_k$ be the broken geodesic going through $\gamma(j2^{-k})$ for $j \in \{0,...,2^k\}$.  Given any $k$, we have that
  \begin{align*}
    \gamma_{k+1}(j2^{-k}) = \gamma_k(j2^{-k}) = \gamma(j2^{-k}), \qquad \forall j \in \{0,...,2^k\}.
  \end{align*}
  This gives that
  \begin{align*}
    d_X(\gamma_{k+1}(2j2^{-k-1}),\gamma_{k+1}((2j+1)2^{-k-1})) &\leq 2^{-k-1} L + \epsilon L, \\ 
    d_X(\gamma_{k+1}(j2^{-k}),\gamma_{k+1}(j2^{-k})) &\geq 2^{-k} L - \epsilon L.
  \end{align*}
  By setting $x = \gamma_k(j2^{-k})$, $y = \gamma_k((j+1)2^{-k})$, $m = \gamma_k((2j+1)2^{-k-1})$, and $z = \gamma_{k+1}((2j+1)2^{-k-1})$, Lemma \ref{CAT-midpoint-ineq} gives for all $j \in \{0,...,2^k-1\}$ that
  \begin{align*}
    d_X(\gamma_k((2j+1)2^{-k-1}), &\gamma_{k+1}((2j+1)2^{-k-1}))^2 \\
    &\leq \frac{\left(2^{-k-1}L + \epsilon L\right)^2 + \left(2^{-k-1}L + \epsilon L\right)^2}{2} - \frac{\left( 2^{-k}L - \epsilon L\right)^2}{4} \\
    &\leq \frac{\epsilon}{2^{k-1}} + \frac{3}{4} \epsilon^2 L^2.
  \end{align*}
  By the $CAT(0)$ triangle comparison property, we then have that
  \begin{align}
    \sup_{t \in [0,1]} d_X(\gamma_k(t), \gamma_{k+1}(t)) \leq \sqrt{\frac{\epsilon}{2^{k-1}}} + \sqrt{\frac{3}{4}} \epsilon L. \label{gamma-k}
  \end{align}
  Set $k = \log \frac{1}{\sqrt{\epsilon}}$.  Then given any $t \in [0,1]$, there exists some $j \in \{0,...,2^k\}$ so that $j2^{-k} - t \leq 2^{-k} \leq \sqrt{\epsilon}$.  Thus, we have
  \begin{align}
    d_X(\gamma(t),\gamma_k(t)) &\leq d_X(\gamma(t),\gamma(j2^{-k})) + d_X(\gamma_k(j2^{-k}),\gamma_k(t)) \leq 2\sqrt{\epsilon} L + 2\epsilon L. \label{gamma-0}
  \end{align}
  We can now bound
  \begin{align*}
    d_X(\gamma(t),\gamma_0(t)) &\leq d_X(\gamma(t),\gamma_k(t)) + \sum_{j=0}^{k-1} d_X(\gamma_j(t),\gamma_{j+1}(t)) \\
    &\overset{\eqref{gamma-k} \wedge \eqref{gamma-0}}{\leq} 2\sqrt{\epsilon} L + 2\epsilon L + k \sqrt{\frac{3}{4}} \epsilon L + \sqrt{2\epsilon} L \sum_{j=0}^{k-1} 2^{\frac{-k}{2}}.
  \end{align*}
  As $k = \log \frac{1}{\sqrt{\epsilon}} \leq \frac{1}{\sqrt{\epsilon}}$, we get that there exists some universal constant $C > 0$ so that
  \begin{align*}
    d_X(\gamma(t),\gamma_0(t)) \leq C\sqrt{\epsilon} L.
  \end{align*}
\end{proof}

Our final ingredient is the following lemma, which can be found as Proposition 5.4 of \cite{Ballman}.

\begin{lemma}
  Let $I$ be an interval and let $\gamma_1,\gamma_2 : I \to X$ be two geodesics in a $CAT(0)$ space $X$.  Then $d(\gamma_1(t),\gamma_2(t))$ is convex in $t$.
\end{lemma}

\begin{proof}[Proof of Theorem \ref{CAT-nonembed}]
  The proof is the same as the proof of Theorem \ref{CBB-nonembed} as the fact that geodesics diverge convexly can only help.
\end{proof}

\subsection{Nonembeddability for finitely generated torsion-free nilpotent groups}
In this subsection, we will prove Theorem \ref{fg-nonembed}.  We will first need the following lemma relating Gromov-Hausdorff distances to quasi-isometries.

\begin{lemma} \label{GH-quasi}
  Suppose $d_{GH}((X,d_X),(Y,d_Y)) \leq \epsilon$.  Then there exists a quasi-isometry $f : X \to Y$ so that
  \begin{align*}
    d_X(x,y) - 6\epsilon \leq d_Y(f(x),f(y)) \leq d_X(x,y) + 6\epsilon.
  \end{align*}
\end{lemma}

\begin{proof}
  By the definition of quasi-isometry, there exists a metric space $(Z,d_Z)$ and isometric embeddings $i_X : X \to Z$ and $i_Y : Y \to Z$ so that the Hausdorff distance of $i_X(X)$ and $i_Y(Y)$ in $Z$ is at most $2\epsilon$.  Thus, for $x \in X$, we let $f(x) \in Y$ be chosen so that $d_Z(i(x),i(f(x))) \leq 3\epsilon$.  The upper and lower bounds are now easily verified by the triangle inequality of $d_Z$.
\end{proof}

In this section, $G$ will be an infinite torsion-free group of nilpotency step $r$ that is generated by the finite symmetric set $S \subset G$.  We will then let $d_S$ be the word metric assocated to $S$ on $G$.  We have the following theorem of \cite{BLD}.

\begin{theorem} \label{GH-convergence}
  There exists positive constants $C,\gamma > 0$ and a Carnot group $\Gamma$, both depending only on $G$ and $S$ such that, as $n \to \infty$,
  \begin{align*}
    d_{GH}\left( (B_S(n),d_S), \left(B_\Gamma(n),d_\Gamma\right)\right) \leq Cn^{1-\gamma}.
  \end{align*}
  Here, $d_\Gamma$ is a subFinsler metric on $\Gamma$, $B_\Gamma(n)$ is the ball of radius $n$ around the identify of $\Gamma$, and $B_S(n)$ is the ball of radius $n$ around the identity of $G$.  Additionally, if $G$ is nonabelian, then so is $\Gamma$.
\end{theorem}

Although the statement shows convergence to a subFinsler metric instead of a Carnot-Carath\'{e}odory metric ({\it i.e.} the horizontal subbundle is equipped with a Finsler norm instead of a scalar product), all the previous differentiability results still apply as the two metrics are biLipschitz equivalent.  This is because both metrics are homogeneous metrics.

We will take advantage of the fact that coarse differentiation holds for maps that are Lipschitz in the large.  Specifically, we will precompose a given Lipschitz map $f : G \to X$ with a quasi-isometric embedding $g : \Gamma \to G$ to get a map $F = f \circ g : \Gamma \to X$ that is $\psi$-LLD.  Using the same reasoning as we did in the previous subsection about quantitative nonembeddability, we will show that $F$ must collapse points.  This will prove that $f$ must also collapse points as $g$ is quasi-isometric.

We will prove Theorem \ref{fg-nonembed} only for embeddings into uniformly convex spaces.  It will be obvious from reading the proof, how to modify it using the proofs of Theorems \ref{CBB-nonembed} and \ref{CAT-nonembed}, for $CAT(0)$ and $CBB(0)$ targets.

\begin{proof}[Proof of Theorem \ref{fg-nonembed} for uniformly convex targets]
  Let $n \in \N$ be prescribed.  Theorem \ref{GH-convergence}, in conjunction with Lemma \ref{GH-quasi}, gives us that there exists a nonabelian Carnot group $\Gamma$ and a quasi-isometry $g : B_\Gamma(n) \to B_S(n)$ so that
  \begin{align}
    d_\Gamma(x,y) - C_0n^{1-\gamma} \leq d_S(g(x),g(y)) \leq d_\Gamma(x,y) + C_0n^{1-\gamma}, \qquad \forall x,y \in B_\Gamma(n). \label{quasi-bounds}
  \end{align}
  Here, $\gamma$ and $C_0$ are constants that depend only on $G$.  Let $f : B_S(n) \to (X,\|\cdot\|)$ be a Lipschitz embedding.  By rescaling the image, we may suppose that $f$ is 1-Lipschitz.  Then it can be easily verified from \eqref{quasi-bounds} that $F = f \circ g : B_\Gamma(n) \to X$ is $C_0n^{1-\gamma}$-LLD with $\Lip_F(C_0n^{1-\gamma}) \leq 2$.  We will suppose that there exists some $D > 1$ so that for all $x,y \in B_S(n)$, we have
  \begin{align}
    \|f(x) - f(y)\| \geq \frac{1}{D} d_S(x,y). \label{discrete-distortion}
  \end{align}
  Set $\epsilon := \frac{1}{4D}$.  Suppose that
  \begin{align}
    n^\gamma \geq \tau^{-C_1 D^\alpha} \label{large-n-assumption}
  \end{align}
  for some sufficiently large $C_1 > 0$ that we will fix later.  Here, $\alpha$ is the same constant as that in Theorem \ref{UC-UAAP}.
  
  By translation and scaling of $\Gamma$, we may suppose by Theorem \ref{christ-cubes} that we have a $T \in \Delta$ so that
  \begin{align*}
    B_\Gamma\left( \frac{n}{C_2} \right) \subseteq T \subseteq B_\Gamma(n)
  \end{align*}
  where $C_2 \geq 1$ is some constant dependent only on $\Gamma$.  Thus, $\ell(T)$ is comparable to $n$, and so Theorem \ref{UC-UAAP} says that there exists some $\beta > 0$, $C_3 > 0$, and $\zeta > 0$ so that if
  \begin{align}
    n \geq C_3 \epsilon^{-\beta} \tau^{-m} n^{1-\gamma}, \label{UAAP-condition}
  \end{align}
  then
  \begin{align*}
    \sum_{k=0}^m \sum_{Q \in \Delta_k(S)} \left\{ |Q| : \qd_h^{UC}(Q,\zeta \epsilon^r) > \epsilon \Lip_F(C_0n^{1-\gamma}) \right\} \leq \epsilon^{-\alpha} |T|.
  \end{align*}
  Here, $r$ is the nilpotency degree of $\Gamma$.  We can choose $C_1$ to be large enough so that \eqref{UAAP-condition} is satisfied for $m = \lceil \epsilon^{-\alpha} \rceil + 1$.  Then we see as we did in the proof of Theorem \ref{UC-nonembed} that there must exist some $Q \in \bigcup_{k=0}^m \Delta_k(T)$ so that
  \begin{align*}
    \sup_{z \in \zeta \epsilon^r B_Q} \frac{\|F(z) - T(z) - v\|}{\zeta \epsilon^r a_0 \ell(Q)} \leq \epsilon \Lip_F(C_0 n^{1-\gamma}) \leq 2 \epsilon.
  \end{align*}
  As $G$ is nonabelian, so is $\Gamma$.  Thus, as in the proof of Theorem \ref{UC-nonembed}, we get that there exists two point $x,y \in \zeta \epsilon^r B_Q$ so that $d_\Gamma(x,y) = \zeta\epsilon^r a_0\ell(Q)$ and
  \begin{align}
    \|F(x)-F(y)\| \leq 2\epsilon d_\Gamma(x,y). \label{composition-upper-bound}
  \end{align}
  On the otherhand, as $F = f \circ g$, we get that
  \begin{align}
    \|f \circ g(x) - f \circ g(y)\| \overset{\eqref{discrete-distortion}}{\geq} 4\epsilon d_S(g(x),g(y)) \overset{\eqref{quasi-bounds}}{\geq} 4\epsilon d_\Gamma(x,y) - 4C_0\epsilon n^{1-\gamma}. \label{composition-lower-bound}
  \end{align}
  In the first inequality, we used the distortion bound for $f$ and the definition of $\epsilon$.  Now assume we've taken $C_1$ to be large enough again so that
  \begin{align}
    4C_0 n^{1-\gamma} \leq \zeta \epsilon^r a_0 \tau^m \ell(T) \leq d_\Gamma(x,y). \label{quasi-distort-small}
  \end{align}
  Then we get that
  \begin{align*}
    \|F(x) - F(y)\| \overset{\eqref{composition-lower-bound} \wedge \eqref{quasi-distort-small}}{\geq} 3\epsilon d_\Gamma(x,y).
  \end{align*}
  This together with \eqref{composition-upper-bound} gives a contradiction.  Thus, \eqref{large-n-assumption} is false for sufficiently large $C_1$:
  \begin{align*}
    n^\gamma < \tau^{-C_1 D^\alpha}.
  \end{align*}
  This then gives that
  \begin{align*}
    D \geq \left( \frac{\gamma}{C_1 \log 1/\tau} \log n\right)^{1/\alpha}.
  \end{align*}
\end{proof}

\subsection{Discretization}
In this section, we prove an analogue of Bourgain's discretization theorem in the setting of Carnot groups.  One of the steps in the proof of discretization in \cite{LN} was the invocation of the Lipschitz extension theorem of \cite{JLS}.  While there are Lipschitz extension theorems between Carnot groups for restricted cases \cite{WY}, there is no general theorem.  Indeed, it was shown in \cite{BF,RW} that there exists a Lipschitz map from $S^3 \subset \R^3$ to the three dimensional Heisenberg group that has no Lipschitz extension.  A Lipschitz extension theorem for maps from a discrete net into Carnot groups may still be possible, and is left as an interesting open problem.  Instead, we will take advantage of the fact that the coarse differentiation technique only requires maps that are Lipschitz at large distances.  Thus, we will use a noncontinuous piecewise constant extension based on the Voronoi cell decomposition of metric spaces to get our needed extension.

Recall that, given a discrete set of a metric space $Z \subseteq (X,d_X)$, the Voronoi cell decomposition of $X$ is simply the partition $\{P_z\}_{z \in Z}$ where $x \in P_y$ if $y \in Z$ is the closest point of $Z$ to $x$, that is $\inf_{z \in Z} d_X(z,x) = d_X(y,x)$.  Ties are arbitrarily broken.  In this section, $\alpha,\beta,\zeta > 0$ will be the constants from Theorem \ref{C-UAAP}.

\begin{theorem}
  Given two Carnot groups $G,H$, there exists a $c > 0$ so that for every $\epsilon \in (0,1/2)$, we have
  \begin{align}
    \delta_{G \hookrightarrow H}(\epsilon) \geq \exp \left[ -e^{(c_H(G)/\epsilon)^c} \right]. \label{discretization-bound}
  \end{align}
\end{theorem}

\begin{proof}
  If $c_H(G) = \infty$, then the statement holds vacuously, so we assume this is not the case.  By reading the proof of Proposition \ref{carleson-cubes} and Theorem \ref{C-UAAP}, we see that, to use Theorem \ref{C-UAAP} on a dyadic cube $S \in \Delta$, it suffices to have the function $h$ defined only on $6S$.

  We set
  \begin{align}
    m &= \left\lceil \exp \left( \left( \frac{256c_H(G)}{\epsilon} \right)^\alpha \right) \right\rceil, \notag \\
    \delta &= \tau^{C_0m}. \label{discretization-delta}
  \end{align}
  Here, $C_0 > 0$ will be a sufficiently large constant to be chosen.  Let $\Ndelta$ be a $\delta$-net of the unit ball of $G$ and write $D = c_H(\Ndelta) \leq c_H(G)$.  Take $f : \Ndelta \to H$ satisfying
  \begin{align}
    d_G(x,y) \leq d_H(f(x),f(y)) \leq \left(1+ \frac{\epsilon}{16} \right) D d_G(x,y), \qquad \forall x,y \in \Ndelta. \label{net-bounds}
  \end{align}
  We then define the function $F : B_G \to H$ where $F(x) = f(x)$ for $x \in \Ndelta$ and $F$ is constant on the Voronoi cells of $B_G$ as determined by $\Ndelta$.  Notice that $F$ is $2\delta$-LLD with $\Lip_F(2\delta) \leq 2(1+\epsilon/16) D \leq 4D$.  Indeed, pick $x,y \in B_G$ so that $d_G(x,y) \geq 2\delta$.  Then there exist two elements $u,v \in \Ndelta$ so that $x \in P_u$ and $y \in P_v$.  Note then that $d_G(x,u) \leq \delta$ and $d_G(y,v) \leq \delta$ and so
  \begin{align*}
    \Lip_F(2\delta) &\leq \frac{d_H(f(x),f(y))}{d_G(x,y)} = \frac{d_H(f(u),f(v))}{d_G(x,y)} \leq \left( 1+ \frac{\epsilon}{16} \right) D \frac{d_G(u,v)}{d_G(x,y)} \leq 2\left( 1 + \frac{\epsilon}{16} \right) D.
  \end{align*}

  Choose a maximal $S \in \Delta$ so that $6S \subseteq B_G$.  By the properties of the Christ cube, there exists some constant $C_1 > 0$ so that $\ell(S) \geq C_1$.  As $2\delta \leq m^{-1} \tau^m$ for sufficiently large $C_0$, we can, after translating the image of $F$, use Theorem \ref{C-UAAP} to get that there exists some subball $g \cdot RB_G\subseteq Q \in \bigcup_{k=0}^m \Delta_k(S)$ of radius
  \begin{align}
    R \geq \zeta\left( \frac{\epsilon}{256c_H(G)} \right)^{\beta r} \tau^m a_0\ell(S), \label{discretization-radius}
  \end{align}
  and some homomorphism $T : G \to H$ so that
  \begin{align}
    \sup_{x \in B} d_H(F(x),T(x)) \leq \frac{\epsilon R}{32}. \label{discretization-UAAP}
  \end{align}
  Choosing $C_0$ sufficiently large again, we can obtain the bound
  \begin{align}
    R \geq \frac{64 c_H(G) \delta}{\epsilon}. \label{R-delta}
  \end{align}
  The proof can now follow exactly the same proof of discretization in \cite{LN}.  We will reproduce it here, nearly word for word, for convenience.
  
  Choose $x \in G$ with $d_G(0,x) = 1$ and $u,v \in \Ndelta \cap (g \cdot RB_G)$ so that $d_G(g,u) \leq \delta$ and $d_G(v,g\delta_{R/2}(x)) \leq \delta$.  Note then that
  \begin{align*}
    d_G(u,v) \leq d_G(u,g) + d_G(g,g\delta_{R/2}(x)) + d_G(g\delta_{R/2}(x),v) \leq \frac{R}{2} + 2\delta.
  \end{align*}
  Similarly $d_G(u,v) \geq \frac{R}{2} - 2\delta$.  Using the fact that $F$ extends $f$, we get
  \begin{multline*}
    d_H(T(u),T(v)) \overset{\eqref{discretization-UAAP}}{\leq} \frac{\epsilon}{16} R + d_H(f(u),f(v)) \overset{\eqref{net-bounds}}{\leq} \frac{\epsilon}{16} R + \left(1+\frac{\epsilon}{16}\right)D d_G(u,v) \\
    \leq \frac{\epsilon}{16} R + \left(1 + \frac{\epsilon}{16} \right) D \left(\frac{R}{2} + 2\delta \right) \overset{\eqref{R-delta}}{\leq} \left( 1 + \frac{\epsilon}{4} \right) \frac{R}{2} D.
  \end{multline*}
  Hence,
  \begin{align*}
    d_H(T(x),T(0)) &\leq \frac{2}{R} \left( d_H(T(g\delta_{R/2}(x)),T(v)) + d_H(T(v),T(u)) + d_H(T(u),T(g)) \right) \\
    &\leq \left(1 + \frac{\epsilon}{4} \right) D + \frac{4\delta \|T\|_{lip}}{R}.
  \end{align*}
  As this holds for each $x$ with unit norm, we get that
  \begin{align}
    \|T\|_{lip} \leq \frac{1+\epsilon/4}{1-4\delta/R} D \leq \left(1 + \frac{\epsilon}{2} \right)D \leq 2 c_H(G). \label{T-upper-bound}
  \end{align}
  Now,
  \begin{align*}
    d_H(T(u),T(v)) &\overset{\eqref{discretization-UAAP}}{\geq} d_H(f(u),f(v)) - \frac{\epsilon R}{16} \geq d_G(u,v) - \frac{\epsilon R}{16} \\
    &\geq \frac{R}{2} - 2\delta - \frac{\epsilon R}{16} \overset{\eqref{R-delta}}{\geq} \left( 1-  \frac{\epsilon}{4} \right) \frac{R}{2}.
  \end{align*}
  Hence,
  \begin{multline*}
    d_H(T(x),0) \geq \frac{2}{R} \left( d_H(T(v),T(u)) - d_H(T(g\delta_{R/2}(x)),T(v)) - d_H(T(u),T(g))\right) \\
    \geq 1 - \frac{\epsilon}{4} - \frac{4\delta \|T\|_{lip}}{R} \overset{\eqref{T-upper-bound}}{\geq} 1 - \frac{\epsilon}{4} - \frac{8c_H(G)\delta}{R} \overset{\eqref{R-delta}}{\geq} 1 - \frac{\epsilon}{2}.
  \end{multline*}
  Thus, we have proven that $c_H(G) \leq \frac{1-\epsilon/2}{1-\epsilon/2} D = \frac{1+\epsilon/2}{1-\epsilon/2} c_H(\Ndelta) \leq \frac{1}{1-\epsilon} c_H(\Ndelta)$.  Thus, recalling the choice of $\delta$ in \eqref{discretization-delta}, we get
  \begin{align*}
    \delta_{G \hookrightarrow H}(\epsilon) \geq \tau^{C_0m}.
  \end{align*}
  Notice that, for sufficiently large $c$, this $\delta$ satisfies the lower bound of \eqref{discretization-bound}.
\end{proof}

\bibliographystyle{abbrv}
\bibliography{carnot}

\end{document}